\documentclass[11pt]{article}
\usepackage{mathrsfs,amsmath,amsfonts,amssymb,bm,bbm,eufrak,dsfont,pifont,amscd,
stmaryrd,euscript,amsthm,appendix,color,epsfig,xr,yfonts,tikz,verbatim}
\usepackage[round]{natbib}
\usepackage{graphicx,color,authblk}
\usepackage{amsfonts}
\usepackage{bbm}
\usepackage{amsmath}
\usepackage{mathtools}
\usepackage{multirow}
\usepackage{tikz}
\usepackage{hyperref}
\newtheorem{theorem}{Theorem}

\theoremstyle{remark}
\newtheorem{remark}{Remark}
\theoremstyle{example} 
\newtheorem{example}{Example}

\newtheorem{appxthm}{Theorem}[section]
\newtheorem{appxlem}[appxthm]{Lemma}
\newtheorem{appxpro}[appxthm]{Proposition}

\theoremstyle{definition}

\theoremstyle{remark}

\newcommand{\bb}{\mathbb}

\newcommand{\Cal}{\mathcal}

\newcommand{\norm}[1]{\left\lVert#1\right\rVert}

\newcommand*\circled[1]{\tikz[baseline=(char.base)]{
            \node[shape=circle,draw,inner sep=2pt] (char) {#1};}}

\title{Shrinkage Estimation of Higher Order Bochner Integrals}

\author{Saiteja Utpala}%\thanks{nzs5368@psu.edu}}
\author{Bharath K. Sriperumbudur$^*$}%\thanks{bks18@psu.edu}} 
\affil{$^*$Department of Statistics,  
Pennsylvania State University\\
University Park, PA 16802, USA.\\
\texttt{saitejautpala@gmail.com; bks18@psu.edu}}

\newcommand\Myperm[2][^n]{\prescript{#1\mkern-3.5mu}{}{\,\textup{P}}_{#2}}
\newcommand\Mycomb[2][^n]{\prescript{#1\mkern-3.5mu}{}\,\textup{C}_{#2}}
\DeclarePairedDelimiterX{\inp}[2]{\langle}{\rangle}{#1, #2}
\DeclarePairedDelimiter\abs{\lvert}{\rvert}%

\makeatletter
\newcommand{\vast}{\bBigg@{3.5}}
\newcommand{\Vast}{\bBigg@{4}}
\makeatother
\newcommand\numberthis{\addtocounter{equation}{1}\tag{\theequation}}

\DeclarePairedDelimiter\floor{\lfloor}{\rfloor}

\usepackage{amsmath}

\DeclareMathOperator*{\argmin}{arg\,min}
\date{}

\usepackage{scalerel,stackengine}
\stackMath
\newcommand\reallywidehat[1]{%
\savestack{\tmpbox}{\stretchto{%
  \scaleto{%
    \scalerel*[\widthof{\ensuremath{#1}}]{\kern-.6pt\bigwedge\kern-.6pt}%
    {\rule[-\textheight/2]{1ex}{\textheight}}%WIDTH-LIMITED BIG WEDGE
  }{\textheight}% 
}{0.5ex}}%
\stackon[1pt]{#1}{\tmpbox}%
}
\begin{document}
\maketitle

\begin{abstract}%   <- trailing '%' for backward compatibility of .sty file
 We consider shrinkage estimation of higher order Hilbert space-valued Bochner integrals in a  non-parametric setting. We propose estimators that shrink the $U$-statistic estimator of the Bochner integral towards a pre-specified target element in the Hilbert space. Depending on the degeneracy of the kernel of the $U$-statistic, we construct consistent shrinkage estimators 
%with fast rates of convergence, 
and develop oracle inequalities comparing the risks of the $U$-statistic estimator and its shrinkage version. Surprisingly, we show that the shrinkage estimator designed by assuming complete degeneracy of the kernel of the $U$-statistic is a consistent estimator even when the kernel is not complete degenerate. This work subsumes and improves upon \citet{muandet2016kernel} and \citet{zhou2019class}, which only handle mean element and covariance operator estimation in a reproducing kernel Hilbert space. We also specialize our results to normal mean estimation and show that for $d\ge 3$, the proposed estimator strictly improves upon the sample mean in terms of the mean squared error.
 %  all of existing cases in terms of oracle bounds. We design another shrinkage estimator  that obtains fast rates under assumption of complete degeneracy. As far as we know, ours is first work that obtains fast rates for shirnkage estimators under such assumptions. Surprisingly we show that estimator designed by assuming complete degeneracy is good estimator even when $r$ is not complete-degenerate and error term only gets increased by $\sqrt{n}.$ We also analyze estimator for mean of gaussian 
\end{abstract}
% \begin{keywords}
%   Bochner integral, empirical estimator, kernel mean embeddings, minimax lower bounds, nonparametric function estimation, reproducing kernel Hilbert space
% \end{keywords}
\setlength{\parskip}{4pt}
% 
% \setlength{\parskip}{4pt}

% \input{introduction}
% \input{notation}
% %\input{kme-estimation}
% \input{pca}
% \input{results}
% %\input{rff-pca}
% \input{discussion}
% \input{proofs}
% \input{acknowledgements}
% \bibliographystyle{plainnat}
% \bibliography{Reference}
% \input{appendix}
%\newpage
% \bibliographystyle{plainnat}
% \bibliography{Reference}
%\newpage

\section{Introduction }  
Let $\mathcal{X}$ be a separable topological space and $\mathcal{H}$ be a separable Hilbert space. For a Bochner measurable function---for example, continuous functions are Bochner measurable---$r:\mathcal{X}^k\rightarrow\mathcal{H}$, where $k\in\mathbb{N}$, define the Bochner integral \citep{dinculeanu2000vector} with respect to the $k$-fold product measure $\mathbb{P}^k:=\mathbb{P}\times\stackrel{k}{\ldots}\times\mathbb{P}$ as
\begin{equation*}
C=\int_{\mathcal{X}^k} r(x_1,\ldots,x_k)\,d\mathbb{P}^k(x_1,\ldots,x_k)=\int_{\mathcal{X}^k} r(x_1,\ldots,x_k)\,\prod^k_{i=1}d\mathbb{P}(x_i).
\label{Eq:bochner}
\end{equation*}
Given $X_1,\ldots,X_n\stackrel{i.i.d.}{\sim}\mathbb{P}$, the goal of this paper to construct and analyze shrinkage estimators of $C$, of the form \begin{equation}\check{C}:=(1-\hat{\alpha})\hat{C}+\hat{\alpha} f^*=(1-\hat{\alpha})(\hat{C}-f^*)+f^*,
\label{Eq:shrinkage-estimator}
\end{equation}
where $0<\hat{\alpha}<1$ is a random variable that depends on $(X_i)^n_{i=1}$, $f^*$ is a fixed target in $\mathcal{H}$ towards which $\hat{C}$ is shrunk to, and $\hat{C}$ is the $U$-statistic estimator of $C$ given by
$$\hat{C}=\frac{1}{\Mycomb[n]{k}} \sum_{J^n_k} r(X_{i_1},\dots,X_{i_k}),$$
with $J_{k}^{n} = \{ (i_{1}, \dots, i_{k}) :   1<i_1 < i_2 <\dots <  i_{k}<n \}$. Without loss of generality, we assume that $r$ is \emph{symmetric} (see Section~\ref{Sec:notation} for the definition).

Traditionally, shrinkage estimators of the form in \eqref{Eq:shrinkage-estimator} are studied for $k=1$ and $r(x)=x$, $x\in\mathbb{R}^d$, which in fact corresponds to shrinking the empirical mean, $\bar{X}:=\frac{1}{n}\sum^n_{i=1}X_i$,  towards a fixed vector $f^*\in\mathbb{R}^d$. For $\mathbb{P}=N(\mu,\sigma^2 I)$ where $\sigma^2$ is known, \citet{Stein1956InadmissibilityOT,james1992estimation} constructed a shrinkage estimator, $\check{\mu}$ of $\mu$ of the form in \eqref{Eq:shrinkage-estimator}, given by
$$\check{\mu}=\left(1-\frac{(d-2)\sigma^2}{n\Vert \bar{X}-f^*\Vert^2_2}\right)\bar{X}+\frac{(d-2)\sigma^2}{n\Vert \bar{X}-f^*\Vert^2_2}f^*,$$ and showed that for $d\ge 3$, the shrinkage estimator, $\check{\mu}$ improves upon $\bar{X}$ in terms of the mean-squared error, i.e.,
\begin{equation}\mathbb{E}\Vert \check{\mu}-\mu\Vert^2_2<\mathbb{E}\Vert \bar{X}-\mu\Vert^2_2,\,\forall\,\mu\in\mathbb{R}^d.\label{Eq:better}\end{equation} 
%for all $\mu\in\mathbb{R}^d$. 
% In other words, when $k=1$, $r(x)=x,\,x\in\mathbb{R}^d$, they showed
% $\mathbb{E}\Vert \check{C}-C\Vert^2_\mathcal{H}<\mathbb{E}\Vert \hat{C}-C\Vert^2_\mathcal{H}$ for all $C$. 
When $\sigma^2$ is unknown, it can be replaced by its estimator $\hat{\sigma}^2=\frac{1}{n}\sum^n_{i=1}(X_i-\bar{X})^2$ in $\check{\mu}$ while still maintaining \eqref{Eq:better} for $d\ge 3$. Similar type of results have been established for location families of spherically symmetric distributions (see \citealp{brandwein1990stein, brandwein2012stein} and references therein).

For $k=2$ and $r(x_1,x_2)=\frac{1}{2}(x_1-x_2)(x_1-x_2)^\top,\,x_1,x_2\in\mathbb{R}^d$, \eqref{Eq:bochner} reduces to the covariance matrix associated with $\mathbb{P}$. Starting with \cite{Stein-cov-75}, a lot of work has been carried out on the shrinkage estimation of covariance matrices under the parametric setting of samples being observed from a multivariate normal distribution. Under different losses (e.g., Frobenius loss, Stein loss) and under different settings of $d\le n$, $d>n$, $d$ growing to infinity with $n$, the shrinkage estimator has been shown to strictly improve upon the sample covariance matrix (e.g., see \citealp{IEEESP:Chen-2010}, \citealp{fisher2011improved}, \citealp{Ledoit-Bernoulli-18} and references therein).  In the non-parametric setting where no specific parametric assumption is made on $\mathbb{P}$, consistent shrinkage estimators of the sample covariance matrix have been developed in the high-dimensional setting \citep{ledoit2004well,touloumis2015nonparametric}.
% when estimating mean of $d$-dimensional random vector $X$ having normal distribution with identity covariance matrix, maximum likelihood estimator is inadmissable when $d > 2$. 
    %  In 1956 \cite{Stein1956InadmissibilityOT} proved a remarkable result that when estimating under square error loss, the unknown mean vector $\theta$ of $d$-dimensional random vector $X$ having a normal distribution with identity covariance matrix, estimators of the form $(1-a/(\norm{X}^2+b))X$ dominate the usual estimator theta, for $a$ sufficiently small and $b$ sufficinetly large when $p \geq 3. $ James and Stein  \cite{james1992estimation} sharpened the result and gave an explicit class of domninating esitmators, $(1- a/\norm{X}^2) X$ for $0 < a < 2 (p-2)$ and also show that for $a = p-2$(the james stein estimator) is uniformly the best. 
    % For an  expository development of Stein estimation for nonnormal distribution \cite{brandwein1990stein, brandwein2012stein}. James stein positive estimator 
% \textcolor{red}{    \noindent \cite{fisher2011improved} propose parametric shrinkage covariance matrix estimation and   \cite{touloumis2015nonparametric, ledoit2004well} consider non parametric shrinkage covariance matrix estimation. \cite{yang2019shrinkage} shows how to shrinkage estimation on manifold SPD matrices. }
    
While most of the above mentioned works deal with parametric families of distributions, recently, \cite{muandet2016kernel} proposed shrinkage estimators for $C$ in the non-parametric setting without making parametric assumptions on $\mathbb{P}$, with $k=1$ and $r(x)=K(\cdot,x)$, where $K$ is the \emph{reproducing kernel} (i.e., a positive definite kernel) of a \emph{reproducing kernel Hilbert space} (RKHS)---see Section~\ref{Sec:notation} for the definition. This corresponds to the shrinkage estimation of the mean element, which is an infinite dimensional object if the RKHS is infinite dimensional. This is in sharp contrast to the above mentioned works where the parameter is finite dimensional or its dimension grows with the sample size. Extending this idea, \cite{zhou2019class} proposed shrinkage estimators for $C$ when $k=2$ and $r(x_1,x_2)=\frac{1}{2}\left(K(\cdot,x_1)-K(\cdot,X_2)\right)\otimes_\mathcal{H} \left(K(\cdot,x_1)-K(\cdot,x_2)\right)$, which corresponds to the covariance operator on an RKHS with reproducing kernel, $K$. The mean element and covariance operator has been widely used in nonparametric goodness-of-fit testing \citep{JMLR:v22:17-570}, two-sample testing \citep{JMLR:v13:gretton12a}, independence testing \citep{NIPS2007_d5cfead9}, supervised dimensionality reduction \citep{fukumizu2004dimensionality}, feature selection \citep{JMLR:v13:song12a}, etc., and therefore their shrinked versions are also useful in these applications. Of course, the choice of $K(\cdot,x)=x,\,x\in\mathbb{R}^d$, results in the mean and covariance matrix of $\mathbb{P}$ with $\mathcal{H}=\mathbb{R}^d$.

One of the key ideas in constructing a shrinkage estimator is based on minimizing an unbiased estimator of the risk, referred to as Stein Unbiased Shrinkage Estimation (SURE). Formally, suppose $\Delta=\mathbb{E}\Vert \hat{C}-C\Vert^2_\mathcal{H}$ is the mean squared error (i.e., risk) of the empirical estimator $\hat{C}$. Define $\Delta_\alpha=\mathbb{E}\Vert \hat{C}_\alpha-C\Vert^2_\mathcal{H}$, where $\hat{C}_\alpha\in \mathcal{C}=\{(1-\alpha)\hat{C}+\alpha f^*:\alpha\in\mathbb{R}\}$. Note that $(\Delta_\alpha)_\alpha$ corresponds to the family of risks associated with the estimators in $\mathcal{C}$. 
$\check{C}$ is constructed as $\hat{C}_{\hat{\alpha}}$, where $\hat{\alpha}=\arg\min_\alpha \hat{\Delta}_\alpha$, which means $\check{C}=(1-\hat{\alpha})\hat{C}+\hat{\alpha}f^*$. It can be shown that $\hat{\alpha}=\hat{\Delta}_u/\Vert \hat{C}-f^*\Vert^2_\mathcal{H}$, so that the shrinkage estimator of $C$ based on SURE is given by
$$\check{C}=\left(1-\frac{\hat{\Delta}_u}{\Vert \hat{C}-f^*\Vert^2_\mathcal{H}}\right)\hat{C}+\frac{\hat{\Delta}_u}{\Vert \hat{C}-f^*\Vert^2_\mathcal{H}}f^*,$$
where $\hat{\Delta}_u$ is an unbiased estimator of $\Delta$.

Another approach to find  $\hat{\alpha}$ is based on the observation that $\Delta_\alpha<\Delta$ if and only if     $\alpha \in \left( 0 , \frac{2 \Delta}{\Delta+ \norm{C-f^{*}}_{\mathcal{H}}^2} \right)$ with $|\Delta_\alpha-\Delta|$ maximized at \begin{equation}\alpha_{*} =  \frac{\Delta}{\Delta+ \norm{C-f^{*}}_{\mathcal{H}}^2},\label{Eq:alpha-star}\end{equation} which corresponds to the midpoint of the above interval. $\alpha_*$ can be estimated as \begin{equation}\tilde{\alpha}=\frac{\hat{\Delta}}{\hat{\Delta}+ \Vert \hat{C}-f^{*}\Vert_{\mathcal{H}}^2}\label{Eq:tilde-alpha}\end{equation}
so that
\begin{equation}\check{C}=\left(1-\frac{\hat{\Delta}}{\hat{\Delta}+\Vert \hat{C}-f^*\Vert^2_\mathcal{H}}\right)\hat{C}+\frac{\hat{\Delta}}{\hat{\Delta}+\Vert \hat{C}-f^*\Vert^2_\mathcal{H}}f^*,\label{Eq:estim-krik}\end{equation}
where $\hat{\Delta}$ is some estimator (not necessarily unbiased) of $\Delta$. This means, the SURE approach first estimates the risk and then minimizes it to find $\hat{\alpha}$ while the latter approach first finds the optimal $\alpha$ (in population) which is then estimated to find $\hat{\alpha}$. The difference in these approaches is an additional term of $\hat{\Delta}$ in the denominator of $\tilde{\alpha}$ compared to that of $\hat{\alpha}$ obtained from SURE. 

\cite{muandet2016kernel} and \cite{zhou2019class} considered the latter approach to construct a shrinkage estimator of $C$ and showed the oracle bound 
\begin{equation}\Delta_{\alpha_*}<\Delta_{\tilde{\alpha}}\le \Delta_{\alpha_*}+\mathcal{O}(n^{-3/2}),\,\,\,\text{as}\,\,\,n\rightarrow\infty,\label{Eq:oracle-1}\end{equation}
which holds for all $\bb{P}$ that satisfy certain moment conditions, and also showed $\check{C}$ to be a $\sqrt{n}$-consistent estimator of $C$. A motivation to consider this approach is as follows:  For $f^*=0$ and $r(x)=K(\cdot,x)$, we have $$\Vert \hat{C}\Vert^2_\mathcal{H}=\frac{1}{n^2}\sum^n_{i,j=1}\langle K(\cdot,X_i),K(\cdot,X_j)\rangle_\mathcal{H}=\frac{1}{n^2}\sum^n_{i,j=1}K(X_i,X_j)=\frac{1}{n^2}\mathbf{1}^\top\mathbf{K1},$$
where $\mathbf{1}=(1,\stackrel{n}{\ldots},1)^\top$ and $[\mathbf{K}]_{i,j}=K(X_i,X_j),\,i,j=1,\ldots,n$. If $K$ is not strictly positive definite, then there exists $(X_1,\ldots,X_n)$ such that $\mathbf{1}^\top\mathbf{K1}=0$, which means $\Vert \hat{C}\Vert^2_\mathcal{H}=0$ resulting in an invalid estimator. 
\subsection{Contributions} 
In this work, we generalize and improve the results of   \citep{muandet2016kernel} and \citep{zhou2019class} to any $k$ and any separable Hilbert space $\mathcal{H}$ (that is not necessarily an RKHS) without making any parametric assumptions on $\mathbb{P}$. Using the variance decomposition of the $U$-statistics, we construct an unbiased estimator, $\hat{\Delta}_{\textup{general}}$ of $\Delta$, which is used in \eqref{Eq:estim-krik} to construct the shrinkage estimator, $\check{C}=\hat{C}_{\tilde{\alpha}_{\textup{general}}}$, where $\tilde{\alpha}_{\textup{general}}$ is obtained by replacing $\hat{\Delta}$ by $\hat{\Delta}_{\textup{general}}$ in \eqref{Eq:tilde-alpha}. In Theorem~\ref{Th:NonDegC-NonDegDelta}, we show this estimator to be a $\sqrt{n}$-consistent estimator of $C$ and improve on the oracle bound in \eqref{Eq:oracle-1} by showing
\begin{equation}\Delta_{\alpha_*}<\Delta_{\tilde{\alpha}_{\textup{general}}}\le \Delta_{\alpha_*}+\mathcal{O}(n^{-2}),\,\,\,\text{as}\,\,\,n\rightarrow\infty.
\label{Eq:oracle-2}
\end{equation}
For $k\ge 2$, if $r-C$ is $\mathbb{P}$-\emph{complete degenerate} (see Section~\ref{Sec:notation} for the definition), again using the variance decomposition of degenerate $U$-statistics, we obtain an alternate estimator of $\Delta$, i.e., $\hat{\Delta}_{\textup{degen}}$, using which we show (see Theorem~\ref{Th:DegC-DegDelta}) the resulting estimator $\check{C}=\hat{C}_{\tilde{\alpha}_{\textup{degen}}}$ (obtained by using $\hat{\Delta}_{\textup{degen}}$ in \eqref{Eq:estim-krik}) to be $n^{k/2}$-consistent estimator of $C$ along with significantly faster error rates in the oracle bound:
\begin{equation}\Delta_{\alpha_*}<\Delta_{\tilde{\alpha}_{\textup{degen}}}\le \Delta_{\alpha_*}+\mathcal{O}(n^{-(3k+1)/2}),\,\,\,\text{as}\,\,\,n\rightarrow\infty,
\label{Eq:oracle-3}
\end{equation}
where $\tilde{\alpha}_{\textup{degen}}$ is obtained by replacing $\hat{\Delta}$ by $\hat{\Delta}_{\textup{degen}}$ in \eqref{Eq:tilde-alpha}. Note that in these results (Theorems~\ref{Th:NonDegC-NonDegDelta} and \ref{Th:DegC-DegDelta}), the estimator is constructed based on whether $r-C$ is $\mathbb{P}$-complete degenerate or not. In Theorem~\ref{Th:NonDegC-DegDelta}, we analyze the scenario of using $\hat{C}_{\tilde{\alpha}_{\textup{degen}}}$ as an estimator of $C$ irrespective of whether $r-C$ is $\mathbb{P}$-complete degenerate or not. We show that for $k\ge 2$, $\hat{C}_{\tilde{\alpha}_{\textup{degen}}}$ is a $\sqrt{n}$-consistent estimator of $C$ and satisfies the oracle bound:
\begin{equation*}\Delta_{\alpha_*}<\Delta_{\tilde{\alpha}_{\textup{general}}}\le \Delta_{\alpha_*}+\mathcal{O}_\mathbb{P}(n^{-3/2}),\,\,\,\text{as}\,\,\,n\rightarrow\infty,
%\label{Eq:oracle-2}
\end{equation*}
without assuming the complete degeneracy of $r-C$. This means, $\hat{C}_{\tilde{\alpha}_{\textup{degen}}}$ has a slightly weaker oracle bound than the one in \eqref{Eq:oracle-2} but the bound improves significantly to \eqref{Eq:oracle-3} if $r-C$ is $\mathbb{P}$-complete degenerate. All these results are based on Bernstein-type inequalities for unbounded, Hilbert space-valued random elements. For the degenerate case, we extended Bernstein's inequality of 
\cite{arcones1993limit,de2012decoupling} to unbounded Hilbert space-valued random elements (see Theorem~\ref{Th:BernUstatsDeg}), which is of independent interest.

Since all the above mentioned results are obtained in the non-parametric setting, we are not able to show exact improvement of the shrinkage estimator over $\hat{C}$ but only show oracle bounds that include an additional error term. In order to understand the behavior of the proposed estimator in the parametric setting, in Section~\ref{Sec:normal}, we specialize and analyze our estimator  $\hat{C}_{\tilde{\alpha}_{\textup{general}}}$ in the normal mean estimation problem. In other words, we use $k=1$, $r(x)=x,\,x\in\mathbb{R}^d$ and $\mathbb{P}=N(\mu,\sigma^2 I)$, where $\mu$ is the parameter of interest and $\sigma^2>0$ may not be known. In this setting with $f^*=0$, it is easy to verify that 
$$\hat{C}_{\tilde{\alpha}_{\textup{general}}}=\hat{C}_{\tilde{\alpha}_{\textup{degen}}}=\frac{\Vert \bar{X}\Vert^2_2}{\frac{S^2}{n}+\Vert \bar{X}\Vert^2_2}\bar{X}=\left(1-\frac{\frac{S^2}{n}}{\frac{S^2}{n}+\Vert \bar{X}\Vert^2_2}\right)\bar{X},$$
where $S^2:=\frac{1}{n-1}\sum^n_{i=1}\Vert X_i-\bar{X}\Vert^2_2$. In Theorem~\ref{Thm:shrinkage}, we show $\hat{C}_{\tilde{\alpha}_{\textup{general}}}$ to strictly improve upon $\bar{X}$ in terms of the mean squared error, for all $\mu\in\mathbb{R}^d$, if 
$n\ge 2$ and $d\ge 4+\frac{2}{n-1}$. A small modification to this estimator, i.e.,
$$\left(1-\frac{2n-2}{3n-1}\cdot\frac{\frac{S^2}{n}}{\frac{S^2}{n}+\Vert \bar{X}\Vert^2_2}\right)\bar{X}$$
yields that for all $d\ge 3$, the above modified estimator strictly improves upon $\bar{X}$ for all $\mu\in\mathbb{R}^d$ (see Theorem~\ref{Thm:family})---a result similar to that of the James-Stein estimator. The proofs of these results are provided in Section~\ref{Sec:proofs} and additional results are provided in an appendix.

\section{Definitions and Notation}\label{Sec:notation}
For $a \triangleq (a_{1}, \dots,a_{d}) \in \mathbb{R}^{d}$, $b \triangleq (b_{1}, \dots,b_{d}) \in \mathbb{R}^{d}$, $\norm{a}_{2} \triangleq \sqrt{\sum_{i=1}^{d} a_{i}^2}$ and $\langle a,b\rangle_2=\sum^d_{i=1}a_ib_i$.   $\Mycomb[n]{i} = \frac{n!}{(n-i)!i!}$, $\Myperm[n]{i} = \frac{n!}{(n-i)!}$ and $S_{n}$ denotes the  symmetric group on $\{1,\dots, n\}$ with $\sigma \in S_{n}$ being a permutation.  $\text{U}_{k}^{n}(r) = \frac{1}{\Myperm[n]{k}} \sum_{I_{k}^{n}}  r(X_{i_1},\dots,X_{i_k}) $ denotes a $U$-statistic with kernel $r$ of order $k$ computed with $n$ variables, where $I_{k}^{n} = \{ (i_{1}, \dots, i_{k}) :   i_1 \neq i_2 \neq \dots \neq  i_{k} \}$. 
% $C(\mathcal{X})$ (\emph{resp.} $C_{b}(\mathcal{X})$) denotes the space of all continuous (\emph{resp}. bounded functions) on a topological space $\mathcal{X}.$ For a locally compact Hausdorff space $\mathcal{X}$, $f \in C(\mathcal{X})$ is said to \emph{vanish at infinity} if for every $\epsilon>0$, the set $\{x :  \abs{f(x)} \geq \epsilon \}$ is compact. The class  of all continuous functions on $\mathcal{X}$ which vanish at infinity is denoted as $C_{0}(\mathcal{X}).$  $M_{b}(\mathcal{X})$ (\emph{resp}. $M_{+}^{1}(\mathcal{X})$) denotes the set of all finite (\emph{resp}. probability) measures defined on $\mathcal{X}.$ For $\mathcal{X} \subset \mathbb{R}^{d}$, $L^{r}(\mathcal{X})$ denotes the Banach space of $r$-power $(r \geq 1)$ Lebesgue integrable functions. For $f \in L^{r}(\mathcal{X})$, $\norm{f}_{L^{r}} \triangleq (\int_{\mathcal{X}} \abs{f(x)}^{r} dx)^{1/r}.$
A function $r : \mathcal{X}^{k} \rightarrow \mathcal{H}$ is said to be \emph{symmetric} if it does not depend on the order of its inputs, i.e., $r(x_{1}, \dots, x_{k}) = r(x_{\sigma(1)}, \dots, x_{\sigma(k)}), \,\, \forall \sigma \in S_{k}$. When $r$ is symmetric, $\textup{U}^n_k(r)$ reduces to  $\text{U}_{k}^{n}(r) = \frac{1}{\Mycomb[n]{k}} \sum_{J_{k}^{n}}  r(X_{i_1},\dots,X_{i_k}),$ where $J_{k}^{n} = \{ (i_{1}, \dots, i_{k}) :   1<i_1 < i_2 <\dots <  i_{k}<n \}$. For a symmetric function $r : \mathcal{X}^{k} \rightarrow \mathcal{H}$ and a probability measure $\mathbb{P}$ on $\mathcal{X}$, the \emph{canonical function of order i with respect to $\mathbb{P}$}, denoted as $r_{i} : \mathcal{X}^{i} \rightarrow \mathcal{H}$, is defined as
    \begin{align*}
        r_{i}(x_1,\dots,x_i) & = \int_{\mathcal{X}^{k-i}}r (x_1,\dots,x_k) \prod_{j=i+1}^{k} d \mathbb{P}(x_j) ,
    \end{align*}
    with the convention $r_{0} := \int_{\mathcal{X}^{k}}r (x_1,\dots,x_k) \prod_{j=1}^{k} d \mathbb{P}(x_j)$ and $r_{k} := r(x_1,\dots,x_k)$.  
% \begin{definition}[Symmetric function]
% A function $r : \mathcal{X}^{k} \rightarrow \mathcal{H}$ is symmetric if it doesn't depend on order of its inputs. $f(x_{1}, \dots, x_{k}) = f(r_{\sigma(1)}, \dots, r_{\sigma(k)}) \,\, \forall \sigma \in S_{k}$ 
% \end{definition}
% with the convention $r_{0} = \int_{\mathcal{X}^{k}}r (x_1,\dots,x_k) \prod_{j=1}^{k} d \mathbb{P}(x_j)  = \mathbb{E}[r (x_1,\dots,x_k)]$ and $r_{k} = r(x_1,\dots,x_k)$. 
% \begin{definition}
%     Let $r : \mathcal{X}^{k} \rightarrow \mathcal{H}$ be symmetric function and $\mathbb{P}$ be measure on $\mathcal{X}$. Given $i \in \{0,1,\dots,k\}$, then $r_{i} : \mathcal{X}^{i} \rightarrow \mathcal{H}$ is defined as, 
%     \begin{align*}
%         r_{i}(x_1,\dots,x_i) & = \int_{\mathcal{X}^{k-i}}r (x_1,\dots,x_k) \prod_{j=i+1}^{k} d \mathbb{P}(x_j)  
%     \end{align*}
%     with the convention $r_{0} = \int_{\mathcal{X}^{k}}r (x_1,\dots,x_k) \prod_{j=1}^{k} d \mathbb{P}(x_j)  = \mathbb{E}[r (x_1,\dots,x_k)]$ and $r_{k} = r(x_1,\dots,x_k)$ 
% \end{definition}
A symmetric function $r : \mathcal{X}^{k} \rightarrow \mathcal{H}$ is \emph{$\mathbb{P}$-complete degenerate} if \emph{(i)} $ \forall  i \in \{0,1,\dots, k-1\}$ and  $\forall x_{1},\dots, x_{i} \in \mathcal{X} $,  $r_{i}(x_1,\dots,x_{i}) = 0$; and \emph{(ii)} $r_{k}$ is not a constant function.

A real-valued symmetric function $K:\mathcal{X} \times \mathcal{X} \to \bb{R}$ is called a positive definite (pd) kernel if, for all $n \in \bb{N}$, $\{\alpha_i\}^n_{i=1} \in \mathbb{R}$ and $\{x_i\}^n_{i=1} \in \mathcal{X}$, we have $\sum_{i,j=1}^n \alpha_i \alpha_j K(x_i,x_j) \geq 0$. A function $K:\mathcal{X} \times \mathcal{X} \to \bb{R}$, $(x,y) \mapsto K(x,y)$ is a \textit{reproducing kernel} of the Hilbert space $(\mathscr{H}_K, \langle\cdot,\cdot \rangle_{\mathscr{H}_K})$ of functions if and only if \emph{(i)} $\forall x \in \mathcal{X}$, $K(\cdot, x) \in \mathscr{H}_K$ and \emph{(ii)} $\forall x \in \mathcal{X}$, $\forall f \in \mathscr{H}_K$, $\langle K(\cdot, x), f\rangle_{\mathscr{H}_K} = f(x)$ hold. If such a $K$ exists, then $\mathscr{H}_K$ is called a \textit{reproducing kernel Hilbert space}.

  \section{Main Results}\label{Sec:main results}
In this section, we present our main results related to the consistency of the shrinkage estimator and oracle bounds for the mean-squared error. Theorem~\ref{Th:NonDegC-NonDegDelta} deals with $r$ being a symmetric function while Theorem~\ref{Th:DegC-DegDelta} considers the case of when $r-C$ is $\mathbb{P}$-complete degenerate. We show that the shrinkage estimator has a faster rate of convergence when $r-C$ is $\mathbb{P}$-complete degenerate (see Theorem~\ref{Th:DegC-DegDelta}) in contrast to $r$ being simply symmetric (see Theorem~\ref{Th:NonDegC-NonDegDelta}). We would like to mention that the shrinkage estimators considered in Theorems~\ref{Th:NonDegC-NonDegDelta} and \ref{Th:DegC-DegDelta} are different as their construction is based on whether $r-C$ is $\mathbb{P}$-complete degenerate or not. In Theorem~\ref{Th:NonDegC-DegDelta}, we show that the shrinkage estimator of Theorem~\ref{Th:DegC-DegDelta}, i.e., the $\mathbb{P}$-complete degenerate case, is still a $\sqrt{n}$-consistent estimator with a slightly slow error rate in the oracle bound, even if $r-C$ is not $\mathbb{P}$-complete degenerate but only symmetric. This result is interesting as the estimator in the degenerate case is simple to compute than the estimator in the symmetric case. 

Before  we present our results, we state the following result, which provides the motivation for the estimator proposed in Theorem~\ref{Th:NonDegC-NonDegDelta}. This result is a simple extension of \citep[Theorem 3]{lee2019u} and the claim in the proof of Theorem 2 of \cite{lee2019u} to Hilbert space-valued random elements.
\begin{theorem}
%[\citealp{lee2019u, boroskikh2020u}]
    \label{Th:VarUstatDecomp}
    Let $\hat{C} =\frac{1}{\Mycomb[n]{k}} \sum_{J^n_k} r(X_{i_1},\dots,X_{i_k})$ be a U-statistics estimator of $$C=\int_{\mathcal{X}^{k}} r(x_1,\dots,x_{k}) \prod_{i=1}^{k}d \mathbb{P}(x_i),$$ where 
$r : \mathcal{X}^{k} \rightarrow \mathcal{H}$ is a symmetric function.
% $C = \int_{\mathcal{X}^{k}} r(x_1,\dots,x_{k}) \prod_{i=1}^{k}d \mathbb{P}(x_i) $ and $\hat{C} = \sum_{i_1<\dots<i_k}^{n} r(X_1,\dots,X_k)$ be a \textup{U}-statistics estimator of $C$. 
Let $$\kappa_{2k-i}(X_1,\dots, X_{2k-i}) =  \inp{r(X_1,\dots,X_k)}{r(X_{1}, \dots,X_{i}, X_{k+1}, \dots, X_{2k-i}}_{\mathcal{H}}$$ for each $i \in \{0,1,\dots, k\}$  . Then,
    \begin{align}
        \mathbb{E}_{X_1,\dots,X_{2k-i}}\big[ \kappa_{2k-i}(X_1,\dots, X_{2k-i})\big] &= \mathbb{E}_{X_1,\dots,X_i}\norm{ r_{i}(X_1,\dots,X_i)}_{\mathcal{H}}^2.\label{Eq:varr}
    \end{align}
    Further, 
    \begin{align}
        \Delta = \mathbb{E}  \Vert\hat{C}\Vert^2_{\mathcal{H}} - \norm{C}^2_{\mathcal{H}} = \frac{1}{\Mycomb[n]{k}} \sum_{i=1}^{k}\Mycomb[k]{i}\, \Mycomb[n-k]{k-i} \, \sigma_{i}^2,\label{Eq:delta-temp}
    \end{align}
    where $\sigma_{i}^2 =  \mathbb{E} \Vert r_{i}(X_1,\dots,X_i) \Vert_{\mathcal{H}}^2 - \Vert \mathbb{E} [r(X_1,\dots,X_{k})] \Vert_{\mathcal{H}}^2$, with $r_i$ being the canonical function of order $i$ with respect $\mathbb{P}$.
\end{theorem}
%\vspace{1mm}~\\
Combining \eqref{Eq:varr} with the observation that $$\Vert \mathbb{E} [r(X_1,\dots,X_{k})] \Vert_{\mathcal{H}}^2 = \mathbb{E}\big[ \kappa_{2k}(X_1,\dots, X_{2k})\big]$$ yields
\begin{equation}\sigma^2_i=\mathbb{E}_{X_1,\dots,X_{2k-i}}\big[ \kappa_{2k-i}(X_1,\dots, X_{2k-i})\big]-\mathbb{E}\big[ \kappa_{2k}(X_1,\dots, X_{2k})\big],\label{Eq:sigmai}\end{equation} which therefore can be estimated 
% and using \eqref{Eq:varr}, we use estimator for each of 
%$\sigma_{i}^2$ 
as $$\hat{\sigma}^2_{i} = \textup{U}_{2k-i}^{n}\left[ \kappa_{2k-i}(X_1,\dots, X_{2k-i}) \right]  -  \textup{U}_{2k}^{n}\left[ \kappa_{2k}(X_1,\dots, X_{2k}) \right],$$ resulting in an estimator for $\Delta$ as $$\hat{\Delta}_{\text{general}} =  \sum_{i=1}^{k} \frac{\Mycomb[k]{i} \Mycomb[n-k]{k-i}}{\Mycomb[n]{k}} \hat{\sigma}^2_{i}.$$ Note that $\kappa_{2k-i}(X_1,\dots, X_{2k-i})$ and $\kappa_{2k}(X_1,\dots, X_{2k})$ need not be symmetric for any $i\in\{1,\ldots,k\}$ and $k\ge 1$, and therefore, $\textup{U}^n_{2k-i}$ and $\textup{U}^n_{2k}$ uses the permutation definition as mentioned in Section~\ref{Sec:notation}. Based on the above, a shrinkage estimator of $C$ can be defined as \begin{equation}\hat{C}_{\tilde{\alpha}_{\textup{general}}}=(1-\tilde{\alpha}_{\text{general}})\hat{C}+\tilde{\alpha}_{\text{general}}f^*,\label{Eq:general-hatC}\end{equation} where \begin{equation*}\tilde{\alpha}_{\text{general}}=\frac{\hat{\Delta}_{\text{general}}}{ \hat{\Delta}_{\text{general}}  + \Vert \hat{C}-f^* \Vert_{\mathcal{H}}^2 }.
%\label{Eq:alpha-general}
\end{equation*} The following result (proved in Section~\ref{sec:proof-nondeg-nondeg}) analyzes the consistency and mean-squared error of $\hat{C}_{\tilde{\alpha}_{\textup{general}}}$.
% Our first data dependant estimator $\tilde{\alpha}_{\text{general}}$  of $\alpha$ is then is given by $\frac{\hat{\Delta}_{\text{general}}}{ \hat{\Delta}_{\text{general}}  + \Vert \hat{C} \Vert_{\mathcal{H}}^2 }.$ With this we state the theorem, 

\begin{theorem}\label{Th:NonDegC-NonDegDelta}
Let $n \geq 2k$, and %and $f^{*} = 0$. Let 
$r: \mathcal{X}^{k} \rightarrow \mathcal{H}$ be a symmetric function such that $\mathbb{E}\Vert r(X_1,\ldots,X_k)\Vert_\mathcal{H}<\infty$,  %$$\int_{\mathcal{X}^{k}} r(x_1,\dots,x_{k}) \prod_{i=1}^{\infty} d  \mathbb{P}(x_i) < \infty,$$ 
where $\mathcal{X}$ is a separable topological space and $\mathcal{H}$ is a separable Hilbert space. Define $$ 
        \hat{\Delta}_{\textup{general}} =\sum_{i=1}^{k} \frac{\Mycomb[k]{i} \Mycomb[n-k]{k-i}}{\Mycomb[n]{k}}\left( \textup{U}_{2k-i}^{n}\left[ \kappa_{2k-i}(X_1,\dots, X_{2k-i}) \right]  -  \textup{U}_{2k}^{n}\left[ \kappa_{2k}(X_1,\dots, X_{2k}) \right]\right).$$ Suppose for all $m\geq 2$ and all $i \in \{0,1,\dots,k \}$,
     \begin{align*}
     \mathbb{E} \Vert r(X_1,\dots,X_k)-C\Vert_{\mathcal{H}}^m  &\leq \frac{m!}{2} \beta^2 \theta^{m-2},\,\,\text{and} \\
%      \end{align*}
%      \begin{align*}
    \mathbb{E} \abs{\kappa_{2k-i}(X_1,\dots,X_{2k-i})- \mathbb{E}[ \kappa_{2k-i}(X_1,\dots,X_{2k-i} )] }^{m} &\leq   \frac{m!}{2} \beta_{i} \theta_{i}^{m-2},
    \end{align*}
 for some finite positive constants $\beta, \theta,  \{\beta_{i}\}_{i=0}^{k}$, and $\{ \theta_{i} \}_{i=0}^{k}$. 
%  For any $f^*\in\mathcal{H}$, let $\alpha_*=\frac{\Delta}{\Delta+\Vert C-f^*\Vert^2_\mathcal{H}}$ and $    \hat{C}_{\tilde{\alpha}_{\textup{general}}} = (1-\tilde{\alpha}_{\textup{general}}) \hat{C}+\tilde{\alpha}_{\textup{general}}f^*$ be an estimator of $C$
% where $\tilde{\alpha}_{\textup{general}}$ is defined in \eqref{Eq:alpha-general}. 
Then, as $n\rightarrow\infty$, the following hold: 
\begin{itemize}
\item[(i)] $ \abs{\tilde{\alpha}_{\textup{general}} - \alpha_{*}  } =  \mathcal{O}_{\mathbb{P}}(n^{- \frac{3}{2} })$;%\vspace{1.5mm}
\item[(ii)] $\abs{ \Vert \hat{C}_{\tilde{\alpha}_{\emph{general}}}-C \Vert_{\mathcal{H}} - \Vert\hat{C}_{\alpha_*}-C \Vert _{\mathcal{H}} } =  \mathcal{O}_{\mathbb{P}}(n^{-\frac{3}{2} })$;%\vspace{1.5mm}
\item[(iii)] $\hat{C}_{\tilde{\alpha}_{\textup{general}}}$ is a $\sqrt{n}$-consistent estimator of $C$;%\vspace{1.5mm}
\item[(iv)] $\min_{\alpha} \mathbb{E} \Vert\hat{C}_{\alpha} - C \Vert^2_{\mathcal{H}} \leq  \mathbb{E} \Vert \hat{C}_{\tilde{\alpha}_{\textup{general}}} - C\Vert^2_{\mathcal{H}} \leq \min_{\alpha} \mathbb{E} \Vert\hat{C}_{\alpha} - C \Vert^2_{\mathcal{H}}  + \mathcal{O}(n^{-2})$,
\end{itemize}
where $\hat{C}_{\tilde{\alpha}_{\textup{general}}}$ is defined in \eqref{Eq:general-hatC}, $\alpha_*$ is defined in \eqref{Eq:alpha-star}, and $\hat{C}_\alpha=(1-\alpha)\hat{C}+\alpha f^*$.
\end{theorem}

%     \begin{proof}
%     See sec.~\ref{sec:proof-nondeg-nondeg}
% \end{proof}
\begin{remark}
    \begin{enumerate}
        \item[(i)]  It follows from Theorem~\ref{Th:NonDegC-NonDegDelta}$(iv)$ that $\Delta_{\tilde{\alpha}_{\text{general}}} \leq \Delta_{\alpha^{*}} + \mathcal{O}(n^{-2})$ as $n \rightarrow \infty$, which when combined with $\Delta_{\alpha^{*}} < \Delta$, yields $\Delta_{\tilde{\alpha}_{\text{general}}} < \Delta + \mathcal{O}(n^{-2})$ as $n \rightarrow \infty$, for all $\mathbb{P}$ that satisfy the moment conditions.%\vspace{1.5mm}
        \item[(ii)] \cite{muandet2016kernel} considered $k=1$, $\mathcal{H}$ to be a  reproducing kernel Hilbert space (RKHS), $\mathscr{H}_K$, with a continuous reproducing kernel, $K$, $f^*=0$ and $r(X)=K(\cdot,X)\in\mathscr{H}_K$, resulting in the problem of shrinkage estimation of the mean element. \citep[Theorem 7]{muandet2016kernel} provides an oracle bound \begin{equation}\min_{\alpha} \mathbb{E} \Vert\hat{C}_{\alpha} - C \Vert^2_{\mathcal{H}} \leq  \mathbb{E} \Vert \hat{C}_{\tilde{\alpha}_\textup{general}} - C\Vert^2_{\mathcal{H}} \leq \min_{\alpha} \mathbb{E} \Vert\hat{C}_{\alpha} - C \Vert^2_{\mathcal{H}}  + \mathcal{O}(n^{-3/2}),\,\, n\rightarrow\infty,\label{Eq:muandet}
        \end{equation}
        which  Theorem~\ref{Th:NonDegC-NonDegDelta}$(iv)$ improves by a providing an improved error rate of $n^{-2}$.%\vspace{1.5mm}
        % $\hat{\Delta} = \frac{1}{n^2} \sum_{i=1}^{n} \kappa(x_i, x_i) - \frac{1}{n^2(n-1)} \sum_{i \neq j}^n \kappa(x_i, x_j)$ that $\min_{\alpha} \mathbb{E} \Vert\hat{C}_{\alpha} - C \Vert_{\mathcal{H}} \leq  \mathbb{E} \Vert \hat{C}_{\tilde{\alpha}} - C\Vert_{\mathcal{H}} \leq \min_{\alpha} \mathbb{E} \Vert\hat{C}_{\alpha} - C \Vert_{\mathcal{H}}  + \mathcal{O}(n^{-3/2})  \text{ as n } \rightarrow \infty.$ (see Theorem 7).   Note for $k=1$, $\hat{\Delta}_{\text{general}}$ in Theorem.~\ref{Th:NonDegC-NonDegDelta} coincides with $\hat{\Delta}$. Our theorem yields stronger error rate of $\mathcal{O}(n^{-2})$ instead of  $\mathcal{O}(n^{-3/2})$ in \cite{muandet2016kernel}. 
        \item[(iii)] With $k=2$, $f^*=0$ and $r(X,Y) = \frac{1}{2}(K(\cdot,X) - K(\cdot,Y)) \otimes_\mathscr{H} (K(\cdot,X) - K(\cdot,Y))$, i.e., the shrinkage estimation of the covariance operator on $\mathscr{H}_K$ with $\mathcal{H}$ being the space of Hilbert-Schmidt operators on $\mathscr{H}_K$, \citep[Theorem 2]{zhou2019class} showed \eqref{Eq:muandet}, which is again improved by Theorem~\ref{Th:NonDegC-NonDegDelta}. Here $\otimes_{\mathscr{H}_K}$ denotes the tensor product on $\mathscr{H}_K$.%\vspace{1.5mm}
    % For $k=2$ and when $\mathcal{H}$ is reproducible kernel hilbert space (RKHS), and when $r(X,Y) = (k(X,.) - k(Y,.)) \otimes (k(X,.) - k(Y,.))$ showed that for $\tilde{\Delta} = $  showed that $\min_{\alpha} \mathbb{E} \Vert\hat{C}_{\alpha} - C \Vert_{\mathcal{H}} \leq  \mathbb{E} \Vert \hat{C}_{\tilde{\alpha}} - C\Vert_{\mathcal{H}} \leq \min_{\alpha} \mathbb{E} \Vert\hat{C}_{\alpha} - C \Vert_{\mathcal{H}}  + \mathcal{O}(n^{-3/2})  \text{ as n } \rightarrow \infty.$ (see Theorem 7). Note that for $k=2$ and $\hat{\Delta}_{\text{general}}$ in Theorem.~\ref{Th:NonDegC-NonDegDelta} coincides with $\hat{\Delta}$. Our theorem yields stronger error rate of $\mathcal{O}(n^{-2})$ instead of  $\mathcal{O}(n^{-3/2})$ in 
    \item[(iv)] Clearly the moment conditions of Theorem~\ref{Th:NonDegC-NonDegDelta} are satisfied if $r$ is bounded. If $r$ is unbounded, then the moment conditions are quite stringent as they require all the higher moment conditions to exist. These conditions can be weakened and the proof
of Theorem~\ref{Th:NonDegC-NonDegDelta} can be carried out using Chebyshev inequality instead of Bernstein's inequality but at the cost of a slow rate in Theorem~\ref{Th:NonDegC-NonDegDelta}$(iv)$.
    % When $\sup_{X_1, \dots, X_{k} \in \mathcal{X}} \norm{\kappa(r(X_1,\dots,X_{k}), (r()))}_{\mathcal{H}} < M $ i.e., it is bounded, then note that 
    % While the moment conditions in theorem are obviously satisfied by bounded kernels, for unbounded kernels, these conditions are quite stringent as they require all the higher moment conditions to exist. Conditions are required because we used Bernstein Inequality for U statistics. These conditions can be carried out using chebyshev inequality instead of bernstein inequality but at the cost of the slow rate. 
        % \item Note that theorem. \ref{Th:NonDegC-NonDegDelta} is pretty general than previous results in that it holds for any function $r: \mathcal{X}^{k} \rightarrow \mathcal{H}$ without assuming any structure and $\mathcal{H}$ can be any Separable Hilbert Space not just reproducible kernel hilbert space.  
    \end{enumerate}
\end{remark}
The following examples specialize the proposed shrinkage estimator for the mean element and covariance operator on a Hilbert space. 
\begin{example}[Mean element, moment generating function and Weierstrass transform]\label{Ex:Mean}
Suppose $k=1$. Then \begin{align*}
    \hat{\Delta}_{\textup{general}} & = \frac{1}{n} \left[ \frac{1}{n}\sum_{i=1}^{n} \langle r(X_i), r(X_i)\rangle_\mathcal{H} - \frac{1}{\Mycomb[n]{2}} \sum_{i<j}^n \langle r(X_i), r(X_j)\rangle_\mathcal{H} \right]
\end{align*}
and 
\begin{align*}
\Vert \hat{C}-f^*\Vert^2_\mathcal{H}&=\left\Vert \frac{1}{n}\sum^n_{i=1}\left(r(X_i)-f^*\right)\right\Vert^2_\mathcal{H}
=\frac{1}{n^2}\sum_{i,j}\langle r(X_i),r(X_j)\rangle_\mathcal{H}-\frac{2}{n}\sum^n_{i=1}\langle r(X_i),f^*\rangle_\mathcal{H}+\Vert f^*\Vert^2_\mathcal{H}.
\end{align*}
Define $K(x,y)=\langle r(x),r(y)\rangle_\mathcal{H}$, $x,y\in\mathcal{H}$. It is easy to verify that $K$ is a positive definite kernel and therefore a reproducing kernel \citep{Aronszajn-50} of some reproducing kernel Hilbert space (RKHS), $\mathscr{H}_K$ so that $K(x,y)=\langle K(\cdot,x),K(\cdot,y)\rangle_{\mathscr{H}_K}$. Note that these quantities match those proposed in \citep{muandet2016kernel}, where $r(x)=K(\cdot,x)$ and $f^*=0$, resulting in a mean element of $\mathbb{P}$ in $\mathscr{H}_K$. When $\mathcal{X}=\mathbb{R}^d$ and $r(x)=x$ for $x\in\mathbb{R}^d$, $\mathbb{E}[r(X)]$ corresponds to the mean vector in $\mathbb{R}^d$ and $K(x,y)=\langle x,y\rangle_2$ is the linear kernel. We analyze this scenario in detail in Section~\ref{Sec:normal} when $\mathbb{P}$ is a Gaussian distribution. 

The choice of $r(x)=e^{\langle \cdot,x\rangle_2}$ with $\mathcal{H}$ being an RKHS of exponential kernel, i.e., $K(x,y)=e^{\langle x,y\rangle_2}=\langle r(x),r(y)\rangle_\mathcal{H}$, $x,y\in\mathbb{R}^d$, results in a shrinkage estimator for the moment generating function. Equivalently, this choice can be interpreted as 
$$r(x)=\left(1,(x_i)^d_{i=1}, (x_{i_1}x_{i_2}/\sqrt{2!})^d_{i_1,i_2=1},\ldots, \left(\prod^m_{j=1} x_{i_j}/\sqrt{m!}\right)^d_{i_1,\ldots,i_m=1},\ldots\right)$$ with $\mathcal{H}=\ell^2(\mathbb{N})$. Similarly, the choice of $r(x)=e^{\Vert \cdot-x\Vert^2_2}$ with $\Cal{H}$ being an RKHS of a Gaussian kernel, i.e., $K(x,y)=e^{\Vert x-y\Vert^2_2},\,x,y\in\mathbb{R}^d$, results in a shrinkage estimator for the Weierstrass transform of $\mathbb{P}$.
% Then 
% Let $\mathcal{H}$ be a reproducing kernel Hilbert space with reproducing kernel $K:\mathcal{X}\times\mathcal{X}\rightarrow\mathbb{R}$, defined on a topological space $\mathcal{X}$. The choice of $k=1$ and $r(X)=K(\cdot,X)$ yields the mean element, $\int_\mathcal{X} K(\cdot,x)\,d\mathbb{P}(x)$ of $\mathbb{P}$. Then
% \begin{align*}
%     \hat{\Delta}_{\textup{general}} & = \frac{1}{n} \left[ \frac{1}{n}\sum_{i=1}^{n} K(X_i, X_i) - \frac{1}{\Mycomb[n]{2}} \sum_{i<j}^n K(X_i, X_j) \right]
% \end{align*}
% and 
% \begin{align*}
% \Vert \hat{C}\Vert^2_\mathcal{H}=\left\Vert \frac{1}{n}\sum^n_{i=1}K(\cdot,X_i)\right\Vert^2_\mathcal{H}=\frac{1}{n^2}\sum_{i,j}K(X_i,X_j),    
% \end{align*}
% which match with the estimator proposed in \citep{muandet2016kernel}. When $\mathcal{X}=\mathbb{R}^d$ and $K(\cdot,x)=x$ for $x\in\mathbb{R}^d$, $\mathbb{E}[r(X)]$ corresponds to the mean vector in $\mathbb{R}^d$. \textcolor{red}{Choosing $K(\cdot,X)=e^{\langle \cdot,X\rangle_2}$ yields a mgf.}
\end{example}

\begin{example}[Covariance operator]\label{Ex:cov} %\textcolor{red}{(put a note on tensors)}
Let $\mathcal{H}$ be the space of Hilbert-Schmidt operators defined on a reproducing kernel Hilbert space $\mathscr{H}_K$ with $K:\mathcal{X}\times\mathcal{X}\rightarrow\mathbb{R}$ as the reproducing kernel,  defined on a topological space $\mathcal{X}$. Choosing $k=2$ and $$r(X,Y) = \frac{1}{2}(K(\cdot,X) - K(\cdot,Y)) \otimes_{\mathscr{H}_K} (K(\cdot,X) - K(\cdot,Y))$$ yields the covariance operator on $\mathscr{H}_K$. 
% If  $K(\cdot,X)=X\in H$, where $X$ is a $H$-valued random element with $H$ being a separable Hilbert space, then $\mathbb{E}[r(X,Y)]$ is the covariance operator on $H$. 
Note that
\begin{align*}
4\inp{r(X,Y)}{r(U,V)}_\mathcal{H} &= \left\langle(K(\cdot,X) - K(\cdot,Y)) \otimes_{\mathscr{H}_K} (K(\cdot,X) - K(\cdot,Y)),\right.\\
&\qquad\qquad \left.(K(\cdot,U) - K(\cdot,V)) \otimes_{\mathscr{H}_K} (K(\cdot,U) - K(\cdot,V))\right\rangle_{\mathcal{H}}\\
&= \inp[\Big]{K(\cdot,X) - K(\cdot,Y)} {K(\cdot,U)- K(\cdot,V)}^2_{\mathscr{H}_K}\\
&=\left[K(X,U)-K(X,V)-K(Y,U)+K(Y,V)\right]^2.
\end{align*}
    % $\norm{(K(\cdot,U) - K(\cdot,V)) \otimes_\mathcal{H} (K(\cdot,U) - K(\cdot,V))}^2_{\mathcal{H}}$ and again simplifying we have that $\inp{r(X,Y)}{r(U,V)} =  \left[ K(X,X) - 2k(X,Y) + K(Y,Y) \right] \left[ K(U,U) - 2k(U,V) + K(V,V) \right]$
    % from which we have that 
Therefore,    
\begin{align*}
\hat{\Delta}_{\textup{general}} 
&=  \frac{2n-4 }{\Mycomb[n]{2}} \textup{U}^n_3\left[\kappa_3(X_1,X_2,X_3)\right]-\frac{2n-3}{\Mycomb[n]{2}}\textup{U}^n_4\left[\kappa_4(X_1,X_2,X_3,X_4)\right]+\frac{1}{\Mycomb[n]{2}}\textup{U}^n_2\left[\kappa_2(X_1,X_2)\right]\\
&=  \frac{2n-4 }{\Mycomb[n]{2}} \textup{U}^n_3\left[\left\langle r(X_1,X_2),r(X_1,X_3)\right\rangle_\mathcal{H}\right]+\frac{1}{\Mycomb[n]{2}}\textup{U}^n_2\left[\left\langle r(X_1,X_2),r(X_1,X_2)\right\rangle_\mathcal{H}\right]\\
&\qquad\qquad-\frac{2n-3}{\Mycomb[n]{2}}\textup{U}^n_4\left[\left\langle r(X_1,X_2),r(X_3,X_4)\right\rangle_\mathcal{H}\right]\\
&=\frac{2n-4 }{\Mycomb[n]{2}\cdot \Myperm[n]{3}}\sum_{i\ne j\ne l}\left\langle r(X_i,X_j),r(X_i,X_l)\right\rangle_\mathcal{H}+\frac{1}{\Mycomb[n]{2}\cdot \Myperm[n]{2}}\sum_{i\ne j}\left\langle r(X_i,X_j),r(X_i,X_j)\right\rangle_\mathcal{H}\\
&\qquad\qquad-\frac{2n-3 }{\Mycomb[n]{2}\cdot \Myperm[n]{4}}\sum_{i\ne j\ne l\ne m}\left\langle r(X_i,X_j),r(X_l,X_m)\right\rangle_\mathcal{H}\\
&=  \frac{2n-4 }{4\cdot\Mycomb[n]{2}\cdot \Myperm[n]{3}} \sum_{i \ne j \ne l}  
\left[K(X_i,X_i)-K(X_i,X_l)-K(X_i,X_j)+K(X_j,X_l)\right]^2
% \left[ \kappa(X_i,X_i) - 2\kappa(X_i,X_j) + K(X_j,X_j) \right] \left[ \kappa(X_i,X_i) - 2\kappa(X_i,X_k) + \kappa(X_k,X_k) \right]  
\\
&\qquad+ \frac{1}{4\cdot\Mycomb[n]{2}\cdot \Myperm[n]{2} } \sum_{i\ne j}  \left[ K(X_i,X_i) - 2K(X_i,X_j) + K(X_j,X_j) \right]^2  \\
%         \end{align*}
%         \begin{align*}
&\qquad\qquad-  \frac{2n-3}{4\cdot\Mycomb[n]{2}\cdot \Myperm[n]{4}} \sum_{i \ne j \ne l \ne m }  \left[ K(X_i,X_l) - K(X_i,X_m) - K(X_j,X_l)+ K(X_j,X_m)\right]^2.
\end{align*}
Also for any $f^*\in\mathcal{H}$,
\begin{align*}
&\Vert \hat{C}-f^*\Vert^2_\mathcal{H}=\left\Vert \frac{1}{\Myperm[n]{2}}\sum_{i\ne j}\left(r(X_i,X_j)-f^*\right)\right\Vert^2_\mathcal{H}\\
&= \frac{1}{\Myperm[n]{2}\cdot \Myperm[n]{2}}\sum_{i\ne j}\sum_{l\ne m}\left\langle r(X_i,X_j),r(X_l,X_m)\right\rangle_\mathcal{H}-\frac{2}{\Myperm[n]{2}}\sum_{i\ne j}\left\langle r(X_i,X_j),f^*\right\rangle_\mathcal{H}+\Vert f^*\Vert^2_\mathcal{H}\\
&=\frac{1}{4\cdot\Myperm[n]{2}\cdot \Myperm[n]{2}}\sum_{i\ne j}\sum_{l\ne m}\left[ K(X_i,X_l) - K(X_i,X_m) - K(X_j,X_l)+ K(X_j,X_m)\right]^2\\
&\qquad-\frac{1}{\Myperm[n]{2}}\sum_{i\ne j}\left\langle K(\cdot,X_i)-K(\cdot,X_j), f^*\left(K(\cdot,X_i)-K(\cdot,X_j)\right)\right\rangle_{\mathscr{H}_K}+\Vert f^*\Vert^2_\mathcal{H}.
\end{align*}
We would like to highlight that the expressions provided in \citep{zhou2019class} for the above quantities are only asymptotically equivalent to ours when $f^*=0$ because of the approximations the authors employed to simplify their asymptotic analysis. 

For $K(x,y)=\langle x,y\rangle_2,\,x,y\in\mathbb{R}^d$ and $f^*=I_d$ (the $d\times d$ identity matrix), it can be shown that (see Proposition~\ref{appxpro:cov})
\begin{align*}
   \hat{\Delta}_{\textup{general}} &=  \frac{1}{(n-2)(n-3)} \sum_{i=1}^{n} \Vert \tilde{X}_i \Vert_{2}^4 - \frac{n(n+1)}{(n-1)^2(n-3)} \textup{Tr} [\hat{\Sigma}^2]
%    \end{align*}
% \begin{align*}
   %&\qquad\qquad
   - \frac{n}{(n-1)(n-2)(n-3)} \textup{Tr}^2[\hat{\Sigma}],\end{align*}
   and
   \begin{align*}
%    \hat{\Delta}_{\textup{degen}} &=\frac{n (n^2-3n+4)}{ 2\cdot\Mycomb[n]{2}\cdot \Myperm[n]{4}}  \sum_{i=1}^{n} \Vert \tilde{X_i} \Vert_{2}^4 -  \frac{2n^2(n-2)}{\Mycomb[n]{2}\cdot \Myperm[n]{4}} \textup{Tr}[\hat{\Sigma}^2] +  \frac{n^2 (n^2-5n+4)}{2\cdot \Mycomb[n]{2} \cdot\Myperm[n]{4} }  \textup{Tr}^2 [\hat{\Sigma}],\,\,\,\text{and}\\
%    \end{align*}
% %   and
% \begin{align*}
   \Vert \hat{C}-I\Vert^2_F&=\frac{n^2}{(n-1)^2} \textup{Tr}[\hat{\Sigma}^2] -\frac{2n}{n-1} \textup{Tr} [\hat{\Sigma}] + d,
\end{align*}
where $\tilde{X}_i=X_i-\bar{X},\,i=1,\ldots,n$, $\hat{\Sigma}=\frac{1}{n}\sum^n_{i=1}\tilde{X}_i\tilde{X}^\top_i$, and $\hat{C} = \frac{1}{\Mycomb[n]{2}} \sum_{i < j} \frac{(X_i - X_j)(X_i - X_{j})^\top}{2}$, with $\Vert\cdot\Vert_F$ being the Frobenius norm.
% 
% 
% 
% In Proposition~\ref{appxpro:cov} of the supplement, we specialize the calculations of the above terms when $K(x,y)=\langle x,y\rangle_2$, $x,y\in\mathbb{R}^d$ and $f^*=I_d$, the $d$-dimensional identity matrix.
% \textcolor{red}{Write example of cross-covariance operator.}
\end{example}
Theorem~\ref{Th:NonDegC-NonDegDelta} is based on Bernstein's inequality for Hilbert space-valued $U$-statistics, which guarantees that $\hat{C}$ and $\hat{C}_{\tilde{\alpha}_{\textup{general}}}$ are $\sqrt{n}$-consistent estimators of $C$. However, if $r-C$ is bounded, real-valued, symmetric, $\mathbb{P}$-complete degenerate of $k \geq 2$ variables, \cite{arcones1993limit,de2012decoupling} showed that there exists finite positive constants $c_1,c_2$ depending only on $k$ such that for all $\delta \in (0,1)$,  
\begin{align*}
    \mathbb{P} \left\{ \abs{\textup{U}_{k}^n(r) - C } \geq  \left( \frac{ \sigma^2 \log (\frac{c_1}{\delta}) }{c_{2}n} \right)^{\frac{k}{2}} +  \norm{r}_{\infty} \left( \frac{\log (\frac{c_1}{\delta}) }{c_{2}n} \right)^{\frac{k+1}{2}}  \right\} \leq \delta,
\end{align*}
where $\norm{r}_{\infty} = \sup_{x_1,\dots,x_{k}} |r(x_1, \dots,x_{k})|$  and $\sigma^2 = \mathbb{E}(\textup{U}_{k}^n(r) - C)^2$ denotes variance.  For $k=2$, this implies a rate of $n^{-1}$ to estimate $C$ using $\textup{U}^n_k(r)$, which is significantly faster than the usual $n^{-1/2}$-rate that is obtained by Bernstein's inequality that does not take into account the complete degeneracy of $r-C$. \cite{joly2016robust} showed a similar result for median-of-means estimator with motivation of robust mean estimation in presence of heavy tails. In Theorem~\ref{Th:BernUstatsDeg}, we generalize this result to unbounded, $\mathcal{H}$-valued,  $\mathbb{P}$-complete degenerate $U$-statistics using the ideas from \citep{de2012decoupling}. Using this result, we devise an estimator of $\alpha$ denoted as $\tilde{\alpha}_{\textup{degen}}$ when $r-C$ is $\mathbb{P}$-complete degenerate, using which we show  $\hat{C}_{\tilde{\alpha}_{\textup{degen}}}=(1-\tilde{\alpha}_{\textup{degen}})\hat{C}+\tilde{\alpha}_{\textup{degen}}f^*$ to be $n^{k/2}$-consistent estimator of $C$. Further, we provide improved error bound rates in the oracle inequality associated with $\hat{C}_{\tilde{\alpha}_{\textup{degen}}}$. 
\par Our design of  $\tilde{\alpha}_{\textup{degen}}$ is based on the variance decomposition of $U$-statistics (see Theorem~\ref{Th:VarUstatDecomp}) and the definition of degeneracy. First if $r-C$ is $\mathbb{P}$-complete degenerate we have that 
 $ \forall  i \in \{0,1,\dots, k-1\}$ and  $\forall x_{1},\dots, x_{i} \in \mathcal{X} $,
 $ r_{i}(x_1,\dots,x_i)- C =0$, which implies that $\sigma_{i}^2=0$. It therefore follows from \eqref{Eq:delta-temp} and \eqref{Eq:sigmai} that $$\Delta = \frac{1}{\Mycomb[n]{k}} \left[ \mathbb{E}\big[ \kappa_{k}(X_1,\dots, X_{k})\big] -  \mathbb{E}\big[ \kappa_{2k}(X_1,\dots, X_{2k})\big] \right].$$ Using this observation, we consider the following estimator for $\Delta$,
\begin{align*}
    \hat{\Delta}_{\text{degen}} &= \frac{1}{\Mycomb[n]{k}} \Big[ \textup{U}_{k}^{n}\big[ \kappa_{k}(X_1,\dots, X_{k})\big] -  \textup{U}_{2k}^{n}\big[ \kappa_{2k}(X_1,\dots, X_{2k})\big] \Big]%\label{Eq:delta-degen}
\end{align*}
so that \begin{equation}\tilde{\alpha}_{\textup{degen}}  = \frac{\hat{\Delta}_{\textup{degen}}}{\hat{\Delta}_{\textup{degen}} + \Vert\hat{C}\Vert^2_{\mathcal{H}}}.
\label{Eq:alpha-degen}
\end{equation}
Note that $\hat{\Delta}_{\textup{general}}=\hat{\Delta}_{\textup{degen}}$ when $k=1$.
%With this we state the following theorem. 
%\subsection{Speeding up Shrinkage Estimator } \label{sec:nondeg-deg}
The following result (proved in Section~\ref{sec:proof-deg-deg}) presents the statistical behavior of $\hat{C}_{\tilde{\alpha}_{\textup{degen}}}$.
\begin{theorem}\label{Th:DegC-DegDelta}
Let $n \geq 2k$, $k\ge 2$, $r:\mathcal{X}^k\rightarrow\mathcal{H}$ be a symmetric function such that $\mathbb{E}\Vert r(X_1,\ldots,X_k)\Vert^2_\mathcal{H}<\infty$ and $r-C$ is $\mathbb{P}$-complete degenerate, where $\mathcal{X}$ is a separable topological space and $\mathcal{H}$ is a separable Hilbert space. Suppose there exists positive constants $M,\sigma_1,\sigma_2$ and $ \theta, \theta_1, \theta_2$, such that $\forall p \geq 2$,
    \begin{align*}
      \mathbb{E} \Big| \norm{r(X_1,\dots,X_k)-C}^2 - \mathbb{E} \norm{r(X_1,\dots,X_k)-C}^2 \Big|^p &\leq \frac{p!}{2} \theta^2 M^{p-2}, \\
      \mathbb{E} \Big|\kappa_{k}(X_1,\dots,X_{k})- \mathbb{E}[\kappa_{k}(X_1,\dots,X_{k})] \Big|^{p} &\leq \frac{p!}{2} \sigma^2_{1} \theta_1^{p-2},\,\,\,\text{and} \\
%       \end{align*}
%         \begin{align*}
      \mathbb{E} \Big|\kappa_{2k}(X_1,\dots,X_{2k})- \mathbb{E}[\kappa_{2k}(X_1,\dots,X_{2k})] \Big|^{p} &\leq \frac{p!}{2} \sigma^2_{2} \theta_2^{p-2}.
    \end{align*}
% Define 
% \begin{align*}
% \hat{\Delta}_{\textup{degen}} &= \frac{1}{\Mycomb[n]{k}}  \left[ \textup{U}_{k}^n \Big(\kappa( (X_{1},\dots,X_{k}), (X_{1},\dots,X_{k}) ) \Big) - \textup{U}_{2k}^n \Big( \kappa( (X_{1},\dots,X_{k}), (X_{k+1},\dots,X_{2k}) ) \Big)  \right].
% \end{align*}
% Suppose the following conditions hold:
% \begin{enumerate}
% \item $\mathbb{E}\Vert r(X_1,\ldots,X_k)\Vert^2_\mathcal{H}<\infty$, $\Vert C\Vert_\mathcal{H}<\infty$;\vspace{1.5mm}
%     \item $r-C$ is $\mathbb{P}$-complete degenerate;\vspace{1.5mm}
%     \item There exists positive constants $M,\sigma_1,\sigma_2$ and $ \theta, \theta_1, \theta_2$ and  such that $\forall p \geq 2$,
%     \begin{align*}
%       \mathbb{E} \Big| \norm{r(X_1,\dots,X_k)-C}^2 - \mathbb{E} \norm{r(X_1,\dots,X_k)-C}^2 \Big|^p &\leq \frac{p!}{2} \theta^2 M^{p-2}, \\
%       \mathbb{E} \Big|\kappa_{k}(X_1,\dots,X_{k})- \mathbb{E}[\kappa_{k}(X_1,\dots,X_{k})] \Big|^{p} &\leq \frac{p!}{2} \sigma^2_{1} \theta_1^{p-2},\,\,\,\text{and} \\
%       \mathbb{E} \Big|\kappa_{2k}(X_1,\dots,X_{2k})- \mathbb{E}[\kappa_{2k}(X_1,\dots,X_{2k})] \Big|^{p} &\leq \frac{p!}{2} \sigma^2_{2} \theta_2^{p-2}.
%     \end{align*}
% \end{enumerate}
% Define $\hat{C}_{\tilde{\alpha}_{\textup{degen}}} = (1-\tilde{\alpha}_{\textup{degen}}) \hat{C}$, where $\tilde{\alpha}_{\textup{degen}}$ is defined in \eqref{Eq:alpha-degen}. 
%is an estimator of $C$
% where $\tilde{\alpha}_{\textup{degen}}  = \frac{\hat{\Delta}_{\textup{degen}}}{\hat{\Delta}_{\textup{degen}} + \Vert\hat{C}\Vert^2_{\mathcal{H}}}$. 
Then, as $n\rightarrow\infty$, the following hold:
         \begin{itemize}
\item[(i)] $ \left|\tilde{\alpha}_{\textup{degen}} - \alpha_{*}  \right| =  \mathcal{O}_{\mathbb{P}}(n^{- (2k+1)/2 })$;%\vspace{1.5mm}
\item[(ii)] $\left| \Vert\hat{C}_{\tilde{\alpha}_{\textup{degen}}}-C
\Vert_{\mathcal{H}} - \Vert\hat{C}_{\alpha_*}-C\Vert_{\mathcal{H}}  \right|=  \mathcal{O}_{\mathbb{P}}(n^{- (2k+1)/2 })$;%\vspace{1.5mm}
\item[(iii)] $\hat{C}_{\tilde{\alpha}_{\textup{degen}}}$ is a $n^{k/2}$-consistent estimator of $C$;%\vspace{1.5mm}
\item[(iv)] $\min_{\alpha} \mathbb{E} \Vert\hat{C}_{\alpha} - C\Vert^2_\mathcal{H} \leq  \mathbb{E} \Vert\hat{C}_{\tilde{\alpha}_{\textup{degen}}} - C\Vert^2_\mathcal{H}\leq \min_{\alpha} \mathbb{E} \Vert\hat{C}_{\alpha} - C\Vert^2_\mathcal{H} + \mathcal{O}(n^{-(3k+1)/2})$,
\end{itemize}
where $\alpha_*$ is defined in \eqref{Eq:alpha-star}, $\tilde{\alpha}_{\textup{degen}}$ is defined in \eqref{Eq:alpha-degen}, and $\hat{C}_\alpha=(1-\alpha)\hat{C}+\alpha f^*$.
\end{theorem}
% \begin{proof}
%     see sec.~\ref{sec:proof-deg-deg}
% \end{proof}

% \begin{remark}
%     \begin{enumerate}
%         \item when $r-C$ is completely degenerate, \cite{arcones1993limit, de2012decoupling} prove that $\hat{C}$ is $n^{k/2}$-consistent estimator of $C$, above theorem shows how to do it in shrinkage setting. 
%         \item $\mathbb{E} \Vert\hat{C}_{\tilde{\alpha}_{\textup{degen}}} - C\Vert\leq \min_{\alpha} \mathbb{E} \Vert\hat{C}_{\alpha} - C\Vert + \mathcal{O}(n^{-(3k+1)/2})$ so for $k=2$ we have error rate of $\mathcal{O}(n^{-3.5})$ instead of $\mathcal{O}(n^{-2})$ from Theorem.\eqref{Th:NonDegC-NonDegDelta}, which is much slower. Difference is much higher for larger $k$.
        
%         \item Note that number of summations in 
%     \end{enumerate}
% \end{remark}

 Now, inspired by our analysis of completely degenerate case, we show that $ \tilde{\alpha}_{\textup{degen}}$ is a good estimator of $\alpha_*$ even if $r-C$ is not $\mathbb{P}$-complete degenerate. Specifically, we show that without any assumption of degeneracy,   $\left|\tilde{\alpha}_{\textup{degen}} - \alpha_{*}  \right| =  \mathcal{O}_{\mathbb{P}}(n^{-1 })$ (compared to $\mathcal{O}_{\mathbb{P}}(n^{-3/2})$ with $\tilde{\alpha}_{\textup{general}}$), $\hat{C}_{\tilde{\alpha}_{\textup{degen}}}$ is a $\sqrt{n}$-consistent estimator of $C$ and more importantly that $ \mathbb{E} \Vert\hat{C}_{\tilde{\alpha}_{\textup{degen}}} - C\Vert^2_\mathcal{H}\leq \min_{\alpha} \mathbb{E} \Vert\hat{C}_{\alpha} - C\Vert^2_\mathcal{H} + \mathcal{O}(n^{-3/2})$ (in contrast to $\mathcal{O}(n^{-2})$) as $n \rightarrow \infty.$ This is surprising because the number of terms in $\hat{\Delta}_{\textup{degen}}$ remains constant with $k$ whereas the number of terms in $\hat{\Delta}_{\textup{general}}$ grows linearly with $k$. This means $\hat{\Delta}_{\textup{degen}}$ is computationally efficient than $\hat{\Delta}_{\textup{general}}$ and therefore is $\hat{C}_{\tilde{\alpha}_{\textup{degen}}}$ over $\hat{C}_{\tilde{\alpha}_{\textup{general}}}$. These are captured in the following result, which is proved in Section~\ref{sec:proof-nondeg-deg}.
  \begin{theorem} \label{Th:NonDegC-DegDelta}
   Let $n \geq 2k$, $k\ge 2$, $r:\mathcal{X}^k\rightarrow\mathcal{H}$ be a symmetric function such that $\mathbb{E}\Vert r(X_1,\ldots,X_k)\Vert^2_\mathcal{H}<\infty$ and $r-C$ is $\mathbb{P}$-complete degenerate, where $\mathcal{X}$ is a separable topological space and $\mathcal{H}$ is a separable Hilbert space. Suppose there exists positive constants $\sigma,\sigma_1,\sigma_2$ and $ \theta, \theta_1, \theta_2$ such that $\forall p \geq 2$, 
 \begin{align*}
     \mathbb{E} \Vert r(X_1,\ldots,X_k)-C\Vert_{\mathcal{H}}^p  &\leq \frac{p!}{2} \sigma^2 \theta^{p-2}, \\
   \mathbb{E} \Big|\kappa_{k}(X_1,\dots,X_{k})- \mathbb{E}[\kappa_{k}(X_1,\dots,X_{k})] \Big|^{p} &\leq \frac{p!}{2} \sigma^2_{1} \theta_1^{p-2},\,\,\,\text{and} \\
      \mathbb{E} \Big|\kappa_{2k}(X_1,\dots,X_{2k})- \mathbb{E}[\kappa_{2k}(X_1,\dots,X_{2k})] \Big|^{p} &\leq \frac{p!}{2} \sigma^2_{2} \theta_2^{p-2}.
\end{align*}
% Define $\alpha_*=\frac{\Delta}{\Delta+\Vert C\Vert^2_\mathcal{H}}$ and $    \hat{C}_{\tilde{\alpha}} = (1-\tilde{\alpha}) \hat{C}$ is an estimator of $C$
% where $\tilde{\alpha}_{\textup{degen}}  = \frac{\hat{\Delta}_{\textup{degen}}}{\hat{\Delta}_{\textup{degen}} + \Vert\hat{C}\Vert^2_{\mathcal{H}}}$. 
Then, as $n\rightarrow\infty$, the following hold:
\begin{itemize}
\item[(i)] $ \left|\tilde{\alpha}_{\textup{degen}} - \alpha_{*}  \right| =  \mathcal{O}_{\mathbb{P}}(n^{-1 })$;%\vspace{1.5mm}
\item[(ii)] $\left| \Vert\hat{C}_{\tilde{\alpha}_{\textup{degen}}}-C
\Vert_{\mathcal{H}} - \Vert\hat{C}_{\alpha_*}-C\Vert_{\mathcal{H}}  \right|=  \mathcal{O}_{\mathbb{P}}(n^{-1 })$;%\vspace{1.5mm}
\item[(iii)] $\hat{C}_{\tilde{\alpha}_{\textup{degen}}}$ is a $\sqrt{n}$-consistent estimator of $C$;%\vspace{1.5mm}
\item[(iv)] $\min_{\alpha} \mathbb{E} \Vert\hat{C}_{\alpha} - C\Vert^2_\mathcal{H} \leq  \mathbb{E} \Vert\hat{C}_{\tilde{\alpha}_{\textup{degen}}} - C\Vert^2_\mathcal{H}\leq \min_{\alpha} \mathbb{E} \Vert\hat{C}_{\alpha} - C\Vert^2_\mathcal{H} + \mathcal{O}(n^{-3/2})$.
\end{itemize}
\end{theorem}

In the above result, we assumed $k>1$. The reason being, when $k=1$, we have $\hat{\Delta}_{\textup{general}}=\hat{\Delta}_{\textup{degen}}$, and the claims follow from Theorem~\ref{Th:NonDegC-NonDegDelta}. 
\begin{example}[Covariance operator]\label{Ex:cov-2}
For the same setting as in Example~\ref{Ex:cov}, we obtain
    \begin{align*}
        \hat{\Delta}_{\textup{degen}}
        &=\frac{1}{\Mycomb[n]{2}}\textup{U}^n_2\left[\kappa_2(X_1,X_2)\right]-\frac{1}{\Mycomb[n]{2}}\textup{U}^n_4\left[\kappa_4(X_1,X_2,X_3,X_4)\right]\\
        &= \frac{1}{4\cdot\Mycomb[n]{2}\cdot \Myperm[n]{2} } \sum_{i\ne j}  \left[ K(X_i,X_i) - 2k(X_i,X_j) + K(X_j,X_j) \right]^2  \\
        &\quad-  \frac{1}{4\cdot \Mycomb[n]{2}\cdot \Myperm[n]{4}} \sum_{i \ne j \ne l \ne m }  \left[ K(X_i,X_l) - k(X_i,X_m) - K(X_j,X_l) + K(X_j,X_m) \right]^2,
    \end{align*}
    which reduces to
    \begin{align*}
    \hat{\Delta}_{\textup{degen}} &=\frac{n (n^2-3n+4)}{ 2\cdot\Mycomb[n]{2}\cdot \Myperm[n]{4}}  \sum_{i=1}^{n} \Vert \tilde{X_i} \Vert_{2}^4 -  \frac{2n^2(n-2)}{\Mycomb[n]{2}\cdot \Myperm[n]{4}} \textup{Tr}[\hat{\Sigma}^2] +  \frac{n^2 (n^2-5n+4)}{2\cdot \Mycomb[n]{2} \cdot\Myperm[n]{4} }  \textup{Tr}^2 [\hat{\Sigma}],
    \end{align*}
    when $K(x,y)=\langle x,y\rangle_2,\,x,y\in\mathbb{R}^d$. See Proposition~\ref{appxpro:cov} for details.%(see the supplement) to the case of $K(x,y)=\langle x,y\rangle_2,\,x,y\in\mathbb{R}^d$.
\end{example}
The proposed shrinkage estimators $\hat{C}_{\tilde{\alpha}_{\textup{general}}}$ and $\hat{C}_{\tilde{\alpha}_{\textup{degen}}}$ can be shown to be solutions to regularized minimization problems. Since
$$\hat{C}_\alpha=\arg\inf_{g\in\mathcal{H}}\frac{1}{\Mycomb[n]{k}}\sum_{(i_1,\ldots,i_k)\in J^n_k}\left\Vert r(X_{i_1},\ldots,X_{i_k})-g\right\Vert^2_\mathcal{H}+\frac{\alpha}{1-\alpha}\Vert g-f^*\Vert^2_\mathcal{H},$$
where $\frac{\alpha}{1-\alpha},\,0<\alpha<1$ acts as the regularization parameter, it follows that the choice of $\frac{\tilde{\alpha}_{\textup{general}}}{1-\tilde{\alpha}_{\textup{general}}}$ and $\frac{\tilde{\alpha}_{\textup{degen}}}{1-\tilde{\alpha}_{\textup{degen}}}$ as regularization parameters yield $\hat{C}_{\tilde{\alpha}_{\textup{general}}}$ and $\hat{C}_{\tilde{\alpha}_{\textup{degen}}}$, respectively. This demonstrates the regularization effect of shrinkage estimators. A similar result was shown in \citep{muandet2016kernel} when $f^*=0$, $\mathcal{H}=\mathscr{H}_K$, $k=1$ and $r(x)=K(\cdot,x)$. % \textcolor{blue}{There is Annals papers and a Jmlr paper and Neurips about  EMPIRICAL MINIMIZATION OF U -STATISTICS  we motivate saying this is interesting problem  ?}

  \section{Normal Mean Estimation}\label{Sec:normal}
In Section~\ref{Sec:main results}, we only established oracle bounds on the mean squared error that include an error term, since no parametric assumptions were made on $\mathbb{P}$. In this section, we study the estimator $\hat{C}_{\tilde{\alpha}_{\textup{general}}}$ when $\Cal{X}=\mathbb{R}^d$, $\mathcal{H}=\mathbb{R}^d$, $r(x)=x$ and $\mathbb{P}$ is a normal distribution, i.e., the shrinkage estimation of normal mean. Note that the degenerate case is not applicable in this setting as  $k=1$. This is the classical setting studied heavily in the literature \citep{brandwein2012stein}. Since $\mathbb{P}$ is Gaussian, we show that concrete results can be obtained on the mean-squared error of $\hat{C}_{\tilde{\alpha}_{\textup{general}}}$, in contrast to oracle inequalities of the previous section. 

Define $C=\int_\mathcal{X} r(x)\,d\mathbb{P}(x)=\int x\,d\mathbb{P}(x)=:\mu$ and $\hat{C}=\frac{1}{n}\sum^n_{i=1}X_i=\bar{X}=:\hat{\mu}$. In this setting with $f^*=0$, it is easy to verify that
 %   All above theorems doesn't have strict inequality. In this section we will prove strict inequality when estimation mean of  normal distribution.  Suppose $\mathcal{X}=\mathbb{R}^d$, $r(x)=x$ (i.e., $H=\mathbb{R}^d$) and $\mathbb{P}$ is a normal distribution with mean $\mathbb{R}^d$ and covariance $\sigma^2 I$.
\begin{eqnarray*}
\hat{\Delta}_{\textup{general}}&{}={}&\frac{1}{n}\left[\frac{1}{n}\sum^n_{i=1}\Vert X_i\Vert^2_2-\frac{1}{n(n-1)}\sum_{i\ne j}\langle X_i,X_j\rangle_2\right] \nonumber\\
&{}={}&\frac{1}{n}\left[\frac{1}{n}\sum^n_{i=1}\Vert X_i\Vert^2_2-\frac{1}{n(n-1)}\sum^n_{i,j=1}\langle X_i,X_j\rangle_2+\frac{1}{n(n-1)}\sum^n_{i=1}\Vert X_i\Vert^2_2\right] \nonumber\\
% \end{eqnarray*}
% \begin{eqnarray*}
&{}={}&\frac{1}{n}\left[\frac{1}{n-1}\sum^n_{i=1}\Vert X_i\Vert^2_2-\frac{n}{n-1}\left\Vert\frac{1}{n}\sum^n_{i=1}X_i\right\Vert^2_2\right] \nonumber\\
&{}={}& \frac{1}{n}\left[\frac{1}{n-1}\sum^n_{i=1}\Vert X_i\Vert^2_2-\frac{n}{n-1}\Vert \bar{X}\Vert^2_2\right]=\frac{1}{n(n-1)}\sum^n_{i=1}\Vert X_i-\bar{X}\Vert^2_2=:\frac{S^2}{n},
%\label{Eq:delta}
\end{eqnarray*}
and
\begin{equation*}\hat{C}_{\tilde{\alpha}_{\textup{general}}}=:\check{\mu}=\frac{\Vert\bar{X}\Vert^2_2}{\frac{S^2}{n}+\Vert \bar{X}\Vert^2_2}\bar{X}=\left(1-\frac{\frac{S^2}{n}}{\frac{S^2}{n}+\Vert \bar{X}\Vert^2_2}\right)\bar{X}.
%\label{Eq:shrink-new}
\end{equation*}
The following result (proved in Section~\ref{subsec:gauss}) shows that the shrinkage estimator, $\check{\mu}$ has strictly smaller mean squared error compared to $\hat{\mu}$ when $d\ge 4+\frac{2}{n-1}$.  
\begin{theorem}\label{Thm:shrinkage}
Let $X_1,\ldots,X_n\stackrel{i.i.d.}{\sim} N_d(\mu,\sigma^2 I)$. For $n\ge2$ and $d\ge 4+\frac{2}{n-1}$, $$\mathbb{E}\norm{\check{\mu}-\mu}^2_2<\mathbb{E}\Vert \hat{\mu}-\mu\Vert^2_2$$
for all $\mu\in\mathbb{R}^d$ and $\sigma^2>0$.
\end{theorem}
When $n=2$, $\check{\mu}$ improves upon $\hat{\mu}$ for $d\ge 6$. For all $n\ge 3$, the improvement phenomenon occurs for $d\ge 5$. By slightly modifying the estimator $\check{\mu}$, the following result (proved in Section~\ref{Subsec:family}) shows improvement over $\hat{\mu}$ when $d\ge 3$.
% \begin{remark}
% ~\\
% \begin{enumerate}
%     \item Note that the estimator, $\hat{\mu}_{\tilde{\alpha}}$ does not require the knowledge of $\sigma^2$ but the analysis critically exploits the fact that the covariance matrix is $\sigma^2 I$
    
%     \item When $n=2$, $\hat{\mu}_{\tilde{\alpha}}$ improves upon $\hat{\mu}$ for $d\ge 6$. For all $n\ge 3$, the improvement phenomenon occurs for $d\ge 5$.
% \end{enumerate}
% \end{remark}
\begin{theorem}\label{Thm:family}
Let $X_1,\ldots,X_n\stackrel{i.i.d.}{\sim} N_d(\mu,\sigma^2 I)$. For $n\ge 2$, $c\in(0,2)$ and $d\ge \frac{4}{2-c}+\frac{2c}{(n-1)(2-c)}$, \begin{equation}\mathbb{E}\Vert\check{\mu}_c -\mu \Vert^2_2<\mathbb{E}\Vert\hat{\mu}-\mu\Vert^2_2 \label{Eq:improve} \end{equation}
for all $\mu\in\mathbb{R}^d$ and $\sigma^2>0$ where $\check{\mu}_c=(1-c\tilde{\alpha}_{\textup{general}})\hat{\mu}$ with $\tilde{\alpha}_{\textup{general}}=\frac{\frac{S^2}{n}}{\frac{S^2}{n}+\Vert \bar{X}\Vert^2_2}$. In particular,
if $c=\frac{2n-2}{3n-1}$, then (\ref{Eq:improve}) holds for all $d\ge 3$.
\end{theorem}
It is interesting to note that the estimator $\check{\mu}_c$ with $c=\frac{2n-2}{3n-1}$ behaves similar to that of the James-Stein estimator in showing improvement over $\hat{\mu}$ for $d\ge 3$ but with important differences. $\check{\mu}$ has an additional term of $\frac{S^2}{n}$ in the denominator and $c$ depends only on $n$ instead of $d$---James-Stein estimator has $c=d-2$. Because of this additional term in the denominator, establishing Theorem~\ref{Thm:family} is far more tedious than proving such a result for the James-Stein estimator. In fact, because of this additional term in the denominator, we are not able to establish concrete results in the non-spherical Gaussian scenario and it remains as an open question.
 \section{Proofs}\label{Sec:proofs}
%   \begin{lemma}
%     Let $X$ be zero mean real valued variable  and there exists a constants $c_1,c_2,d \geq 0$ and $c1 < c_2$ s.t $\forall \tau > 0$ following holds with atleast probability $1-h. \exp{(-\tau)}$

%         \[
%          \abs{X} \leq 
%         \begin{cases}
%           d\left(\frac{1+\tau}{n}\right)^{c_1} & 0 < \tau < n-1 \\
%           d\left( \frac{1+\tau}{n}\right)^{c_2} & \tau \geq n-1
%         \end{cases}
%         \]
    
%     then there exists constants $e_1\geq 0,e_2 \geq 0$ s.t
%     \begin{align*}
%         \mathbb{E}\left[\abs{X} \right] &\leq  \frac{e_1}{n^{c_1}}  + \frac{e_2}{n^{c_2}}  
%     \end{align*}
%     \hfill \solidqed
%     \end{lemma}
    % This is the theorem we will be keep reusing for proof of theorem 0.2, theorem 0.3, theorem 0.4.
The following is a master theorem, which we will repeatedly use to prove the results of Section~\ref{Sec:main results}.      
\begin{theorem}
\label{centralTheorem}
Let $\hat{C}$ and $\hat{\Delta}$ be unbiased estimators of $C$ and $\Delta$, respectively,  where $\Delta = \mathbb{E}\Vert\hat{C}-C\Vert_{\mathcal{H} }^2$. For $\tau >0$, suppose there exists positive constants $a,b,c_{1},c_{2},c_{3},d_{1},d_{2}$  that does not depend on $\tau$ and $n$ such that the following statements hold with probability at least $1- c_{3}e^{-\tau}$:
\begin{align}
\Vert\hat{C}- C\Vert_{\mathcal{H}} &\leq c_{1} \left( \frac{1+\tau}{n}\right)^{a/2} + c_{2} \left( \frac{1+\tau}{n}\right)^{(a+1)/2}, \label{Eq:C}\\
\abs{\hat{\Delta} - \Delta} &\leq  d_1 \left(\frac{1+\tau}{n}\right)^{b/2} + d_2  \left( \frac{1+\tau}{n}\right)^{(b+1)/2}. \nonumber %\label{Eq:Delta} 
\end{align}
Define $\alpha_*=\frac{\Delta}{\Delta+\Vert C-f^{*}\Vert^2_\mathcal{H}} $ and $    \hat{C}_{\tilde{\alpha}} = (1-\tilde{\alpha}) \hat{C} +  \tilde{\alpha} f^{*}$ as an estimator of $C$
where $\tilde{\alpha}  = \frac{\hat{\Delta}}{\hat{\Delta} + \Vert\hat{C} -f^{*}\Vert^2_{\mathcal{H}}}$. Then as $n\rightarrow\infty$, the following hold:
\begin{itemize}
\item[(i)] $\abs{\tilde{\alpha} - \alpha_{*}  } =  \mathcal{O}_{\mathbb{P}}\left(n^{- \min\{3a, b\}/2} \right)$;%\vspace{1.5mm}
\item[(ii)] $\abs{ \Vert\hat{C}_{\tilde{\alpha}}-C\Vert_\mathcal{H} - \Vert\hat{C}_{\alpha_*}-C\Vert_{\mathcal{H}}}=   \mathcal{O}_{\mathbb{P}}\left(n^{- \min\{3a, b\}/2} \right)$;%\vspace{1.5mm}
\item[(iii)] $\hat{C}_{\tilde{\alpha}}$ is a $n^{\min\{a,b\}/2}$- consistent estimator of $C$;%\vspace{1.5mm}
\item[(iv)] $\min_{\alpha} \mathbb{E} \Vert\hat{C}_{\alpha} - C\Vert^2_\mathcal{H} \leq  \mathbb{E} \Vert\hat{C}_{\tilde{\alpha}} - C\Vert^2_\mathcal{H} \leq \min_{\alpha} \mathbb{E} \Vert\hat{C}_{\alpha} - C\Vert^2_\mathcal{H} + \mathcal{O}(n^{-\min\{4a,(a+b),2b \}/2})$.
\end{itemize}
\end{theorem}
\begin{proof} Consider
\begin{align*}
    \alpha_{*}-\tilde{\alpha}  
    &= \frac{\Delta}{\Delta+\Vert C -f^{*}\Vert^2_\mathcal{H}}  -  \frac{\hat{\Delta}}{\hat{\Delta}+\Vert \hat{C}-f^{*}\Vert^2_\mathcal{H}}\\
    &=  \frac{\Delta \Vert \hat{C} -f^{*} \Vert_{\mathcal{H}}^2-\hat{\Delta} \Vert C- f^{*}\Vert^2_\mathcal{H}}{(\Delta+\Vert C- f^{*}\Vert^2_\mathcal{H})(\hat{\Delta}+\Vert \hat{C}-f^{*}\Vert^2_\mathcal{H})}\\
    &=  \frac{\Delta \left( \Vert\hat{C} -f^{*}\Vert_{\mathcal{H}}^2 - \Vert C -f^{*}\Vert^2_\mathcal{H}  \right)  + \norm{C-f^*}_{\mathcal{H}}^2 \left( \Delta - \hat{\Delta}  \right) }{(\Delta+\Vert C- f^{*}\Vert^2_\mathcal{H})(\hat{\Delta}+\Vert \hat{C}-f^{*}\Vert^2_\mathcal{H})} \\
    & =  \frac{\alpha_{*} \left( \Vert \hat{C}-f^{*} \Vert_{\mathcal{H}}^2 - \Vert C-f^{*}\Vert^2_{\mathcal{H}} \right) + (1-\alpha_{*}) \left(\Delta-\hat{\Delta} \right)  }{\hat{\Delta}+\Vert \hat{C} -f^{*}\Vert^2_\mathcal{H}}\\
%     \end{align*}
%     \begin{align*}
    &=    \frac{\alpha_{*} \left( \Vert\hat{C}-f^{*}\Vert_{\mathcal{H}}^2 - \Vert C-f^{*}\Vert_{\mathcal{H}}^2 \right) + (1-\alpha_{*}) \left(\Delta-\hat{\Delta}  \right)  }{\Delta + \norm{C-f^{*}}_{\mathcal{H}}^2 - \left( \norm{C-f^{*}}_{\mathcal{H}}^2 - \Vert\hat{C}-f^{*}\Vert_{\mathcal{H}}^2 \right) + \left(\hat{\Delta}- \Delta \right)  }
\end{align*}
from which we have
\begin{align}\label{Eq:AlphaTildeMinusAlphaStar}
    \abs{\tilde{\alpha}- \alpha_{*}} &\leq \frac{\alpha_{*} \left|\norm{C-f^{*}}_{\mathcal{H}}^2 - \Vert \hat{C} - f^{*}\Vert_{\mathcal{H}}^2 \right| + (1-\alpha_{*}) \left|\hat{\Delta}- \Delta \right|  }{\Delta + \norm{C-f^{*}}_{\mathcal{H}}^2 - \left| \norm{C-f^{*}}_{\mathcal{H}}^2 - \Vert \hat{C}-f^{*}\Vert_{\mathcal{H}}^2 \right| - \left|\hat{\Delta}- \Delta \right|  }
\end{align}
if \begin{equation}\Delta+\Vert C-f^*\Vert^2_\mathcal{H}>\left| \norm{C-f^{*}}_{\mathcal{H}}^2 - \Vert \hat{C}-f^{*}\Vert_{\mathcal{H}}^2 \right| + \left|\hat{\Delta}- \Delta \right|.\label{Eq:verify}\end{equation} 
%In the following, we consider $n\ge \tau+1$.
% Assuming that $n \geq 1 + \tau$, we will upper bound $\alpha_{*}$, $\left|\norm{C}_{\mathcal{H}}^2 - \Vert \hat{C}\Vert_{\mathcal{H}}^2 \right|$. 
\vspace{1mm}\\
(i) Consider 
        \begin{align}\label{Eq:SquaredDeivationC}
            \Vert\hat{C} - C\Vert^2_{\mathcal{H}} &\stackrel{(*)}{\leq} \left[ c_1 \left( \frac{1+\tau}{n}\right)^{a/2} + c_{2} \left( \frac{1+\tau}{n}\right)^{(a+1)/2} \right]^2 
            \stackrel{(**)}{\leq}  d \left( \frac{1+\tau}{n}\right)^{a},
        \end{align}
     for some constant $d$ that doesn't depend on $\tau,n$ and we used \eqref{Eq:C} in $(*)$ and assume $\frac{1+\tau}{n} \leq 1$ in $(**)$. Using Lemma~\ref{lem:split_to_exp_bound} for \eqref{Eq:SquaredDeivationC} yields $\Delta \leq e_1 n^{-a}$, which implies that,
     \begin{align}
         \alpha_{*} &= \frac{\Delta}{ \Delta + \norm{C-f^{*}}_{\mathcal{H}}^2} \leq \frac{\Delta}{\norm{C-f^{*}}_{\mathcal{H}}^2} \leq \frac{e_2}{n^{a}},  \label{Eq:alphastar}
     \end{align}
    for some positive constants $e_1, e_2$ that does not depend on $\tau$ and $n$. Next, $\abs{\Vert C-f^*\Vert_{\mathcal{H}}^2 -\Vert\hat{C}-f^*\Vert_{\mathcal{H}}^2}$ can be bounded as \begin{align*}
     &\abs{\Vert C -f^{*}\Vert_{\mathcal{H}}^2 -\Vert\hat{C} - f^{*}\Vert_{\mathcal{H}}^2} \leq   \Vert\hat{C} - C\Vert_{\mathcal{H}}^2 + 2 \norm{C-f^{*}}_{\mathcal{H}} \Vert\hat{C} - C\Vert_{\mathcal{H}} \\
%      \end{align*}
%      \begin{align*}
        &\stackrel{(*)}{\leq} d \left( \frac{1+\tau}{n}\right)^{a}  + 2 \norm{C-f^{*}}_{\mathcal{H}} \left(  c_{1} \left( \frac{1+\tau}{n}\right)^{a/2} + c_{2} \left( \frac{1+\tau}{n}\right)^{(a+1)/2}  \right) 
        \leq f \left(\frac{1+\tau}{n}\right)^{a/2} \numberthis \label{Eq:CSquareDIff},
     \end{align*}
for some positive constant $f$ that does not depend on $\tau$ and $n$, and we used \eqref{Eq:C} and \eqref{Eq:SquaredDeivationC} in $(*)$ along with the assumption that $n\ge \tau+1$. Also note that there exists a constant $g$ such that  
\begin{align*}
    \abs{\hat{\Delta} - \Delta} &\leq  d_1 \left(\frac{1+\tau}{n}\right)^{b/2} + d_2  \left( \frac{1+\tau}{n}\right)^{(b+1)/2} \leq g \left(\frac{1+\tau}{n}\right)^{b/2}.    \numberthis \label{Eq:DeltaDIff} 
\end{align*} 
If $n \geq \max \left\{1, \left( \frac{4f}{\norm{C-f^{*}}_{\mathcal{H}}^2}\right)^{2/a} , \left( \frac{4g}{\norm{C-f^{*}}_{\mathcal{H}}^2} \right)^{2/b} \right\} (1+\tau)$, the denominator in \eqref{Eq:AlphaTildeMinusAlphaStar} can be bounded as
\begin{align*}
    &\Delta + \norm{C-f^{*}}_{\mathcal{H}}^2 - \left| \norm{C-f^{*}}_{\mathcal{H}}^2 - \Vert \hat{C-f^{*}}\Vert_{\mathcal{H}}^2 \right| - \left|\hat{\Delta}- \Delta \right|  \\
    &\geq \norm{C-f^{*}}_{\mathcal{H}}^2 - f \left(\frac{1+\tau}{n}\right)^{\frac{a}{2}} - g \left(\frac{1+\tau}{n}\right)^{\frac{b}{2}}  \\
    &\geq \norm{C-f^{*}}_{\mathcal{H}}^2 - \frac{1}{4} \norm{C-f^{*}}_{\mathcal{H}}^2 - \frac{1}{4} \norm{C-f^{*}}_{\mathcal{H}}^2 \geq \frac{1}{2}\norm{C-f^{*}}_{\mathcal{H}}^2. \numberthis  \label{Eq:Denominator}
\end{align*}
Therefore, using \eqref{Eq:alphastar}--\eqref{Eq:Denominator}
% , \eqref{Eq:CSquareDIff},  \eqref{Eq:DeltaDIff}, and \eqref{Eq:Denominator} 
in \eqref{Eq:AlphaTildeMinusAlphaStar}, we obtain
\begin{align*}
    \abs{ \tilde{\alpha} - \alpha_{*} } &
    % \leq \norm{C}_{\mathcal{H}}^2 \left(\frac{e_2}{n^{a}} f \left(\frac{1+\tau}{n}\right)^{a/2}  + 1 g \left(\frac{1+\tau}{n}\right)^{b/2}  \right) 
     \leq h \left( \frac{1+\tau}{n}
    \right)^{\min\{3a,b\}/2 }, \numberthis \label{Eq:AlphaTildeMinusAlphaStarBound}
\end{align*}
where $h$ is a constant that does not depend on $\tau$ and $n$, thereby yielding the result.
\vspace{1mm}\\
(ii) We now bound $\abs{ \Vert\hat{C}_{\tilde{\alpha}}-C\Vert_{\mathcal{H}} - \Vert\hat{C}_{\alpha_{*}}-C\Vert_{\mathcal{H}} }$ as
    \begin{align*}
          &\abs{ \Vert\hat{C}_{\tilde{\alpha}}-C\Vert_{\mathcal{H}} - \Vert\hat{C}_{\alpha_{*}}-C\Vert_{\mathcal{H}}}\\ &\leq \Vert\hat{C}_{\alpha_{*}} - \hat{C}_{\tilde{\alpha}}  \Vert_{\mathcal{H}} \leq \abs{\tilde{\alpha} - \alpha_{*}}  \Vert\hat{C}-C\Vert_{\mathcal{H}} + \abs{\tilde{\alpha} - \alpha_{*}}\norm{C-f^{*}}_{\mathcal{H}}  \\
     &\leq\abs{\tilde{\alpha} - \alpha_{*}} \left[ \Vert\hat{C}-C\Vert_{\mathcal{H}} + \norm{C-f^{*}}_{\mathcal{H}} \right] \\
     &\stackrel{\eqref{Eq:AlphaTildeMinusAlphaStarBound}}{\leq}  h \left( \frac{1+\tau}{n} \right)^{\min\{3a,b\}/2 } \left[  c_{1} \left( \frac{1+\tau}{n}\right)^{a/2} + c_{2} \left( \frac{1+\tau}{n}\right)^{(a+1)/2} + \norm{C-f^*}_{\mathcal{H}} \right] \\
     &\leq p \left( \frac{1+\tau}{n} \right)^{\min\{3a,b\}/2 }, \numberthis \label{Eq:BigDeviation}
    \end{align*}
    where $p$ is constant that does not depend on $\tau$ and $n$ and the result follows.
    % so we have that $\abs{ \norm{\hat{C}_{\tilde{\alpha}}-C}_{\mathcal{H}} - \norm{\hat{C}_{\alpha_{*}}-C}_{\mathcal{H}} } = \mathcal{O}_{\mathbb{P}}(n^{- \frac{\min\{3a,b\}}{2} })$
\vspace{1mm}\\
(iii) $\Vert\hat{C}_{\alpha_{*}}-C \Vert_{\mathcal{H}}$ can be bounded as
\begin{align*}
               \Vert\hat{C}_{\alpha_*} - C\Vert_{\mathcal{H}} %&= \Vert\alpha_* f^{*} + (1 - \alpha_*) \hat{C}  -C \Vert_{\mathcal{H}} \\ 
            %   &= \Vert \alpha_{*}.0 + (1- \alpha_{*})\hat{C} - C  - \alpha_{*} C + \alpha_{*} C \Vert_{\mathcal{H}} \\
               &= \Vert(1-\alpha_{*}) (\hat{C}-C)+\alpha_*f^* - \alpha_{*} C\Vert_{\mathcal{H}}\\
                &\leq (1-\alpha_{*}) \Vert\hat{C} - C\Vert_{\mathcal{H}} + \alpha_{*} \norm{C-f^{*}}_{\mathcal{H}}\\ 
                &\leq  c_{1} \left( \frac{1+\tau}{n}\right)^{a/2} + c_{2} \left( \frac{1+\tau}{n}\right)^{(a+1)/2} +  \frac{e_2}{n^{a}}\Vert C-f^{*}\Vert_\mathcal{H} \leq q \left( \frac{1+\tau}{n} \right)^{a/2}, \numberthis \label{Eq:HatCAlphaStarMinusC}
    \end{align*}
    where $q$ is constant that does not depend on $\tau$ and $n$. The result follows from \eqref{Eq:BigDeviation} and \eqref{Eq:HatCAlphaStarMinusC} by noting that
   \begin{align*}
       \Vert \hat{C}_{\tilde{\alpha}} - C\Vert_{\mathcal{H}} &\leq  \Vert\hat{C}_{\alpha_{*}} -C\Vert_{\mathcal{H}} + \mathcal{O}_{\mathbb{P}} (n^{-\min\{3a,b \}/2}) \leq  \mathcal{O}_{\mathbb{P}} (n^{-a/2}) + \mathcal{O}_{\mathbb{P}} (n^{-\min\{3a,b \}/2})
   \end{align*}
   as $n \rightarrow \infty$.
\vspace{1mm}\\
(iv) We now bound $\Vert \hat{C}_{\tilde{\alpha}} - C\Vert_{\mathcal{H}}^2 - \Vert\hat{C}_{\alpha_{*}} -C\Vert_{\mathcal{H}}^2$ as
\begin{align*}
        &\Vert \hat{C}_{\tilde{\alpha}} - C\Vert_{\mathcal{H}}^2 - \Vert\hat{C}_{\alpha_{*}} -C\Vert_{\mathcal{H}}^2 \\
         &\leq \left(\Vert\hat{C}_{\tilde{\alpha}} - C\Vert_{\mathcal{H}}- \Vert\hat{C}_{\alpha_{*}} - C\Vert_{\mathcal{H}} \right)^2 +2 \Vert\hat{C}_{\alpha_{*}}-C\Vert_{\mathcal{H}} \left( \Vert\hat{C}_{\tilde{\alpha}} - C\Vert_{\mathcal{H}}- \Vert\hat{C}_{\alpha_{*}} - C\Vert_{\mathcal{H}} \right), \\
         &\leq  \left(p \left( \frac{1+\tau}{n} \right)^{\min\{3a,b\}/2 } \right)^2 +   2q \left( \frac{1+\tau}{n} \right)^{a/2} \left(p \left( \frac{1+\tau}{n} \right)^{\min\{3a,b\}/2 } \right) \\
         &\leq s \left( \frac{1+\tau}{n}\right)^{\min\{4a, a+b, 2b \}/2}, 
    \end{align*}
    where $s$ is a constant that does not depend on $\tau$ and $n.$ The result therefore follows by using Lemma~\ref{lem:split_to_exp_bound}. Finally, note that the assumptions on $n$ and the condition in \eqref{Eq:verify} hold as $n\rightarrow\infty$. 
    % we can conclude that 
    % \vspace{1mm}\\
    %     $\mathbb{E}\Vert \hat{C}_{\tilde{\alpha}} - C\Vert_{\mathcal{H}}^2 - \mathbb{E} \Vert \hat{C}_{\alpha_{*}} -C\Vert_{\mathcal{H}}^2 = \mathcal{O}(n^{-\min\{4a, a+b, 2b \}/2})  \text{ as } n \rightarrow \infty.$
    % \begin{align*}
    %     &\norm{ \hat{C}_{\tilde{\alpha}} - C}_{\mathcal{H}}^2 - \norm{\hat{C}_{\alpha_{*}} -C}_{\mathcal{H}}^2 \\
    %     &\leq \begin{cases}
    %     \left( \frac{g_1}{\gamma_n}\left( \frac{1+\tau}{n} \right)^{\frac{\min\{3a,b\}}{2}} \right)^2 + \left(  c_{1} \left(\frac{1+\tau}{n}\right)^{\frac{a}{2}} + \frac{f}{n^{a}} \norm{C}_{\mathcal{H}} \right) \left( \frac{g_1}{\gamma_n}\left( \frac{1+\tau}{n} \right)^{\frac{\min\{3a,b\}}{2}}  \right) &, 0  < \tau < n-1 \\
    %      \left( \frac{g_2}{\gamma_n} \left( \frac{1+\tau}{n} \right)^{\frac{\max\{4a,a+b\}+2}{2}}   \right)^2 \\
    %      \qquad\qquad + \left(  c_{2} \left(\frac{1+\tau}{n}\right)^{\frac{a+1}{2}} + \frac{f}{n^{a}} \norm{C}_{\mathcal{H}} \right) \left( \frac{g_2}{\gamma_n} \left( \frac{1+\tau}{n} \right)^{\frac{\max\{4a,a+b\}+2}{2}}   \right) &, \tau > n-1 \\
    %     \end{cases} \\
    %     &\leq \begin{cases}
    %       h \left( \frac{1}{\gamma_{n}^2} + \frac{1}{\gamma_{n}} \right)  \left( \frac{1+\tau}{n} \right)^{\min\{4a,a+b\}} &, 0 < \tau < n-1 \\
    %       h \left(\frac{1}{\gamma_{n}^2} + \frac{1}{\gamma_{n}} \right) \left( \frac{1+\tau}{n} \right)^{\frac{\max \{8a+4, 2a+2b+4, 3a+b+2 \}}{2}} &, \tau \geq n-1
    %     \end{cases}.
    % \end{align*}
\end{proof}

    \subsection{Proof of Theorem~\ref{Th:NonDegC-NonDegDelta}} \label{sec:proof-nondeg-nondeg}
%\begin{proof}
Note that 
\begin{align*}
        \Vert\hat{C}-C\Vert_{\mathcal{H}}&= \norm{\text{U}^{n}_k \Big[ r(X_{1},\dots,X_{k}) - \mathbb{E} \left( r(X_{1},\dots,X_{k})\right) \Big]   }_{\mathcal{H}}.
\end{align*}
Using Theorem~\ref{Th:BernUstats} on $r(X_{1},\dots,X_{k}) - \mathbb{E}( r(X_{1},\dots,X_{k})),$ we get that with probability at least $1-\exp(-\tau)$,
\begin{align} \label{eq:nondeg-nondeg-C}
 \Vert\hat{C}-C\Vert_{\mathcal{H}} & \leq 4  \beta \sqrt{k} \left(\frac{1+\tau}{n} \right)^{\frac{1}{2}} + 4  \theta k \left( \frac{1+\tau}{n} \right) = c_1 \left( \frac{1+\tau}{n} \right)^{\frac{1}{2}} + c_{2} \left( \frac{1+\tau}{n} \right),
\end{align}
 where $c_1,c_2 > 0$ are constants that do not depend on $\tau$ and $n$. Now consider 
 \begin{align*}
     &\abs{\hat{\Delta}_{\text{general}}- \Delta} \\&=  \Bigg| \sum_{i=1}^{k} \frac{\Mycomb[k]{i} \Mycomb[n-k]{k-i}}{\Mycomb[n]{k}}\left( \textup{U}_{2k-i}^{n}\left[ \kappa_{2k-i}(X_1,\dots, X_{2k-i}) \right]  -  \textup{U}_{2k}^{n}\left[ \kappa_{2k}(X_1,\dots, X_{2k}) \right]\right)  \\
     &\qquad- \sum_{i=1}^{k} \frac{\Mycomb[k]{i} \Mycomb[n-k]{k-i}}{\Mycomb[n]{k}} \left(\mathbb{E}\left[ \kappa_{2k-i}(X_1,\dots, X_{2k-i}) \right]  -  \mathbb{E}\left[ \kappa_{2k}(X_1,\dots, X_{2k}) \right] \right) \Bigg| \\
     &\stackrel{(*)}{\leq} \spadesuit,
%      \sum_{i=1}^{k} \frac{\Mycomb[k]{i} \Mycomb[n-k]{k-i}}{\Mycomb[n]{k}} \Bigg| \textup{U}_{2k-i}^{n}\Big[ \kappa_{2k-i}(X_1,\dots, X_{2k-i}) -  \mathbb{E}\left[ \kappa_{2k-i}(X_1,\dots, X_{2k-i}) \right] \Big]  \Bigg| \\
%      & \qquad  +  \left|\frac{\Mycomb[n-k]{k}}{\Mycomb[n]{k}} - 1\right| \Bigg| \textup{U}_{2k}^{n}\Big[ \kappa_{2k}(X_1,\dots, X_{2k}) -  \mathbb{E}\left[ \kappa_{2k}(X_1,\dots, X_{2k}) \right] \Big] \Bigg| \\
%      &=:\spadesuit,
\end{align*}
where we used Vandermonde's identity in $(*)$ and 
\begin{align*}
\spadesuit&:=\sum_{i=1}^{k} \frac{\Mycomb[k]{i} \Mycomb[n-k]{k-i}}{\Mycomb[n]{k}} \Bigg| \textup{U}_{2k-i}^{n}\Big[ \kappa_{2k-i}(X_1,\dots, X_{2k-i}) -  \mathbb{E}\left[ \kappa_{2k-i}(X_1,\dots, X_{2k-i}) \right] \Big]  \Bigg| \\
     & \qquad  +  \left|\frac{\Mycomb[n-k]{k}}{\Mycomb[n]{k}} - 1\right| \Bigg| \textup{U}_{2k}^{n}\Big[ \kappa_{2k}(X_1,\dots, X_{2k}) -  \mathbb{E}\left[ \kappa_{2k}(X_1,\dots, X_{2k}) \right] \Big] \Bigg|. 
\end{align*}
Now applying Theorem~\ref{Th:BernUstats} to $$\kappa_{2k-i}(X_1,\dots, X_{2k-i}) -  \mathbb{E}\left[ \kappa_{2k-i}(X_1,\dots, X_{2k-i}) \right]$$ for each $i \in \{0,1,\dots,k\}$, we have that with probability at least $1 - (k+1) \exp{(-\tau)}$,
\begin{align*}
\spadesuit &\leq \sum_{i=1}^{k}  \frac{\Mycomb[k]{i} \Mycomb[n-k]{k-i}}{\Mycomb[n]{k}}  \left[   4 \beta_{i} \sqrt{2k-i} \left(\frac{1+\tau}{n} \right)^{\frac{1}{2}} + 4  \theta_{i} (2k-i) \left( \frac{1+\tau}{n} \right) \right] \\
&\qquad + \left|\frac{\Mycomb[n-k]{k}}{\Mycomb[n]{k}} - 1\right|  \left[   4 \beta_{2k} \sqrt{2k} \left(\frac{1+\tau}{n} \right)^{\frac{1}{2}} + 4  \theta_{2k} (2k) \left( \frac{1+\tau}{n} \right) \right]\\
&\leq c_{3}\left[ \sqrt{2k} \left(\frac{1+\tau}{n} \right)^{\frac{1}{2}} +  2k \left( \frac{1+\tau}{n} \right) \right] \left[ \sum_{i=1}^{k}\frac{\Mycomb[k]{i} \Mycomb[n-k]{k-i}}{\Mycomb[n]{k}} + \left|\frac{\Mycomb[n-k]{k}}{\Mycomb[n]{k}} - 1\right|  \right]\\
% \end{align*}
% \begin{align*}
&\stackrel{(*)}{=} c_{3}\left[ \sqrt{2k} \left(\frac{1+\tau}{n} \right)^{\frac{1}{2}} +  2k \left( \frac{1+\tau}{n} \right) \right] \left[ 1 - \frac{\Mycomb[n-k]{k}}{\Mycomb[n]{k}}  + \left|\frac{\Mycomb[n-k]{k}}{\Mycomb[n]{k}} - 1\right|  \right] \\
&\stackrel{(**)}{=}   2c_{3}\left[ \sqrt{2k} \left(\frac{1+\tau}{n} \right)^{\frac{1}{2}} +  2k \left( \frac{1+\tau}{n} \right) \right] \left[ 1 - \frac{\Mycomb[n-k]{k}}{\Mycomb[n]{k}} \right] \\
&\stackrel{(\dagger)}{=}   2c_{3}\left[ \sqrt{2k} \left(\frac{1+\tau}{n} \right)^{\frac{1}{2}} +  2k \left( \frac{1+\tau}{n} \right) \right] \left[ \frac{n^{k}- (n-2k)^{k}}{n^k} \right]\\
&\stackrel{(\ddagger)}{\leq}  2c_{3}\left[ \sqrt{2k} \left(\frac{1+\tau}{n} \right)^{\frac{1}{2}} +  2k \left( \frac{1+\tau}{n} \right) \right] \left[\frac{2n^{k-1}k^2}{n^{k}} \right]\\ 
% \end{align*}
% \begin{align*}
&\leq c_4 \left( \frac{1+\tau}{n} \right)^{3/2}+ c_5 \left( \frac{1+\tau}{n} \right)^{2}, \numberthis \label{eq:nondeg-nondeg-var}
\end{align*}
where we used vandermonde's identity that $ \sum_{i=0}^{k} \frac{\Mycomb[k]{i} \Mycomb[n-k]{k-i}}{\Mycomb[n]{k}} = 1$ in $(*)$ and  $\frac{\Mycomb[n-k]{k}}{\Mycomb[n]{k}} < 1$ in $(**)$ , $\Mycomb[n]{k} \leq n^{k}$, and $\Mycomb[n-k]{k} \geq (n-2k)^{k}$ in $(\dagger)$. %and % numerator and $\Mycomb[n]{k} \geq n^{k}k^{-k}$ and 
In $(\ddagger)$, we used $0<b<a \implies a^k-b^k \leq  k a^{k-1} (a-b)$. 
%and in $(**)$ there exists constants $c_4$ and $c_5$ that does not depend on $\tau$ and $n$. 
Now applying Theorem~\ref{centralTheorem}  with $a =1$ (see \eqref{eq:nondeg-nondeg-C}) and $b = 3$ (see \eqref{eq:nondeg-nondeg-var}), the result follows.
%\end{proof}
    
\subsection{Proof of Theorem~\ref{Th:DegC-DegDelta} }\label{sec:proof-deg-deg}
%\begin{proof}
% \begin{align*}
% \Vert\hat{C}-C\Vert_{\mathcal{H}} 
% &= \norm{\text{U}_{k}^n \Big[r(X_{1},\dots,X_{k}) - \mathbb{E}\left( r(X_{1},\dots,X_{k}) \right) \Big]  }_{\mathcal{H}}. 
% \end{align*}
% \vspace{1mm}~\\
Using Theorem~\ref{Th:BernUstatsDeg} on 
$r(X_{1},\dots,X_{k}) - \mathbb{E}( r(X_{1},\dots,X_{k}))$ yields that with probability at least $1-\tilde{a} \exp{(-\tau)}$, 
\begin{align}\label{Eq:deg-deg-C}
\Vert\hat{C}-C\Vert_{\mathcal{H}} & \leq qk^k \left( \frac{\tau}{na'} \right)^{k/2} + Mk^k  \left( \frac{\tau}{na''} \right)^{(k+1)/2},
\end{align}
where $\tilde{a},a'$ and $a''$ are positive constants, and $q=(\theta+\sigma^2+\theta^2M^{-1})$ with $\sigma^2=\mathbb{E}\Vert r(X_1,\ldots,X_k)\Vert^2_\mathcal{H}-\Vert C\Vert^2_\mathcal{H}$. Therefore,
\begin{align*}
\abs{\hat{\Delta}_{\text{degen}} -\Delta} &= \Bigg| (\Mycomb[n]{k})^{-1} \Big[\textup{U}_{k}^{n}\big[ \kappa_{k}(X_1,\dots, X_{k})\big] -  \textup{U}_{2k}^{n}\big[ \kappa_{2k}(X_1,\dots, X_{2k}) \big] \Big] \\
& \qquad- (\Mycomb[n]{k})^{-1} \Big[\mathbb{E}\big[ \kappa_{k}(X_1,\dots, X_{k})\big] -  \mathbb{E}\big[ \kappa_{2k}(X_1,\dots, X_{2k}) \big] \Big]   \Bigg| \\
&\leq \Bigg| (\Mycomb[n]{k})^{-1} \textup{U}_{k}^{n} \Big[ \kappa_{k}(X_1,\dots, X_{k}) - \mathbb{E}\left[\kappa_{k}(X_1,\dots, X_{k})\right]\Big]  \Bigg|  \\
& \qquad +  \Bigg| (\Mycomb[n]{k})^{-1} \textup{U}_{2k}^{n} \Big[ \kappa_{2k}(X_1,\dots, X_{2k}) - \mathbb{E}\left[\kappa_{2k}(X_1,\dots, X_{2k})\right]\Big]  \Bigg| \\
&=:\spadesuit.
\end{align*}
Now, using Theorem~\ref{Th:BernUstats} for $$ \kappa_{k}(X_1,\dots, X_{k}) - \mathbb{E}\left[\kappa_{k}(X_1,\dots, X_{k})\right]\,\,\text{and}\,\, \kappa_{2k}(X_1,\dots, X_{2k}) - \mathbb{E}\left[\kappa_{2k}(X_1,\dots, X_{2k})\right],$$ 
we obtain that with probability at least $1-2e^{-\tau}$,
\begin{align*}
\spadesuit & \leq    (\Mycomb[n]{k})^{-1}  \left[  4 (\sigma_1+\sqrt{2}\sigma_2) \sqrt{k} \left(\frac{1+\tau}{n} \right)^{\frac{1}{2}} + 4  (\theta_1+2\theta_2) k \left( \frac{1+\tau}{n} \right)\right]\\
% +   4 \theta_2 \sqrt{2k} \left(\frac{1+\tau}{n} \right)^{\frac{1}{2}} + 8  M_2 k \left( \frac{1+\tau}{n} \right)    \right] \\
&\leq c_1 \left(\frac{1+\tau}{n}\right)^{\frac{2k+1}{2}} + c_{2} \left(\frac{1+\tau}{n}\right)^{\frac{2k+2}{2}} \numberthis, \label{Eq:deg-deg-var}
\end{align*}
where $c_{1}$ and $c_{2}$ are positive constants that do not depend on $\tau$ and $n$. Now applying Theorem~\ref{centralTheorem} with $a = k$ (see \eqref{Eq:deg-deg-C}) and $b = 2k+1$ (see \eqref{Eq:deg-deg-var}) and noting that $\min\{3a , b \} = \min\{3k , 2k+1\} = 2k+1$, $\min\{2b, (a+b), 4a \} = \min \{4k+2,3k+1, 4k\} = 3k+1$, yields the result.
%   \begin{align*}
%       \abs{\tilde{\alpha} - \alpha_{*}  } &=  \mathcal{O}_{\mathbb{P}}(n^{- \frac{2k+1}{2} }) \text{ as n } \rightarrow \infty \\
%             \abs{ \norm{\hat{C}_{\tilde{\alpha}}-C}_{\mathcal{H}} - \norm{\hat{C}_{\tilde{\alpha}}-C}_{\mathcal{H}}  }&=  \mathcal{O}_{\mathbb{P}}(n^{- \frac{2k+1}{2} }) \text{ as n } \rightarrow \infty \\
%             \min_{\alpha} \mathbb{E} \norm{\hat{C}_{\alpha} - C} \leq  \mathbb{E} \norm{\hat{C}_{\tilde{\alpha}} - C} &\leq \min_{\alpha} \mathbb{E} \norm{\hat{C}_{\alpha} - C} + \mathcal{O}(n^{-\frac{3k+1}{2}})  \text{ as n } \rightarrow \infty \\
%              \hat{C}_{\tilde{\alpha}} \text{ is } n^{k/2} &\text{ consistent estimator of }  C
%   \end{align*}
%    \end{proof}

  \subsection{Proof of Theorem \ref{Th:NonDegC-DegDelta} } \label{sec:proof-nondeg-deg}
 % \begin{proof}
% Note that   
% \begin{align*}
% \Vert\hat{C}-C\Vert_{\mathcal{H}} 
% &= \norm{\text{U}_{k}^n \Big[r(X_{1},\dots,X_{k}) - \mathbb{E}\left( r(X_{1},\dots,X_{k}) \right) \Big]  }_{\mathcal{H}}. 
% \end{align*}
Applying Theorem~\ref{Th:BernUstats} on $r(X_{1},\dots,X_{k}) - \mathbb{E}( r(X_{1},\dots,X_{k})),$ yields that with probability at least $1-\exp(-\tau)$,
 \begin{align*} %\label{eq:nondeg-deg-C}
             \Vert\hat{C}-C\Vert_{\mathcal{H}} & \leq 4\sigma \sqrt{k} \left(\frac{1+\tau}{n} \right)^{\frac{1}{2}} + 4  \theta k \left( \frac{1+\tau}{n} \right)\le c_1\left(\frac{1+\tau}{n} \right)^{\frac{1}{2}},
         \end{align*}
where the second inequality holds for $n\ge \tau+1$. Hence, it follows from
% In other words, there exists a positive constant $d>0$ not depending on $\tau$ and $n$ so that 
%              \[
%          \norm{\hat{C}-C}_{\mathcal{H}}^2 \leq 
%         \begin{cases}
%           d\left(\frac{1+\tau}{n}\right), & 0 < \tau < n-1 \\
%           d\left( \frac{1+\tau}{n}\right)^{2}, & \tau \geq n-1
%         \end{cases}.
%         \]
Lemma~\ref{lem:split_to_exp_bound} that
        \begin{align} \label{Eq:nondeg-deg-Var}
            \Delta &= \mathbb{E} \Vert\hat{C}-C\Vert_{\mathcal{H}}^2 \leq  \frac{e_1}{n},
        \end{align}
        for some positive constant $e_1$. 
Therefore,         
\begin{align*}
    &\abs{ \Delta -  \hat{\Delta}_{\text{degen}}}\\
    &= \Bigg|\Delta -  (\Mycomb[n]{k})^{-1}  \left[\textup{U}_{k}^{n}\big[ \kappa_{k}(X_1,\dots, X_{k})\big] -  \textup{U}_{2k}^{n}\big[ \kappa_{2k}(X_1,\dots, X_{2k})  \big]\right] \Bigg| \\
    &= \Bigg|\Delta - (\Mycomb[n]{k})^{-1} \sigma_{k}^2 + (\Mycomb[n]{k})^{-1} \sigma_{k}^2\\
    &\qquad\qquad-   (\Mycomb[n]{k})^{-1}  \left[\textup{U}_{k}^{n}\big[ \kappa_{k}(X_1,\dots, X_{k})\big] -  \textup{U}_{2k}^{n}\big[ \kappa_{2k}(X_1,\dots, X_{2k})\big]  \right] \Bigg| \\
%     \end{align*}
% \begin{align*}
    &\leq \Bigg|  \Delta - (\Mycomb[n]{k})^{-1} \sigma_{k}^2\Bigg| + (\Mycomb[n]{k})^{-1} \Bigg| \textup{U}_{k}^{n}\Big[ \kappa_{k}(X_1,\dots, X_{k})- \mathbb{E}[\kappa_{k}(X_1,\dots, X_{k})] \Big]  \Bigg|  \\
%     \end{align*}
% \begin{align*}
    &\qquad +  (\Mycomb[n]{k})^{-1} \Bigg| \textup{U}_{2k}^{n}\Big[ \kappa_{2k}(X_1,\dots, X_{2k})- \mathbb{E}[\kappa_{2k}(X_1,\dots, X_{2k})] \Big]  \Bigg|  \\
    &=:\spadesuit.
\end{align*}
Now, using Theorem~\ref{Th:BernUstats} on $$ \kappa_{k}(X_1,\dots, X_{k}) - \mathbb{E}\left[\kappa_{k}(X_1,\dots, X_{k})\right],\,\,\text{and}\,\, \kappa_{2k}(X_1,\dots, X_{2k}) - \mathbb{E}\left[\kappa_{2k}(X_1,\dots, X_{2k})\right],$$ 
we obtain that with probability at least $1-2e^{-\tau}$,
\begin{align*}
\spadesuit &\leq \Big| \Delta- (\Mycomb[n]{k})^{-1} \sigma_{k}^2 \Big| +  (\Mycomb[n]{k})^{-1}  \left[  4 \sigma_1 \sqrt{k} \left(\frac{1+\tau}{n} \right)^{\frac{1}{2}} + 4  \theta_1 k \left( \frac{1+\tau}{n} \right)    \right] \\
 & \qquad+ (\Mycomb[n]{k})^{-1}  \left[  4 \sigma_2 \sqrt{2k} \left(\frac{1+\tau}{n} \right)^{\frac{1}{2}} + 8  \theta_2 k \left( \frac{1+\tau}{n} \right)    \right] \\
 &\leq  \Big| \Delta- (\Mycomb[n]{k})^{-1} \sigma_{k}^2 \Big|+c'(\Mycomb[n]{k})^{-1}\left(\frac{1+\tau}{n}\right)^{1/2}\\
%   \end{align*}
%  \begin{align*}
 &\leq  \Big| \Delta-  (\Mycomb[n]{k})^{-1}\sigma_{k}^2 \Big|+c'\left(\frac{k}{n}\right)^k\left(\frac{1+\tau}{n}\right)^{1/2}\\
%  \end{align*}
% \begin{align*}
  &\leq \begin{cases}
  c'\left(\frac{1}{n}\right)\left(\frac{1+\tau}{n}\right)^{1/2}, & k = 1\\ 
 \max {{\{ \Delta, (\Mycomb[n]{k})^{-1} \sigma_{k}^2   \}}}+c'\left(\frac{k}{n}\right)^k\left(\frac{1+\tau}{n}\right)^{1/2},&  k > 1
 \end{cases} 
  \leq \begin{cases}
  c'\left(\frac{1+\tau}{n}\right)^{3/2}, & k = 1\\ 
  c''\left(\frac{1+\tau}{n}\right),&  k > 1
 \end{cases},
\end{align*}
where $c',c''>0$ are constants that do not depend on $\tau$ and $n$, and we used \eqref{Eq:nondeg-deg-Var} in the above inequality. Now applying Theorem~\ref{centralTheorem} with $a=1$ and $b=2$ (for $k>1$) yields the result.

   %\end{proof}
   \subsection{Proof of Theorem~\ref{Thm:shrinkage}}\label{subsec:gauss}
%   \begin{proof}
Define $\hat{\alpha}:=\frac{\frac{S^2}{n}}{\frac{S^2}{n}+\Vert \bar{X}\Vert^2_2}$ so that $\check{\mu}=(1-\hat{\alpha})\bar{X}$. Define $W:=\bar{X}\sim N(\mu,\frac{\sigma^2}{n}I)$ and $U:=\frac{S^2}{n}$.
Consider
\begin{eqnarray}
\mathbb{E}\Vert \hat{\mu}-\mu\Vert^2_2-\mathbb{E}\Vert\check{\mu}-\mu\Vert^2_2&{}={}&\mathbb{E}\left[\Vert \bar{X}-\mu\Vert^2_2-\Vert (\bar{X}-\mu)-\hat{\alpha}\bar{X}\Vert^2_2\right]\nonumber\\
&{}={}&\mathbb{E}\left[2\hat{\alpha}\left\langle \bar{X}-\mu,\bar{X}\right\rangle_2-\Vert \hat{\alpha}\bar{X}\Vert^2_2\right]\nonumber\\
&{}={}& \mathbb{E}\left[2\left\langle W-\mu,\frac{UW}{U+\Vert W\Vert^2_2}\right\rangle_2-\left\Vert\frac{UW}{U+\Vert W\Vert^2_2}\right\Vert^2_2\right].\label{Eq:diff}
\end{eqnarray}
Note that
$$\left\langle W-\mu,\frac{UW}{U+\Vert W\Vert^2_2}\right\rangle_2=\sum^d_{i=1}(W_i-\mu_i)\left(\frac{UW_i}{U+\sum_i W^2_i}\right)$$
where $W_i\sim N(\mu_i,\sigma^2/n)$. By partial integration, we have
\begin{align*}\mathbb{E}\left[(W_i-\mu_i)\left(\frac{UW_i}{U+\sum_i W^2_i}\right)\right]&=\frac{\sigma^2}{n}\mathbb{E}\left[\frac{d}{dW_i}\left(\frac{UW_i}{U+\sum_i W^2_i}\right)\right]\nonumber\\
&=\frac{\sigma^2}{n}\mathbb{E}\left[\frac{U}{U+\Vert W\Vert^2_2}-\frac{2UW^2_i}{(U+\Vert W\Vert^2_2)^2}\right]
%\label{Eq:partialint}
\end{align*}
and therefore
\begin{equation}
\mathbb{E}\left[(W-\mu)^T\left(\frac{UW}{U+\Vert W\Vert^2_2}\right)\right]=\frac{\sigma^2}{n}\mathbb{E}\left[\frac{dU}{U+\Vert W\Vert^2_2}-\frac{2U\Vert W\Vert^2_2}{(U+\Vert W\Vert^2_2)^2}\right].
\label{Eq:partial}
\end{equation}
Using (\ref{Eq:partial}) in (\ref{Eq:diff}), we have
\begin{equation}
\mathbb{E}\Vert \hat{\mu}-\mu\Vert^2_2-\mathbb{E}\Vert\check{\mu}-\mu\Vert^2_2=\mathbb{E}\left[\frac{2d\sigma^2 U}{n(U+\Vert W\Vert^2_2)}-\frac{4\sigma^2 U\Vert W\Vert^2_2}{n(U+\Vert W\Vert^2_2)^2}-\frac{U^2\Vert W\Vert^2_2}{(U+\Vert W\Vert^2_2)^2}\right].
\label{Eq:step1}
\end{equation}
Note that $\Vert W\Vert^2_2\sim \frac{\sigma^2}{n}\chi^2_d(\lambda)$ where $\lambda=\frac{n\Vert\mu\Vert^2_2}{\sigma^2}$ with 
$\chi^2_d(\lambda)$ denoting a non-central $\chi^2$ distribution with $d$ degrees of freedom and $\lambda$ being the non-centrality parameter. Also note that $\frac{(n-1)S^2}{\sigma^2}\sim\chi^2_{(n-1)d}$ with
$S^2$ being independent of $W$. Define $Z:=\frac{n\Vert W\Vert^2_2}{\sigma^2}\sim\chi^2_d(\lambda)$ and $Y:=\frac{n(n-1)U}{\sigma^2}\sim \chi^2_{(n-1)d}$ where $Y$ and $Z$ are independent. Then (\ref{Eq:step1})
reduces to
\begin{eqnarray}
\mathbb{E}\Vert \hat{\mu}-\mu\Vert^2_2-\mathbb{E}\Vert\check{\mu}-\mu\Vert^2_2&{}={}&\frac{\sigma^2}{n}\mathbb{E}\left[\frac{2d Y}{Y+(n-1)Z}-\frac{4(n-1) YZ}{(Y+(n-1)Z)^2}-\frac{ Y^2Z}{(Y+(n-1)Z)^2}\right]\nonumber\\
&{}={}&\frac{\sigma^2}{n}\mathbb{E}\left[\frac{(2d-Z)Y^2+YZ(n-1)(2d-4)}{(Y+(n-1)Z)^2}\right].\label{Eq:step2}
%&{}={}&\frac{\sigma^2}{n}\int^\infty_0 t 
\end{eqnarray}
Using the fact that $\int^\infty_0 te^{-at}\,dt=\frac{1}{a^2}$ for $a>0$ and employing Fubini's theorem, we have
\begin{equation}
\mathbb{E}\left[\frac{YZ}{(Y+(n-1)Z)^2}\right]= \int^\infty_0 t\mathbb{E}\left[Ye^{-tY}\right]\mathbb{E}\left[Ze^{-t(n-1)Z}\right]\,dt\label{Eq:fubini-1}
\end{equation}
and 
\begin{equation}
\mathbb{E}\left[\frac{Y^2(2d-Z)}{(Y+(n-1)Z)^2}\right]= \int^\infty_0 t\mathbb{E}\left[Y^2e^{-tY}\right]\mathbb{E}\left[(2d-Z)e^{-t(n-1)Z}\right]\,dt.\label{Eq:fubini-2} 
\end{equation}
To compute the above expectations, we require the following: for any $t>0$,
\begin{itemize}
 \item $\mathbb{E}\left[e^{-tY}\right]= (1+2t)^{-\frac{(n-1)d}{2}},$\vspace{1.5mm}
 \item $\mathbb{E}\left[Ye^{-tY}\right]=-\frac{d}{dt}\mathbb{E}\left[e^{-tY}\right]=\frac{(n-1)d}{1+2t}E\left[e^{-tY}\right]$,\vspace{1.5mm}
 \item \begin{align*}\mathbb{E}\left[Y^2e^{-tY}\right]=\frac{d^2}{dt^2}\mathbb{E}\left[e^{-tY}\right]&=\frac{(n-1)d}{1+2t}\left(\mathbb{E}\left[Ye^{-tY}\right]+\frac{2\mathbb{E}\left[e^{-tY}\right]}{1+2t}\right)\\
 &=\frac{(n-1)^2d^2+2(n-1)d}{(1+2t)^2}\mathbb{E}\left[e^{-tY}\right],\end{align*}
 \item $\mathbb{E}\left[e^{-t(n-1)Z}\right]=(1+2(n-1)t)^{-\frac{d}{2}}\exp\left(-\frac{\lambda(n-1)t}{1+2(n-1)t}\right),$\vspace{1.5mm}
 \item \begin{align*}\mathbb{E}\left[Ze^{-t(n-1)Z}\right]&=-\frac{1}{n-1}\frac{d}{dt}\mathbb{E}\left[e^{-t(n-1)Z}\right]\\&=\left(\frac{d}{1+2(n-1)t}+\frac{\lambda}{(1+2(n-1)t)^2}\right)\mathbb{E}\left[e^{-t(n-1)Z}\right].\end{align*}
\end{itemize}
Therefore, (\ref{Eq:fubini-1}) and (\ref{Eq:fubini-2}) reduce to
\begin{eqnarray}
&&\mathbb{E}\left[\frac{YZ}{(Y+(n-1)Z)^2}\right]\nonumber\\
&{}={}& \int^\infty_0 t\mathbb{E}\left[Ye^{-tY}\right]\mathbb{E}\left[Ze^{-t(n-1)Z}\right]\,dt\nonumber\\
&{}={}&(n-1)d\int^\infty_0 \frac{t}{1+2t}\left(\frac{d}{1+2(n-1)t}+\frac{\lambda}{(1+2(n-1)t)^2}\right)\mathbb{E}\left[e^{-tY}\right]\mathbb{E}\left[e^{-t(n-1)Z}\right]\,dt\nonumber\\
% \end{eqnarray}
% \begin{eqnarray}
&{}={}&d\int^\infty_0 \frac{a}{n-1+2a}\left(\frac{d}{1+2a}+\frac{\lambda}{(1+2a)^2}\right)\mathbb{E}\left[e^{-\frac{aY}{n-1}}\right]\mathbb{E}\left[e^{-aZ}\right]\,da\label{Eq:term2}
\end{eqnarray}
and 
\begin{eqnarray}
&{}{}&\mathbb{E}\left[\frac{Y^2(2d-Z)}{(Y+(n-1)Z)^2}\right]
= \int^\infty_0 t\mathbb{E}\left[Y^2e^{-tY}\right]\mathbb{E}\left[(2d-Z)e^{-t(n-1)Z}\right]\,dt\nonumber\\
&{}={}&d(n-1)((n-1)d+2)\int^\infty_0 \frac{t}{(1+2t)^2}\left(2d-\frac{d}{1+2(n-1)t}-\frac{\lambda}{(1+2(n-1)t)^2}\right)\nonumber\\
&{}{}&\qquad\qquad\qquad\qquad\qquad\qquad\qquad\times\mathbb{E}\left[e^{-tY}\right]\mathbb{E}\left[e^{-t(n-1)Z}\right]\,dt\nonumber\\
&{}={}&d(n-1)((n-1)d+2)\int^\infty_0 \frac{a}{(n-1+2a)^2}\left(2d-\frac{d}{1+2a}-\frac{\lambda}{(1+2a)^2}\right)\nonumber\\
&{}{}&\qquad\qquad\qquad\qquad\qquad\qquad\qquad\times\mathbb{E}\left[e^{-\frac{aY}{n-1}}\right]\mathbb{E}\left[e^{-aZ}\right]\,da.\label{Eq:term1}
\end{eqnarray}
% \begin{eqnarray}
% \textcolor{magenta}{\mathbb{E}\left[\frac{Y}{Y+(n-1)Z}\right]}&{}={}&\int^\infty_0 \mathbb{E}\left[Ye^{-tY}\right]\mathbb{E}\left[e^{-t(n-1)Z}\right]\,dt\nonumber\\
% &{}={}&(n-1)d\int^\infty_0 \frac{1}{1+2t}\mathbb{E}\left[e^{-tY}\right]\mathbb{E}\left[e^{-t(n-1)Z}\right]\,dt\nonumber\\
% &{}={}&d\int^\infty_0\frac{n-1}{2a+n-1}\mathbb{E}\left[e^{-\frac{aY}{n-1}}\right]\mathbb{E}\left[e^{-aZ}\right]\,da.\nonumber
% %&{}={}&\frac{d}{2}\int^1_0(1-b)^{\frac{nd}{2}-1}\left(1+\frac{2-n}{n-1}b\right)^{-1-\frac{(n-1)d}{2}}\exp\left(-\frac{b\lambda}{2}\right)\,db.\nonumber
% \end{eqnarray}
Using (\ref{Eq:term2}) and (\ref{Eq:term1}) in (\ref{Eq:step2}), we obtain
\begin{eqnarray}
\mathbb{E}\Vert \hat{\mu}-\mu\Vert^2_2-\mathbb{E}\Vert\check{\mu}-\mu\Vert^2_2
&{}={}&\frac{d(n-1)\sigma^2}{n}\Bigg[\int^\infty_0 \Bigg(\frac{((n-1)d+2)a}{(2a+n-1)^2}\left(2d-\frac{d}{1+2a}-\frac{\lambda}{(1+2a)^2}\right)\nonumber\\
&{}{}&\qquad+\frac{(2d-4)a}{2a+n-1}\left(\frac{d}{1+2a}+\frac{\lambda}{(1+2a)^2}\right)\Bigg)\mathbb{E}\left[e^{-\frac{aY}{n-1}}\right]\mathbb{E}\left[e^{-aZ}\right]\,da\Bigg]\nonumber\\
&{}={}&\int^\infty_0\frac{d(n-1)\sigma^2a}{n(2a+n-1)^2}\mathcal{B}(a,\lambda)(1+2a)^{-\frac{d}{2}-2}\mathbb{E}\left[e^{-\frac{aY}{n-1}}\right]\,da,\nonumber
\end{eqnarray}
with
\begin{eqnarray}
\mathcal{B}(a,\lambda)
&{}:={}& \Big((nd-d+2)\left(2d(1+2a)^2-d(1+2a)-\lambda\right)+(2d-4)(2a+n-1)(d+2ad+\lambda)\Big)\nonumber\\
&{}{}&\qquad\qquad\times e^{-\frac{a\lambda}{1+2a}}\nonumber\\
&{}={}&\Big((nd-d+2)(d(8a^2+6a+1)-\lambda)+(2d-4)(2a+n-1)(d+2ad+\lambda)\Big)e^{-\frac{a\lambda}{1+2a}}\nonumber\\
&{}=:{}&(\theta_1+\theta_2\lambda)e^{-\frac{a\lambda}{1+2a}},\nonumber
\end{eqnarray}
where for all $a\in [0,\infty)$, $$\theta_1:=d(nd-d+2)(8a^2+6a+1)+d(2d-4)(2a+n-1)(1+2a)>0\,\,\text{for}\,\,d\ge 2,$$ and
$$\theta_2:=(2d-4)(2a+n-1)-(n-1)d-2=4a(d-2)+(n-1)(d-4)-2\ge 0$$ if $d\ge \sup_{a}\frac{8a+2+4(n-1)}{4a+n-1}=4+\frac{2}{n-1}$.
%\,\,\text{for}\,\,d\ge 4+\frac{2}{n-1},\,n\ge 2$$ and all $a\in [0,\infty)$.
%and %$$\theta_3:=\frac{a}{1+2a}\ge 0\,\,\text{for all}\,\,a\in [0,\infty).$$
This means for $d\ge 4+\frac{2}{n-1},\,n\ge 2$, $\mathcal{B}(a,\lambda)>0$ for all $\lambda$ and $a\in[0,\infty)$ and the result follows.
% so $\mathbb{E}\Vert\hat{\mu}_{\tilde{\alpha}}-\mu\Vert^2_H<\mathbb{E}\Vert \hat{\mu}-\mu\Vert^2_H$ for all
% $\mu\in\mathbb{R}^d$ and $\sigma^2>0$.
%\end{proof}
\subsection{Proof of Theorem~\ref{Thm:family}}\label{Subsec:family}
%    \begin{proof}
Proceeding as in the proof of Theorem~\ref{Thm:shrinkage}, we obtain
\begin{eqnarray}
\mathbb{E}\Vert \hat{\mu}-\mu\Vert^2_2-\mathbb{E}\Vert\check{\mu}_c-\mu\Vert^2_2
&{}={}&\frac{c\sigma^2}{n}\mathbb{E}\left[\frac{(2d-cZ)Y^2+YZ(n-1)(2d-4)}{(Y+(n-1)Z)^2}\right]\nonumber\\
&{}={}&\int^\infty_0\frac{dc(n-1)\sigma^2a}{n(2a+n-1)^2}\mathcal{A}(a,\lambda)(1+2a)^{-\frac{d}{2}-2}\mathbb{E}\left[e^{-\frac{aY}{n-1}}\right]\,da,\nonumber
%&{}={}&\frac{\sigma^2}{n}\int^\infty_0 t 
\end{eqnarray}
where \begin{eqnarray}
&{}{}&\mathcal{A}(a,\lambda)\nonumber\\
&{}:={}& \Big((nd-d+2)\left(2d(1+2a)^2-cd(1+2a)-c\lambda\right)+(2d-4)(2a+n-1)(d+2ad+\lambda)\Big)e^{-\frac{a\lambda}{1+2a}}\nonumber\\
% &{}{}&\qquad\qquad\times e^{-\frac{a\lambda}{1+2a}}\nonumber\\
&{}={}&\Big((nd-d+2)(d(8a^2+8a-2ac+2-c)-c\lambda)+(2d-4)(2a+n-1)(d+2ad+\lambda)\Big)e^{-\frac{a\lambda}{1+2a}}\nonumber\\
%&{}{}&\qquad\qquad\times e^{-\frac{a\lambda}{1+2a}}\nonumber\\
&{}=:{}&(\theta_3+\theta_4\lambda)e^{-\frac{a\lambda}{1+2a}}\nonumber
\end{eqnarray}
with $$\theta_3:=d(nd-d+2)(8a^2+8a-2ac+2-c)+d(2d-4)(2a+n-1)(1+2a)>0$$ for $d\ge 2$, $c\in[0,2)$ and all $a\in [0,\infty),$ and
$$\theta_4:=(2d-4)(2a+n-1)-(n-1)dc-2c=4a(d-2)+(n-1)(2d-4-dc)-2c\ge 0$$ for $d\ge \frac{4}{2-c}+\frac{2c}{(n-1)(2-c)}$, $n\ge 2$, $c\in(0,2)$ and all $a\in [0,\infty)$. This means, for the choice of $n$, $c$ and $d$ in the statement of Theorem~\ref{Thm:family}, the result follows.
%\end{proof}
% % \appendix
% % \numberwithin{equation}{section}
% % \makeatletter 

% "activate" the preparatory code, but for section-level headers only
%\newcommand{\section@cntformat}{\thesection  \,\,\ }
%\makeatother
\section*{Acknowledgements}
BKS is partially supported by NSF
CAREER Award DMS-1945396. BKS thanks Donald Richards for helpful comments on the proof of Theorem~\ref{Thm:shrinkage}.

\bibliography{bibolography.bib}

\begin{thebibliography}{}

\bibitem[Arcones and Gin{\'e}, 1993]{arcones1993limit}
Arcones, M.~A. and Gin{\'e}, E. (1993).
\newblock Limit theorems for {U}-processes.
\newblock {\em The Annals of Probability}, 21(3):1494--1542.

\bibitem[Aronszajn, 1950]{Aronszajn-50}
Aronszajn, N. (1950).
\newblock Theory of reproducing kernels.
\newblock {\em Trans. Amer. Math. Soc.}, 68:337--404.

\bibitem[Balasubramanian et~al., 2021]{JMLR:v22:17-570}
Balasubramanian, K., Li, T., and Yuan, M. (2021).
\newblock On the optimality of kernel embedding based goodness-of-fit tests.
\newblock {\em Journal of Machine Learning Research}, 22(1):1--45.

\bibitem[Brandwein and Strawderman, 1990]{brandwein1990stein}
Brandwein, A.~C. and Strawderman, W.~E. (1990).
\newblock Stein estimation: The spherically symmetric case.
\newblock {\em Statistical Science}, 5(3):356--369.

\bibitem[Brandwein and Strawderman, 2012]{brandwein2012stein}
Brandwein, A.~C. and Strawderman, W.~E. (2012).
\newblock Stein estimation for spherically symmetric distributions: {R}ecent
  developments.
\newblock {\em Statistical Science}, 27(1):11--23.

\bibitem[Chen et~al., 2010]{IEEESP:Chen-2010}
Chen, Y., Wiesel, A., Eldar, Y.~C., and Hero, A.~O. (2010).
\newblock Shrinkage algorithms for {MMSE} covariance estimation.
\newblock {\em IEEE Transactions on Signal Processing}, 58(10):5016--5029.

\bibitem[De~la Pena and Gin{\'e}, 2012]{de2012decoupling}
De~la Pena, V. and Gin{\'e}, E. (2012).
\newblock {\em Decoupling: From Dependence to Independence}.
\newblock Springer Science \& Business Media.

\bibitem[Dinculeanu, 2000]{dinculeanu2000vector}
Dinculeanu, N. (2000).
\newblock {\em Vector Integration and Stochastic Integration in Banach Spaces},
  volume~48.
\newblock John Wiley \& Sons.

\bibitem[Fisher and Sun, 2011]{fisher2011improved}
Fisher, T.~J. and Sun, X. (2011).
\newblock Improved {S}tein-type shrinkage estimators for the high-dimensional
  multivariate normal covariance matrix.
\newblock {\em Computational Statistics \& Data Analysis}, 55(5):1909--1918.

\bibitem[Fukumizu et~al., 2004]{fukumizu2004dimensionality}
Fukumizu, K., Bach, F.~R., and Jordan, M.~I. (2004).
\newblock Dimensionality reduction for supervised learning with reproducing
  kernel {H}ilbert spaces.
\newblock {\em Journal of Machine Learning Research}, 5(Jan):73--99.

\bibitem[Gretton et~al., 2012]{JMLR:v13:gretton12a}
Gretton, A., Borgwardt, K.~M., Rasch, M.~J., Sch{{\"o}}lkopf, B., and Smola, A.
  (2012).
\newblock A kernel two-sample test.
\newblock {\em Journal of Machine Learning Research}, 13(25):723--773.

\bibitem[Gretton et~al., 2007]{NIPS2007_d5cfead9}
Gretton, A., Fukumizu, K., Teo, C., Song, L., Sch\"{o}lkopf, B., and Smola, A.
  (2007).
\newblock A kernel statistical test of independence.
\newblock In Platt, J., Koller, D., Singer, Y., and Roweis, S., editors, {\em
  Advances in Neural Information Processing Systems}, volume~20. Curran
  Associates, Inc.

\bibitem[James and Stein, 1961]{james1992estimation}
James, W. and Stein, C. (1961).
\newblock Estimation with quadratic loss.
\newblock In {\em Proceedings of the Fourth Berkeley Symposium on Mathematical
  Statistics and Probability, Volume 1: Contributions to the Theory of
  Statistics}, pages 361--379.

\bibitem[Joly and Lugosi, 2016]{joly2016robust}
Joly, E. and Lugosi, G. (2016).
\newblock Robust estimation of {U}-statistics.
\newblock {\em Stochastic Processes and their Applications},
  126(12):3760--3773.

\bibitem[Ledoit and Wolf, 2004]{ledoit2004well}
Ledoit, O. and Wolf, M. (2004).
\newblock A well-conditioned estimator for large-dimensional covariance
  matrices.
\newblock {\em Journal of Multivariate Analysis}, 88(2):365--411.

\bibitem[Ledoit and Wolf, 2018]{Ledoit-Bernoulli-18}
Ledoit, O. and Wolf, M. (2018).
\newblock Optimal estimation of a large-dimensional covariance matrix under
  {S}tein's loss.
\newblock {\em Bernoulli}, 24(4B):3791--3832.

\bibitem[Lee, 2019]{lee2019u}
Lee, A.~J. (2019).
\newblock {\em U-statistics: Theory and Practice}.
\newblock Routledge.

\bibitem[Muandet et~al., 2016]{muandet2016kernel}
Muandet, K., Sriperumbudur, B., Fukumizu, K., Gretton, A., and Sch{\"o}lkopf,
  B. (2016).
\newblock Kernel mean shrinkage estimators.
\newblock {\em Journal of Machine Learning Research}, 17(48):1--41.

\bibitem[Song et~al., 2012]{JMLR:v13:song12a}
Song, L., Smola, A., Gretton, A., Bedo, J., and Borgwardt, K. (2012).
\newblock Feature selection via dependence maximization.
\newblock {\em Journal of Machine Learning Research}, 13(47):1393--1434.

\bibitem[Stein, 1975]{Stein-cov-75}
Stein, C. (1975).
\newblock Estimation of a covariance matrix.
\newblock Rietz Lecture, 39th Annual Meeting, Atlanta, GA.

\bibitem[Stein, 1956]{Stein1956InadmissibilityOT}
Stein, C.~M. (1956).
\newblock Inadmissibility of the usual estimator for the mean of a multivariate
  normal distribution.
\newblock In {\em Proceedings of the Third Berkeley Symposium on Mathematical
  Statistics and Probability, Volume 1: Contributions to the Theory of
  Statistics}, pages 197--206.

\bibitem[Touloumis, 2015]{touloumis2015nonparametric}
Touloumis, A. (2015).
\newblock Nonparametric {S}tein-type shrinkage covariance matrix estimators in
  high-dimensional settings.
\newblock {\em Computational Statistics \& Data Analysis}, 83:251--261.

\bibitem[Yurinsky, 2006]{yurinsky2006sums}
Yurinsky, V. (2006).
\newblock {\em Sums and Gaussian Vectors}.
\newblock Springer.

\bibitem[Zhou et~al., 2019]{zhou2019class}
Zhou, Y., Chen, D.-R., and Huang, W. (2019).
\newblock A class of optimal estimators for the covariance operator in
  reproducing kernel {H}ilbert spaces.
\newblock {\em Journal of Multivariate Analysis}, 169:166--178.

\end{thebibliography}
\bibliographystyle{apalike}
\appendix
%\numberwithin{equation}{section}
\makeatletter 
\begin{appendices}\label{appendix}
\numberwithin{equation}{section}
\section{Supplementary Results}
In this section, we collect the technical results needed to prove the results of Section~\ref{Sec:main results} of the manuscript.

    % {\color{red} \begin{enumerate}
    %     \item font and issues
    %     \item add page number something 
    % \end{enumerate}}
    % \begin{appxlem}\label{lem:SC}
    %  \begin{align*}
    %      \sum_{i=0}^{k} (k-i)!\Myperm[n]{k+i}\left(\Mycomb[k]{k-i}\right)^2    = \left( \Myperm[n]{k} \right)^2
    %  \end{align*}
    % \end{appxlem}
    
    % \begin{proof}
    %     Note that
    % \begin{align*}
    %      \sum_{i=0}^{k} (k-i)!\Myperm[n]{k+i} \left(\Mycomb[k]{k-i}\right)^2   &= \sum_{i=0}^k (k-i)!\frac{n!}{(n-k-i)!} \, \left(\frac{k!}{(k-i)!i!}\right)^2  \\
    %      &= n!k! \, \sum_{i=0}^k \frac{1}{(n-k-i)!i!} \, \frac{k!}{(k-i)! i!} \\
    %      &= \frac{n!k!}{(n-k)!}  \, \sum_{i=0}^k \frac{(n-k)!}{(n-k-i)!i!} \, \frac{k!}{(k-i)! i!} \\
    %      &= \frac{n!k!}{(n-k)!}  \, \sum_{i=0}^k \Mycomb[n-k]{i} \Mycomb[k]{k-i} \\ 
    %      &\stackrel{(*)}{=} \frac{n!k!}{(n-k)!}  \, \Mycomb[n]{k} \\
    %      &= (\Myperm[n]{k})^2,
    %  \end{align*}
    %  where we used Vandermonde's identity 
    %     $
    %      \Mycomb[m+n]{r}  =   \sum_{i=0}^r \Mycomb[m]{i} \Mycomb[n]{r-i}
    %     $
    %     in $(*)$.
    %   \end{proof}
    
    \begin{appxlem}\label{lem:split_to_exp_bound}
        Let $X$ be a random variable. Let $n \in \mathbb{N}$ and suppose there exists a constant $d>0$ that does not depend on $\tau$ and $n$ such that $\forall \tau >0$, with probability at least $1-he^{-\tau}$, $\abs{X} \leq d \left( \frac{1+\tau}{n} \right)^{c}$. Then there exists a non negative constant $g$ such that $\mathbb{E} \abs{X} \leq \frac{g}{n^{c}}$.
    \end{appxlem}
    \begin{proof}
        Let $\epsilon = d \left( \frac{1+\tau}{n} \right)^{c} \implies \tau = n \left(\frac{\epsilon}{d}\right)^{\frac{1}{c}} - 1$. Then
        \begin{align*}
            \mathbb{E}[\abs{X}] = \int_{0}^{\infty} P[\abs{X} > \epsilon] \, d \epsilon &\leq \int_{0}^{\infty} h \exp \left[ 1 - n \left(\frac{\epsilon}{d}\right)^{\frac{1}{c}}  \right] \, d \epsilon  \leq \frac{echd}{n^{c}} \int_{0}^{\infty} t^{c-1} \exp{(-t)}\, dt,
        \end{align*}
        yielding the result.
    \end{proof}
In rest of the section, $\mathcal{H}$ refers to a separable Hilbert space and $X_1,\ldots,X_n$ are independent $\mathcal{H}$-valued random elements defined on a measurable space. 
% Hoeffding proved that Hoeffing type inequality  on U statistics assuming bounded random variables \cite{hoeffding1994probability} but adapting it for Hilbert spaces requires following result.  
  \begin{appxthm}[\citealp{yurinsky2006sums}, Theorem 3.3.1(b)] For  any $t>0$,
        \begin{align*}
            \mathbb{E} \left[ \exp{\norm{ t \sum_{i=1}^n X_i}_{\mathcal{H}}} \right] &\leq  \exp{ \left[  \mathbb{E} \norm{t \sum_{i=1}^n X_i}_{\mathcal{H}} \right]   } \exp{ \left[ \sum_{i=1}^n \mathbb{E} e^{t \norm{X_i}_{\mathcal{H}} } - 1 - t \mathbb{E}\norm{X_i}_{\mathcal{H}} \right] }.
        \end{align*}
        \label{Thm:Yurinsky}
  \end{appxthm}
  
  \begin{appxlem}
  \label{Th:Bernlemma}
   Suppose $r: \mathcal{X}^{k} \rightarrow \mathcal{H}$ satisfies $\mathbb{E} \left[ r(X_{1},\dots,X_{k}) \right] = 0$  and 
   \begin{align}
        \mathbb{E} \norm{r(X_{1},\dots,X_{k})}^{p}_{\mathcal{H}} &\leq \frac{p!}{2} \theta^2 M^{p-2},\,\,\, \forall p \geq 2.\label{Eq:condition}
    \end{align}
    Then for any $ 0< a < \frac{\floor{n/k}}{M}$, 
    
    \begin{align*}
        \mathbb{E}\left( \exp \left[a  \norm{ \textup{U}^{n}_{k}\left(r(X_1,\dots,X_k)\right)}_{\mathcal{H}} \right] \right) &\leq \exp{ \left[a (\floor{n/k})^{-1/2} \theta \right] } \exp{\left[ \frac{a^2 \theta^2 \floor{n/k}^{-1} }{2-2a\floor{n/k}^{-1}M}  \right]}.
    \end{align*}
  \end{appxlem}
  
  \begin{proof}
         By defining
        \begin{align}
            V(X_{1},\dots,X_{n}) \triangleq \frac{1}{\floor{n/k}} \sum_{i=1}^{\floor{n/k}} r\left(X_{1+(i-1)k}, \dots, X_{ik}\right),\label{Eq:v}
        \end{align}
it is easy to verify that
        \begin{align*}
            \text{U}^{n}_{k}\left(r(X_1,\dots,X_k)\right) &= \frac{1}{n!} \sum_{\sigma \in S(n)} V\left(X_{\sigma(1)} , \dots X_{\sigma(n)}\right).
        \end{align*}
        Therefore, 
        \begin{align*}
            &\mathbb{E} \left(\exp \left[ a\norm{\text{U}^{n}_{k}\left( r(X_1,\dots,X_k)\right)}_{\mathcal{H}} \right] \right) 
            = \mathbb{E} \left( \exp \norm{ \frac{a}{n!} \sum_{\sigma \in S(n)} V(X_{\sigma(1)} , \dots X_{\sigma(n)}) }_{\mathcal{H}} \right) \\
            &\leq \mathbb{E} \left(\exp \left[\frac{a}{n!} \sum_{\sigma \in S(n)} \norm{V(X_{\sigma(1)} , \dots, X_{\sigma(n)})}_{\mathcal{H}} \right] \right) \\
            &\stackrel{(\dagger)}{\leq} \mathbb{E} \left( \frac{1}{n!} \sum_{\sigma \in S(n)}  \exp \left( a \norm{V(X_{\sigma(1)} , \dots, X_{\sigma(n)})}_{\mathcal{H}} \right) \right) \\
%              \end{align*}
%  \begin{align*}
             &= \frac{1}{n!} \sum_{\sigma \in S(n)} \mathbb{E} \left(\exp \left[a \norm{V(X_{\sigma(1)} , \dots, X_{\sigma(n)})}_{\mathcal{H}} \right] \right) 
             =\mathbb{E} \left(\exp \left[a \norm{V(X_{\sigma(1)} , \dots, X_{\sigma(n)})}_{\mathcal{H}} \right] \right),\nonumber
             \end{align*}
             where $(\dagger)$ follows from an application of Jensen's inequality. Using \eqref{Eq:v}, we have
             \begin{align}
             &\mathbb{E} \left(\exp \left[ a\norm{\text{U}^{n}_{k}\left( r(X_1,\dots,X_k)\right)}_{\mathcal{H}} \right] \right) 
            = \mathbb{E} \left( \exp \norm{ \frac{a}{\floor{n/k}} \sum_{i=1}^{\floor{n/k}} r(X_{1+(i-1)k}, \dots, X_{ik})}_{\mathcal{H}}   \right)  \nonumber\\
        &\stackrel{(*)}{\leq} \exp{ \left[ \sum_{i=1}^{\floor{n/k}} \mathbb{E} e^{(a/\floor{n/k}) \norm{r(X_{1+(i-1)k}, \dots, X_{ik})}_{\mathcal{H}} } - 1 -\frac{a}{\floor{n/k}} \mathbb{E} \norm{ r(X_{1+(i-1)k}, \dots, X_{ik})}_{\mathcal{H}} \right] } \nonumber\\
        &\qquad \times \exp{ \left[  \mathbb{E} \norm{ \frac{a}{\floor{n/k}} \sum_{i=1}^{\floor{n/k}} r(X_{1+(i-1)k}, \dots, X_{ik})}_{\mathcal{H}} \right]   }\nonumber\\
%         \end{align}
%         \begin{align}
        &\stackrel{(\dagger)}{\leq}  \exp{ \left[ \frac{a}{\floor{n/k}} \sqrt{\sum_{i=1}^{\floor{n/k}} \mathbb{E} \norm{r(X_{1+(i-1)k}, \dots, X_{ik})}_{\mathcal{H}}^2 } \right] }\nonumber\\
        &\qquad\times\exp{ \left[  \sum_{i=1}^{\floor{n/k}} \sum_{p=2}^{\infty} \frac{(a/\floor{n/k})^{p}\mathbb{E}  \norm{r(X_{1+(i-1)k}, \dots, X_{ik})}_{\mathcal{H}}^{p}  }{p!}   \right] },\label{Eq:tt} \end{align}
        where we used Theorem~\ref{Thm:Yurinsky} in $(*)$ and Jensen's inequality in $(\dagger)$. By using \eqref{Eq:condition} in \eqref{Eq:tt}, we have
        \begin{align*}
        \mathbb{E} \left(\exp \left[ a\norm{\text{U}^{n}_{k}\left( r(X_1,\dots,X_k)\right)}_{\mathcal{H}} \right] \right) &\leq  \exp{ \left[ \frac{a\theta}{\sqrt{\floor{n/k}}}  \right] } \exp{ \left[ \frac{\left(\frac{a}{\floor{n/k}}\right)^2 \theta^2 \floor{n/k}}{2}   \sum_{p=0}^{\infty}  \left( \frac{aM}{\floor{n/k}}\right)^p   \right]}  
        \end{align*}
        \begin{align*}
        &\leq \exp{ \left[ \frac{a\theta}{\sqrt{\floor{n/k}}} \right] } \exp{ \left[ \frac{ (a\theta)^2/\floor{n/k}}{2- 2\left( \frac{aM}{\floor{n/k}}\right)}     \right]},
        \end{align*}
        where the last inequality holds since $ 0< a < \frac{\floor{n/k}}{M}$.
  \end{proof}

The following result presents a Bernstein-type inequality for $\mathcal{H}$-valued $U$-statistics, whose proof is based on Lemma~\ref{Th:Bernlemma}.

\begin{appxthm}[Bernstein Inequality for $U$-statistics] \label{Th:BernUstats}
Suppose $r : \mathcal{X}^{k} \rightarrow \mathcal{H}$ is such that 
  \begin{align*}
        \mathbb{E} \left[ r(X_{1},\dots,X_{k}) \right]  = 0\quad\text{and}\quad \mathbb{E} \norm{r(X_{1},\dots,X_{k})}^{p}_{\mathcal{H}} &\leq \frac{p!}{2} \theta^2 M^{p-2},\quad \forall p \geq 2.
    \end{align*}
    Then for all $\tau > 0$,
    \begin{align*}
        \textup{Pr} \left\{   \norm{ \emph{U}^{n}_k\left(r(X_{1}, X_{2},\dots,X_{k}) \right)}_{\mathcal{H}} \geq   4 \theta \sqrt{k} \left(\frac{1+\tau}{n} \right)^{\frac{1}{2}} + 4  M k \left( \frac{1+\tau}{n} \right)    \right\}  & \leq e^{-\tau}.
     \end{align*}
\end{appxthm}
  \begin{proof}
  Note that
      \begin{align}
          &\text{Pr} \left\{   \norm{ \text{U}^{n}_k\left(r(X_{1}, X_{2},\dots,X_{k}) \right)}_{\mathcal{H}} \geq u \right\}\nonumber \\ 
          &=  \text{Pr} \left\{ \exp \left[t \norm{ \text{U}_{k}^n\left(r(X_{1}, X_{2},\dots,X_{k}) \right)}_{\mathcal{H}}\right] \geq \exp {(t u)} \right\}\nonumber \\
          &\leq  \exp(-t u)\mathbb{E} \left( \exp  \left[t \norm{ \text{U}_{k}^n\left(r(X_{1}, X_{2},\dots,X_{k}) \right)}_{\mathcal{H}}\right] \right)  \quad\quad (\because\,\text{Markov's inequality})\nonumber\\
%           \end{align}
%  \begin{align}
          &\stackrel{(*)}{\leq}  \exp{ \left(t (\floor{n/k})^{-1/2} \theta \right) } \exp\left(\frac{t^2\theta^2 (\floor{n/k})^{-1} }{2- 2Mt (\floor{n/k})^{-1}} -t u \right) \quad\quad \left(\forall\, \frac{t}{\floor{n/k}}M < 1\right) \label{Eq:t1}\\
          &= \exp{\left( - \frac{u^2/2-u\theta(\floor{n/k})^{-1/2}}{\theta^2 \floor{n/k}^{-1} + M \floor{n/k}^{-1} u} \right)},\nonumber
    \end{align}
    where the last equality is obtained by choosing $$t = \frac{u}{\theta^2 \floor{n/k}^{-1} + M \floor{n/k}^{-1} u}$$ in \eqref{Eq:t1} and $(*)$ follows from Lemma~\ref{Th:Bernlemma}. Note that this choice of $t$ clearly satisfies $\frac{t}{\floor{n/k}}M < 1$.  The result follows by choosing
    $$\tau =  \frac{u^2/2-u\theta (\floor{n/k})^{-1/2} }{\theta^2 \floor{n/k}^{-1} + M \floor{n/k}^{-1} u}$$ 
    and solving for $u$. Indeed, it is the case as solving $$ u^2 - 2\tau M \floor{n/k}^{-1} u - 2\tau\theta^2 \floor{n/k}^{-1} - 2u \theta(\floor{n/k})^{-1/2}   = 0$$ yields
        \begin{align*}
        u &=  \frac{2\tau M \floor{n/k}^{-1} + 2\theta (\floor{n/k})^{-1/2} + \sqrt{\left(2\tau M \floor{n/k}^{-1}+ 2\theta (\floor{n/k})^{-1/2} \right)^2+ 8\tau\theta^2 \floor{n/k}^{-1}}  }{2} \\
          &\leq  \frac{4 \tau M \floor{n/k}^{-1} +  4\theta (\floor{n/k})^{-1/2}+ 2 \sqrt{2 \tau \theta^2 \floor{n/k}^{-1} }}{2} \\
            %&\leq  \frac{4 (1+\tau) M \floor{n/k}^{-1}  + 4(1+\tau)  \theta (\floor{n/k})^{-1/2} +  2 \sqrt{2 (1+\tau) \theta^2 \ceil{n/k}^{-1}) }}{2} \\
          &\leq  4 \theta \sqrt{k} \left(\frac{1+\tau}{n} \right)^{\frac{1}{2}} + 4  M k \left( \frac{1+\tau}{n} \right),
    \end{align*}
    where we used $\floor{n/k}^{-1}\le 2(n/k)^{-1}$ for $n\geq k$ in the last line and the result follows.
 \end{proof}
  For degenerate $U$-statistics with  \textit{bounded, real valued} kernels, \cite{arcones1993limit} established a better convergence rate than in Theorem~\ref{Th:BernUstats}. The following result extends their result for \emph{unbounded, $\mathcal{H}$-valued} ($U$-statistic) kernels. 
  \begin{appxthm}\label{Th:BernUstatsDeg}
  Let $X_1, X_2, \dots, X_{n}\stackrel{i.i.d.}{\sim}P$ and $\sigma^2 = \mathbb{E}\norm{r(X_{1}, \dots X_{k})}^2_{\mathcal{H}}<\infty$, where $r:\mathcal{X}^{k} \rightarrow \mathcal{H}$ is a $P$-measurable function. Suppose $r$ is $P$-complete degenerate and there exists positive constants $M$ and $\theta$ such that
    \begin{align*}
        \mathbb{E} \Big| \norm{r(X_1,\dots,X_k)}^2_\mathcal{H} - \mathbb{E} \norm{r(X_1,\dots,X_k)}^2_\mathcal{H} \Big|^p &\leq \frac{p!}{2} \theta^2 M^{p-2}, \,\,\,\, \forall p \geq 2. 
    \end{align*}
        % \begin{enumerate}
    %     \item  $r$ is completely P-degenerate and 
    %     \item there exists positive constants $M$ and $\theta$ s.t
    % \begin{align*}
    %     \mathbb{E} \Big| \norm{r(X_1,\dots,X_k)}^2 - \mathbb{E} \norm{r(X_1,\dots,X_k)}^2 \Big|^p &\leq \frac{p!}{2} \theta^2 M^{p-2} \,\,\,\, \forall p \geq 2
    % \end{align*}
    % \end{enumerate}
Then there exists positive constants $a_1$, $b_1$ and $b_2$  such that
            \begin{align*}
            \textup{Pr} \left\{   \norm{ \textup{U}^{n}_k\left(r(X_{1}, X_{2},\dots,X_{k}) \right)}_{\mathcal{H}} \geq qk^k \left( \frac{\tau}{nb_1} \right)^{k/2} + Mk^k  \left( \frac{\tau}{nb_2} \right)^{(k+1)/2}     \right\}  & \leq a_1 \exp{(-\tau)},
        \end{align*}
where $q^2 := (\theta +  \sigma^2 +  \theta^2 M^{-1} )$. 
    \end{appxthm}
The proof of Theorem~\ref{Th:BernUstatsDeg} relies on the following results (Theorems~\ref{thm:de2012decoupling1}--\ref{Th:contractivity}) which are quoted from \cite{de2012decoupling}.   

%   Before proving we will be using following results from \cite{de2012decoupling}
% \textcolor{red}{add page number for Theorem B.5}
   \begin{appxthm}[\citealp{de2012decoupling}, p.168]
   There exists a real valued  non-decreasing convex function $\Psi$ and a constant $a_{\alpha}>0$ that depends only $\alpha$ such that for any $\alpha>0$,
      \begin{align} \label{Eq:convex-exp}
     a_{\alpha} \Psi(\abs{x}) \leq  \exp(\abs{x}^{\alpha}) &\leq  \Psi(\abs{x}).
   \end{align}
   \label{thm:de2012decoupling1}
   \vspace{-4mm}
   \end{appxthm}
    \begin{appxthm}[\citealp{de2012decoupling}, Theorem 3.5.3, Remark 3.5.4]
   Let $\Psi$ be a non-decreasing convex function on $[0, \infty]$. Suppose  $\epsilon_{i_1},\ldots, \epsilon_{i_k}$ are independent Rademacher random variables that are independent of $(X_i)^n_{i=1}$, and $a$ is a constant that depends only $k$. Then
     \begin{align} \label{Th:randomization}
         \mathbb{E} \Psi \left( \norm{ \sum_{i_{1}< \dots< i_{k}}^n r(X_{1}, \dots, X_{k}) }_{\mathcal{H}}  \right) &\leq  \mathbb{E} \Psi \left(  a\norm{ \sum_{i_{1}< \dots< i_{k}}^n \epsilon_{i_1}\cdots \epsilon_{i_k} r(X_{1}, \dots, X_{k}) }_{\mathcal{H}} \right).
     \end{align}
     \label{thm:de2012decoupling2}
     \vspace{-4mm}
   \end{appxthm} 
   \begin{appxthm}[\citealp{de2012decoupling}, Corollary 3.2.7]\label{Th:contractivity}
   For every $d \in \mathbb{N}$ and $0 < \alpha < \frac{2}{d}$, there exists positive constants, $c_1,c_2$, such that for all $t >0$,
   \begin{align}
       &\tilde{\mathbb{E}} \exp{ \left( t \norm{ \sum_{i_{1}< \dots< i_{k}}^n \epsilon_{i_1}\cdots \epsilon_{i_k} r(X_{1}, \dots, X_{k}) }_{\mathcal{H}}^{\alpha} \right)} \nonumber\\
       &\leq c_1 \exp{\left(c_2 t^{\frac{2}{2-\alpha d}}  \left( \tilde{\mathbb{E}}  \norm{ \sum_{i_{1}< \dots< i_{k}}^n \epsilon_{i_1}\cdots \epsilon_{i_k} r(X_{1}, \dots, X_{k}) }_{\mathcal{H}}^{2}  \right)^{\frac{\alpha}{2-\alpha d}} \right)},\label{Eq:de2012decoupling3}
   \end{align}
   where $\tilde{\mathbb{E}}$ denotes the expectation w.r.t.~Rademacher variables conditioned on $(X_i)^n_{i=1}$.
   %\label{thm:de2012decoupling3}
   \end{appxthm}
   \begin{proof}
            Consider 
            \begingroup
            \allowdisplaybreaks
            \begin{align*}
                &\mathbb{E} \exp \left(t \norm{\frac{1}{n^{k/2}} \sum_{i_{1} < \dots < i_{k}}^n r(X_{i_1}, X_{i_2},\dots,X_{i_k})}_{\mathcal{H}}^{\frac{2}{k+1}} \right) \\
                &\stackrel{\eqref{Eq:convex-exp}}{\leq}   \mathbb{E}_{X_1,\dots,X_n} \Psi \left(t^{\frac{k+1}{2}} \norm{\frac{1}{n^{k/2}} \sum_{i_{1} < \dots < i_{k}}^n r(X_{i_1}, X_{i_2},\dots,X_{i_k})}_{\mathcal{H}} \right)\\ 
                %\qquad see.~(\ref{Eq:convex-exp})\\
                &\stackrel{\eqref{Th:randomization}}{\leq}  \mathbb{E}_{X_1,\dots,X_n, \epsilon_{1},\dots,\epsilon_{n}} \Psi \left(c_1 t^{\frac{k+1}{2}} \norm{\frac{1}{n^{k/2}} \sum_{i_{1} < \dots < i_{k}}^n \epsilon_{i_1}\epsilon_{i_2}\cdots\epsilon_{i_k} r(X_{i_1}, X_{i_2},\dots,X_{i_k})}_{\mathcal{H}} \right)\\% \qquad see.~ (\ref{Th:randomization}) \\
             &\stackrel{\eqref{Eq:convex-exp}}{\leq} \frac{1}{c_{2}} \mathbb{E}_{X_1,\dots,X_n, \epsilon_1,\dots,\epsilon_n} \exp \left(c_1^{\frac{2}{k+1}}t \norm{\frac{1}{n^{k/2}} \sum_{i_{1} < \dots < i_{k}}^n \epsilon_{i_1}\epsilon_{i_2}\cdots\epsilon_{i_k} r(X_{i_1}, X_{i_2},\dots,X_{i_k})}^{\frac{2}{k+1}}_{\mathcal{H}} \right)\\ %\qquad see.~(\ref{Eq:convex-exp}) \\
               &\stackrel{\eqref{Eq:de2012decoupling3}}{\leq} a_1 \mathbb{E}_{X_1,\dots,X_{n}}  \exp{\left(a_2 t^{{k+1}} \mathbb{E}_{\epsilon_{1},\dots,\epsilon_{n}} \norm{\frac{1}{n^{k/2}} \sum_{i_{1}<\dots<i_{k}}^n \epsilon_{i_1}\epsilon_{i_2}\cdots\epsilon_{i_k} r(X_{i_1}, X_{i_2},\dots,X_{i_k})}^{2}_{\mathcal{H}} \right)}\\
               &=:\spadesuit,   
               \end{align*}
               where in the last inequality, we employed \eqref{Eq:de2012decoupling3} with $d=k$. Here, $c_1$, $c_2$, $a_1$ and $a_2$ are constants that depend only on $k$. By noting that
               \begin{align*}
               &\mathbb{E}_{\epsilon_{1},\dots,\epsilon_{n}} \norm{\frac{1}{n^{k/2}} \sum_{i_{1}<\dots<i_{k}}^n \epsilon_{i_1}\epsilon_{i_2}\cdots\epsilon_{i_k} r(X_{i_1}, X_{i_2},\dots,X_{i_k})}^{2}_{\mathcal{H}}
               =\frac{1}{n^{k}}   \sum_{i_{1}<\dots<i_{k}}^n  \norm{r(X_{i_1}, X_{i_2},\dots,X_{i_k})}^{2}_{\mathcal{H}},
               \end{align*}
               we obtain
                % & \qquad \text{ Th.~(\ref{Th:contractivity}) } \alpha=\frac{2}{k+1}, d=k \text{ and } a_{1} \triangleq \frac{c_1}{a_{\alpha}},  a_{2} \triangleq c_{2}a^{k+1}  \\
                \begin{align*}
                &\spadesuit= a_1 \mathbb{E}_{X_{1},\dots,X_{n}}  \exp{\left( a_2 t^{{k+1}} \frac{1}{n^{k}}   \sum_{i_{1}<\dots<i_{k}}^n  \norm{r(X_{i_1}, X_{i_2},\dots,X_{i_k})}^{2}_{\mathcal{H}} \right)}\\% \qquad \text{properties of independent rademachers} \\
                &= a_1\mathbb{E}  \exp\left(a_2 t^{{k+1}} \frac{1}{n^{k}}   \sum_{i_{1}<\dots<i_{k}}^n  \norm{r(X_{i_1}, X_{i_2},\dots,X_{i_k})}^{2}_{\mathcal{H}}- a_2t^{k+1} \mathbb{E} \norm{r(X_{1}, X_{2},\dots,X_{k})}^2_{\mathcal{H}}\right. \\
                & \qquad\qquad\qquad\qquad\qquad+ \Bigg. a_2t^{k+1} \mathbb{E} \norm{r(X_{1}, X_{2},\dots,X_{k})}^2_{\mathcal{H}} \Bigg) \\
                % &= c_3\mathbb{E} \exp\left(c_4 t^{{k+1}} \frac{1}{n^{k}}   \sum_{i_{1}<\dots<i_{k}}^n  \norm{r(X_{i_1}, X_{i_2},\dots,X_{i_k})}^{2}_{\mathcal{H}}\right.\\
                % &\qquad\qquad\qquad\Bigg- c_4t^{k+1} \frac{1}{\Mycomb[n]{k}}\sum_{i_{1}<\dots<i_{k}}^n \mathbb{E} \norm{r(X_{i_1}, X_{i_2},\dots,X_{i_k})}^2_{\mathcal{H}} \Bigg.)
                %  \exp{\left(a_2t^{k+1} \mathbb{E}_{X_{1},\dots,X_{n}} \norm{r(X_{1}, X_{2},\dots,X_{k})}^2_{\mathcal{H}}\right)} \\
                &\stackrel{(*)}{\leq} a_1 \mathbb{E} \exp{\left(  \frac{a_2 t^{{k+1}}}{\Mycomb[n]{k}}   \sum_{i_{1}<\dots<i_{k}}^n  \Bigg[\norm{r(X_{i_1}, X_{i_2},\dots,X_{i_k})}^{2}_{\mathcal{H}} - \mathbb{E} \norm{r(X_{i_1}, X_{i_2},\dots,X_{i_k})}^2_{\mathcal{H}} \Bigg]  \right)}\\&\qquad\qquad\times\exp{\left(a_2t^{k+1}\sigma^2\right)}   \\
                &= a_1 \mathbb{E} \exp{\bigg(  a_2 t^{k+1} \,    \text{U}^{k}_{n}\Big( \norm{r(X_{1}, X_{2},\dots,X_{k})}^{2}_{\mathcal{H}} - \mathbb{E} \norm{r(X_{1}, X_{2},\dots,X_{k})}^2_{\mathcal{H}} \Big)  \bigg)}\exp{\left(a_2t^{k+1}\sigma^2\right)} \\
                &=:\clubsuit,
                \end{align*}
                where we used $\frac{1}{n^k} \leq \frac{1}{\Mycomb[n]{k}}$ in $(*)$. It follows from Lemma~\ref{Th:Bernlemma} that for all $0< a_{2}t^{k+1} < \frac{\floor{n/k}}{M} $
                \begin{align*}
                    \clubsuit &\leq  a_1 \exp{\left( a_2 t^{k+1} (\floor{n/k})^{-1/2} \theta \right) } \exp{\left( \frac{a_2^2t^{2(k+1)} \theta^2 \floor{n/k}^{-1} }{2-2a_2t^{k+1}\floor{n/k}^{-1}M}   \right)} \exp{\left(a_2t^{k+1}\sigma^2\right)} \\ 
                    &= a_1 \exp{\left( a_2 t^{k+1} \theta \right) } \exp{\left( \frac{a_2^2t^{2(k+1)} \theta^2  }{2\floor{n/k}-2a_2t^{k+1}M}   \right)} \exp{\left(a_2t^{k+1}\sigma^2\right)} \\
                    &\stackrel{(**)}{\le} a_1 \exp{\left( a_2 t^{k+1} \theta \right) } \exp{\left( a_2 t^{k+1} \theta^2 M^{-1}  \right)}  \exp{\left(a_2t^{k+1}\sigma^2\right)} 
                 = a_1 \exp{\left(a_2 t^{k+1} q^2 \right)},
                \end{align*}
where $q^2 = (\theta +  \sigma^2 +  \theta^2 M^{-1} )$ and $(**)$ holds if $0< t^{k+1} <  \frac{2\floor{n/k}}{3a_2M} $. Therefore, for all $0< t^{k+1} <  \frac{2\floor{n/k}}{3a_2M} $,
                 \begin{align*}
                & \text{Pr} \left\{ \frac{1}{n^{k/2}} \norm{ \sum_{i_{1} < \dots < i_{k}}^n r(X_{i_1}, X_{i_2},\dots,X_{i_k})}_{\mathcal{H}} > u \right\} \\
                &= \text{Pr} \left \{ \exp{\left(t \norm{\frac{1}{n^{k/2}} \sum_{i_{1} < \dots < i_{k}}^n r(X_{i_1}, X_{i_2},\dots,X_{i_k})}_{\mathcal{H}}^{\frac{2}{k+1}}\right)} > \exp{\left(tu^{\frac{2}{k+1}}\right)} \right\}  \\
                &\leq \exp{\left(-tu^{\frac{2}{k+1}}\right)}\,\mathbb{E} \exp \left(t \norm{\frac{1}{n^{k/2}} \sum_{i_1<\dots i_k}r(X_{i_1}, X_{i_2},\dots,X_{i_k})}_{\mathcal{H}}^{\frac{2}{k+1}} \right) \\
                 &\leq  a_1 \exp{\left(  - tu^{\frac{2}{k+1}}+ a_2 t^{k+1} q^2   \right)}=:a_1  F(t),
            \end{align*}
            where $F(t)=\exp\left(- tu^{\frac{2}{k+1}}+ a_2 t^{k+1} q^2\right)$. We now consider two cases. Let $K$ be any constant such that $K \le \frac{(a_{2}k)^{\frac{1}{k}}}{3M}$.
            
            \textit{Case (i):} Suppose $\dfrac{u^{\frac{2}{k}}}{n q^{\frac{2(k+1)}{k}}} \le K$. Note that  $t^{*} = \argmin_{t} F(t)=\left( \frac{u^{\frac{2}{k+1}}}{a_2(k+1) q^2 } \right)^{\frac{1}{k}}$. It is easy to verify that $t^{*}$ is permissible since

            \begin{align*}
                (t^{*})^{k+1} &= \left( \frac{u^{\frac{2}{k+1}}}{a_2(k+1) q^2 } \right)^{\frac{k+1}{k}} 
                = \frac{u^{\frac{2}{k}}}{a_2^{\frac{k+1}{k}} (k+1)^{\frac{k+1}{k}} q^{\frac{2(k+1)}{k}}} 
                \leq  \frac{u^{\frac{2}{k}}}{a_2^{\frac{k+1}{k}} k^{\frac{k+1}{k}} q^{\frac{2(k+1)}{k}}} \\
                 &=  \frac{n u^{\frac{2}{k}}}{n  a_2^{\frac{k+1}{k}} k^{\frac{k+1}{k}} q^{\frac{2(k+1)}{k}}} 
                 =  \frac{(n/k)}{a_2^{\frac{k+1}{k}} k^{\frac{1}{k}}} \left[ \frac{u^{\frac{2}{k}}}{nq^{\frac{2(k+1)}{k}}} \right]   
                 \le \frac{(n/k)}{a_2^{\frac{k+1}{k}} k^{\frac{1}{k}}}  \frac{(a_{2}k)^{\frac{1}{k}}}{3M} 
                 < \frac{2 \floor{n/k}}{3a_{2} M}.
            \end{align*}
            Therefore, 
            \begin{align*}
                F(t^*)& \leq  \exp{\left( -  \left( \frac{u^{\frac{2}{k+1}}}{a_2(k+1) q^2 } \right)^{\frac{1}{k}} u^{\frac{2}{k+1}} + a_2 \left( \frac{u^{\frac{2}{k+1}}}{a_2(k+1) q^2 } \right)^{\frac{k+1}{k}}q^2  \right)} \\ 
                & \leq  \exp{\left( - \left(\frac{u^2}{q^2}\right)^{\frac{1}{k}} \frac{1}{(a_2(k+1))^{\frac{1}{k}}}  +  \left(\frac{u^2}{q^2}\right)^{\frac{1}{k}}\frac{1}{a_2^{\frac{1}{k}} (k+1)^{\frac{k+1}{k}}}    \right)}  
                \leq  \exp{\left( - b_1 \left(\frac{u^2}{q^2}\right)^{\frac{1}{k}}\right)},
            \end{align*}
where $b_1:=\frac{k}{(k+1)^{\frac{k+1}{k}}a^{\frac{1}{k}}_2}$.

            \textit{Case (ii):}  Suppose $\dfrac{u^{\frac{2}{k}}}{n q^{\frac{2(k+1)}{k}}} > K$. Let $L$ be any constant such that $L < \frac{KM^2}{(ka_{2})^{ \frac{k+1}{k}}}$. Define $t^{**} = \left( \frac{Ln}{M^2}\right)^{\frac{1}{k+1}}$. $t^{**}$ is indeed permissible since  
\begin{align*}
                (t^{**})^{k+1} =  \frac{Ln}{M^2}  < \frac{KM^2}{(ka_{2})^{\frac{k+1}{k}}} \frac{n}{M^2} \le  \frac{(a_{2}k)^{\frac{1}{k}}}{3M} \frac{n}{(ka_{2})^{\frac{k+1}{k}}}  = \frac{n/k}{3 a_2 M} < \frac{2 \floor{n/k}}{3a_{2} M}.
            \end{align*}
Since $q^{2} < \frac{u^{\frac{2}{k+1}}}{(nK)^{\frac{k}{k+1}}}$, we have
            \begin{align*}
                F(t^{**})
                &=  \exp{\left(  - L^{\frac{1}{k+1}} \left( \frac{nu^2}{M^2} \right)^{\frac{1}{k+1}} + a_2  \frac{Ln}{M^2} q^2   \right)} \\
                 &<  \exp{\left(  - L^{\frac{1}{k+1}} \left( \frac{nu^2}{M^2} \right)^{\frac{1}{k+1}} +   \left( \frac{nu^2}{M^2} \right)^{\frac{1}{k+1}} \frac{a_2 L M^{\frac{-2k}{k+1}} }{K^{\frac{k}{k+1}}}   \right)} 
                 \leq  \exp{\left( - b_{2}\left( \frac{nu^2}{M^2} \right)^{\frac{1}{k+1}} \right)},
            \end{align*}
            where $b_2$ is a positive constant since $ \frac{a_2 L M^{\frac{-2k}{k+1}} }{K^{\frac{k}{k+1}}} <  L^{\frac{1}{k+1}} \Leftrightarrow   L <\frac{KM^2}{a_{2}^{\frac{k+1}{k}}}\Leftarrow L<\frac{KM^2}{(ka_{2})^{ \frac{k+1}{k}}}$.
Therefore,

          \begin{align*}
               &\text{Pr} \left\{ \frac{1}{n^{k/2}} \norm{ \sum_{i_1 < \dots < i_k}^n r(X_{i_1}, X_{i_2}, \dots, X_{i_k})}_{\mathcal{H}} > u \right\}\\
               &\leq a_1 \min \left( {\exp{ \left( - b_1 \left(  \frac{u}{q} \right)^{\frac{2}{k}} \right) } }, \exp{ \left( -b_2 \left( \frac{nu^2}{M^2} \right)^{\frac{1}{k+1}} \right)} \right) \\
               &= a_1\exp\left( \min \left(  - b_1 \left(  \frac{u}{q} \right)^{\frac{2}{k}}, -b_2 \left( \frac{nu^2}{M^2} \right)^{\frac{1}{k+1}}  \right) \right) \\
               &= a_1\exp\left( - \max \left(  b_1 \left(  \frac{u}{q} \right)^{\frac{2}{k}}, b_2 \left( \frac{nu^2}{M^2} \right)^{\frac{1}{k+1}}  \right) \right),
           \end{align*} 
                   % Now we will simplify this into more usable form.
        % \begin{enumerate}
        %     \item $b_1 \left(  \frac{u}{q} \right)^{\frac{2}{k}} = \tau \implies \frac{u}{q} = \left( \frac{\tau}{b_1} \right)^{k/2} \implies u = q \left( \frac{\tau}{b_1} \right)^{k/2} $ \\
        %     \item $b_2 \left( \frac{nu^2}{(M)^2} \right)^{\frac{1}{k+1}} = \tau \implies \frac{nu^2}{(M)^2}  = \left( \frac{\tau}{b_2} \right)^{k+1}  \implies u = \frac{M}{\sqrt{n}}  \left( \frac{\tau}{b_2} \right)^{(k+1)/2}  $
        % \end{enumerate}
implying, 
        \begin{align*}
        &a_1\exp{(-\tau)}\\
            &\ge \text{Pr} \left\{ \frac{1}{n^{k/2}} \norm{ \sum_{i_{1}< \dots < i_{k}}^n r(X_{i_1}, X_{i_2}, \dots, X_{i_k})}_{\mathcal{H}} \geq \max\left( q \left( \frac{\tau}{b_1} \right)^{k/2} , \frac{M}{\sqrt{n}}  \left( \frac{\tau}{b_2} \right)^{(k+1)/2}   \right) \right\}\\ 
            &=  \text{Pr} \left\{ \frac{\Mycomb[n]{k}}{n^{k/2}}  \norm{ \text{U}^{n}_{k}
             \Big(r(X_{1}, X_{2},\dots,X_{k}) \Big)}_{\mathcal{H}} \geq \max\left( q \left( \frac{\tau}{b_1} \right)^{k/2} , \frac{M}{\sqrt{n}}  \left( \frac{\tau}{b_2} \right)^{(k+1)/2}   \right) \right\}  \\
             & = \text{Pr} \left\{   \norm{ \text{U}^{n}_{k}
             \Big(r(X_{1}, X_{2},\dots,X_{k}) \Big)}_{\mathcal{H}} \geq \frac{n^{k/2}}{\Mycomb[n]{k}}  \max\left( q \left( \frac{\tau}{b_1} \right)^{k/2} , \frac{M}{\sqrt{n}}  \left( \frac{\tau}{b_2} \right)^{(k+1)/2}   \right) \right\}  \\
             & \stackrel{(\dagger)}{\geq} \text{Pr} \left\{   \norm{ \text{U}^{n}_k\Big(r(X_{1}, X_{2},\dots,X_{k}) \Big)}_{\mathcal{H}} \geq \frac{k^k}{n^{k/2}} \max\left( q \left( \frac{\tau}{b_1} \right)^{k/2} , \frac{M}{\sqrt{n}}  \left( \frac{\tau}{b_2} \right)^{(k+1)/2}   \right) \right\}   \\
              & = \text{Pr} \left\{   \norm{ \text{U}^{n}_k\Big(r(X_{1}, X_{2},\dots,X_{k}) \Big)}_{\mathcal{H}} \geq \max\left(k^k q \left( \frac{\tau}{nb_1} \right)^{k/2} , k^k M  \left( \frac{\tau}{nb_2} \right)^{(k+1)/2}   \right) \right\} \\
              & \stackrel{(\ddagger)}{\geq} \text{Pr} \left\{   \norm{ \text{U}^{n}_k\Big(r(X_{1}, X_{2},\dots,X_{k}) \Big)}_{\mathcal{H}} \geq qk^k \left( \frac{\tau}{nb_1} \right)^{k/2} + Mk^k  \left( \frac{\tau}{nb_2} \right)^{(k+1)/2}   \right\},
        \end{align*}
        where we used $\left(\frac{n}{k}\right)^k\le \Mycomb[n]{k} $ in $(\dagger)$ and $ \max\{r,s \} \leq r + s$ in $(\ddagger)$.
        % Hence we have that 
        % \begin{align*}
        %     \text{Pr} \left\{   \norm{ \text{U}_{n}^k\Big(r(X_{1}, X_{2},\dots,X_{k}) \Big)}_{\mathcal{H}} \geq qk^k \left( \frac{\tau}{nb_1} \right)^{k/2} + Mk^k  \left( \frac{\tau}{n.b_2} \right)^{(k+1)/2}     \right\}  & \leq a_1 \exp{(-\tau)}
        % \end{align*}
        \endgroup
        \end{proof}
        
    \section{Shrinkage Estimator of Covariance Matrix}
 In this section, we specialize and simplify the calculations of Example~\ref{Ex:cov} for $K(x,y)=\langle x,y\rangle_2,\,x,y\in\mathbb{R}^d$, yielding a shrinkage estimator of the covariance matrix on $\mathbb{R}^d$. First, we present a lemma which is useful to obtain the simplified expressions of $\hat{\Delta}_{\textup{general}}$, $\hat{\Delta}_{\textup{degen}}$ and $\Vert\hat{C}\Vert^2_\mathcal{H}$ in Proposition~\ref{appxpro:cov}.

\begin{appxlem}\label{appxlem:cov}
Let $\bar{X} = \frac{1}{n} \sum_{i=1}^n X_i $, $\tilde{X}_i = X_{i} - \bar{X} $,  $\hat{\Sigma} = \frac{1}{n} \sum_{i=1}^n \tilde{X_i} \tilde{X_i}^\top $ be the sample mean, centered random variables and empirical covariance matrix respectively, based on independent random variables $(X_i)^n_{i=1}$. Then, the following hold:
\begin{flalign}
 &(i) \sum_{i,j=1}^n \inp{X_i - X_j}{X_i - X_j}_{2}^2 =   2n\sum^n_{i=1} \Vert \tilde{X_i} \Vert_{2}^4 +  4n^2 \textup{Tr} [\hat{\Sigma}^2] + 2 n^2 \textup{Tr}^2 [\hat{\Sigma}]; && \label{Eq:inp-i,j} \\
    &(ii)\sum_{i,j,l=1}^n \inp{X_i-X_j}{X_j- X_{l}}_{2}^2 = n^2 \sum_{i=1}^n \Vert \tilde{X_i} \Vert_{2}^{4} + 3n^3 \textup{Tr}[\hat{\Sigma}^2]; && \label{Eq:inp-i,j,l} \\
    &(iii)\sum_{i,j,l,m=1}^n \inp{X_i - X_j}{ X_{l} - X_{m}}^2_{2} =4n^4 \textup{Tr}[\hat{\Sigma}^2]. \label{Eq:inp-i,j,k,l}
\end{flalign}
\end{appxlem}
\begin{proof}
% \vspace{1mm}
$(i)$
\begin{align*}
&\sum_{i,j=1}^n \inp{X_i - X_j}{X_i - X_j}_{2}^2 
= \sum_{i,j} \Vert X_i - X_j\Vert^4 = \sum_{i,j} \Vert\tilde{X_i}- \tilde{X_j}\Vert^4 \\
&= \sum_{i,j} \left(\Vert \tilde{X}_i\Vert_{2}^2 + \Vert \tilde{X_j}\Vert_{2}^2 - 2 \langle \tilde{X_i},\tilde{X_j}\rangle_2   \right)^2 \\
% \end{align*}
% \begin{align*}
&= \sum_{i,j} \Vert \tilde{X_i} \Vert_{2}^4 + \Vert \tilde{X_j} \Vert_{2}^4 + 4 \langle \tilde{X_i}, \tilde{X_j}\rangle^2_2 - 4\Vert \tilde{X}_i \Vert^2_{2}\langle \tilde{X_i},\tilde{X_j}\rangle_2 -4 \Vert \tilde{X}_{j} \Vert_{2}^2 \langle\tilde{X}_{j}, \tilde{X}_i\rangle_2 + 2 \Vert \tilde{X}_i \Vert_{2}^2  \Vert \tilde{X}_j \Vert_{2}^2\\
&\stackrel{(*)}{=}2n\sum_{i} \Vert \tilde{X_i} \Vert_{2}^4+4\sum_{i,j} \langle \tilde{X_i}, \tilde{X_j}\rangle^2_2+2 \left(\sum_i\Vert \tilde{X}_i \Vert_{2}^2 \right)^2 \stackrel{(**)}{=} 2n\sum_{i} \Vert \tilde{X_i} \Vert_{2}^4 +  4n^2 \text{Tr} [\hat{\Sigma}^2] + 2 n^2\text{Tr}^2 [\hat{\Sigma}],
\end{align*}
where we used $\sum_j \tilde{X}_j=0$ in $(*)$, $\sum_{i,j}\langle \tilde{X}_i,\tilde{X}_j\rangle^2_2=\sum_{i,j}\text{Tr}\left[\tilde{X}_i\tilde{X}^\top_i\tilde{X}_j\tilde{X}^\top_j\right]=n^2\text{Tr}[\hat{\Sigma}^2]$ and $\sum_i\Vert \tilde{X}_i\Vert^2_2=\sum_i\text{Tr}[\tilde{X}_i\tilde{X}^\top_i]=n\text{Tr}[\hat{\Sigma}]$ in $(**)$.\vspace{1mm}\\

\noindent $(ii)$
\begin{align*}
&\sum_{i,j,l=1}^n \inp{X_i-X_j}{X_i- X_{l}}_{2}^2 %= \sum_{i,j,l} \inp{\tilde{X_i}-\tilde{X_j}}{\tilde{X_i}- \tilde{X_{l}}}_{2}^2 \\
= \sum_{i} \text{Tr} \left[ \sum_{j} (X_{i} - X_{j})(X_{i} - X_{j})^\top \sum_{l} (X_i - X_{l})(X_{i} - X_{l})^\top \right] \\
&= \sum_{i} \text{Tr} \left[ \left\{\sum_{j} \left(\tilde{X}_i \tilde{X}_i^\top - \tilde{X}_i \tilde{X}_j^\top - \tilde{X}_j \tilde{X}_i^\top + \tilde{X}_j \tilde{X}_j^\top\right)\right\}^2  \right] = \sum_{i} \text{Tr} \left[ \left(n \tilde{X}_{i} \tilde{X}_{i}^\top +  n \hat{\Sigma}\right)^2 \right] \\
&= \sum_{i} n^2\Vert\tilde{X}_i\Vert_{2}^{4} + n^2 \text{Tr}[\hat{\Sigma}^2] +2 n^2 \sum_i\text{Tr} [\tilde{X}_{i} \tilde{X}_{i}^\top \hat{\Sigma}] = n^2 \sum_{i} \Vert\tilde{X}_i\Vert_{2}^{4} + 3n^3 \text{Tr}[\hat{\Sigma}^2].
\end{align*}
$(iii)$
\begin{align*}
&\sum^n_{i,j,l,m=1} \inp{X_i - X_j}{ X_{l} - X_{m}}^2_{2}= \sum_{i,j,l,m} \inp{\tilde{X_i} - \tilde{X_j}}{ \tilde{X_{l}} - \tilde{X_{m}}}^2_{2}\\
&= \text{Tr} \left[ \left\{\sum_{i,j} (\tilde{X_i} - \tilde{ X_j}) ( \tilde{X_i} - \tilde{X_{j}})^\top  \right\}^2\right] 
= \text{Tr} \left[\left\{ \sum_{i,j} \tilde{X_i}\tilde{X_i}^\top+  \tilde{ X_j} \tilde{X_j}^\top  \right\}^2\right]
= 4n^4 \text{Tr}[\hat{\Sigma}^2].
\end{align*}
Hence the proof.
\end{proof}

\begin{appxpro}\label{appxpro:cov}
For $K(x,y)=\langle x,y\rangle_2,\,x,y\in\mathbb{R}^d$ and $f^*=I_d$ (the $d\times d$ identity matrix) in Example~\ref{Ex:cov} of the manuscript, the following hold:
\begin{align*}
   \hat{\Delta}_{\textup{general}} &=  \frac{1}{(n-2)(n-3)} \sum_{i=1}^{n} \Vert \tilde{X}_i \Vert_{2}^4 - \frac{n(n+1)}{(n-1)^2(n-3)} \textup{Tr} [\hat{\Sigma}^2]
%    \end{align*}
% \begin{align*}
   %&\qquad\qquad
   - \frac{n}{(n-1)(n-2)(n-3)} \textup{Tr}^2[\hat{\Sigma}],\\
   \hat{\Delta}_{\textup{degen}} &=\frac{n (n^2-3n+4)}{ 2\cdot\Mycomb[n]{2}\cdot \Myperm[n]{4}}  \sum_{i=1}^{n} \Vert \tilde{X_i} \Vert_{2}^4 -  \frac{2n^2(n-2)}{\Mycomb[n]{2}\cdot \Myperm[n]{4}} \textup{Tr}[\hat{\Sigma}^2] +  \frac{n^2 (n^2-5n+4)}{2\cdot \Mycomb[n]{2} \cdot\Myperm[n]{4} }  \textup{Tr}^2 [\hat{\Sigma}],\,\,\,\text{and}\\
%    \end{align*}
% %   and
% \begin{align*}
   \Vert \hat{C}-I\Vert^2_F&=\frac{n^2}{(n-1)^2} \textup{Tr}[\hat{\Sigma}^2] -\frac{2n}{n-1} \textup{Tr} [\hat{\Sigma}] + d,
\end{align*}
where $(\tilde{X}_i)^n_{i=1}$ and $\hat{\Sigma}$ are defined in Lemma~\ref{appxlem:cov}, and $\hat{C} = \frac{1}{\Mycomb[n]{2}} \sum_{i < j} \frac{(X_i - X_j)(X_i - X_{j})^\top}{2}$, with $\Vert\cdot\Vert_F$ being the Frobenius norm.
\end{appxpro}
\begin{proof} We have from Example~\ref{Ex:cov} that
\begin{align}
\hat{\Delta}_{\text{general}} &= \frac{2n-4}{4\cdot\Mycomb[n]{2}\cdot \Myperm[n]{3}} \underbrace{ \sum_{i \neq j \neq l} \inp{X_{i} - X_{j}}{X_{i}- X_{l}}_{2}^2}_{\circled{\tiny{1}}}  + \frac{1}{4\cdot\Mycomb[n]{2}\cdot \Myperm[n]{2}}  \underbrace{\sum_{i \neq j}\norm{X_{i} - X_{j}}_{2}^4}_{\circled{\tiny{2}}}  \nonumber\\
&\qquad - \frac{2n-3}{4\cdot\Mycomb[n]{2}\cdot \Myperm[n]{4} } \underbrace{\sum_{i \neq j \neq l \neq m} \inp{X_i - X_j}{X_{l}- X_{m}}^2_{2}}_{\circled{\tiny{3}}}.\label{Eq:1-3}
\end{align}
We now simplify $\circled{\tiny{1}}-\circled{\tiny{3}}$ as follows.
\begin{align*}
    \circled{\tiny{1}}& = \sum_{i \neq j \neq l} \inp{X_{i} - X_{j}}{X_{i}- X_{l}}_{2}^2  
= \left[\sum_{i} \left( \sum_{j}  - \sum_{j=i} \right) \left( \sum_{l}  - \sum_{l \in \{j,i\}}   \right) \right] \inp{X_{i} - X_{j}}{X_{i}- X_{l}}_{2}^2 \\
    &= \left[ \sum_{i,j,l} - \sum_{i,j, l \in \{j,i\}  } - \sum_{i,j=i, l}   + \sum_{i,j =i , l \in \{j,i\}}  \right] \inp{X_{i} - X_{j}}{X_{i}- X_{l}}_{2}^2 \\
    &= \sum_{i,j,l} \inp{X_{i} - X_{j}}{X_{i}- X_{l}}_{2}^2 - \sum_{i,j, l \in \{j,i\}  } \inp{X_{i} - X_{j}}{X_{i}- X_{l}}_{2}^2 \\
    & \qquad - \sum_{i,j=i, l}  \inp{X_{i} - X_{j}}{X_{i}- X_{l}}_{2}^2 +  \sum_{i,j =i , l \in \{j,i\}} \inp{X_{i} - X_{j}}{X_{i}- X_{l}}_{2}^2 \\
     &= \sum_{i,j,l} \inp{X_{i} - X_{j}}{X_{i}- X_{l}}_{2}^2 - \sum_{i,j, l=j  } \inp{X_{i} - X_{j}}{X_{i}- X_{l}}_{2}^2 
          \end{align*}
\begin{align*}
     &= \sum_{i,j,l} \inp{X_{i} - X_{j}}{X_{i}- X_{l}}_{2}^2 - \sum_{i,j } \inp{X_{i} - X_{j}}{X_{i}- X_{j}}_{2}^2 \\
     & \stackrel{\eqref{Eq:inp-i,j,l} , \eqref{Eq:inp-i,j}}{=}  n^2 \sum_{i=1}^n \Vert\tilde{X_i}\Vert_{2}^{4} + 3n^3 \text{Tr}[\hat{\Sigma}^2] -  \left(2n\sum_{i=1}^n \Vert \tilde{X_i} \Vert_{2}^4 +  4n^2 \text{Tr} [\hat{\Sigma}^2] + 2 n^2 \text{Tr}^2 [\hat{\Sigma}] \right) \\
     &= (n^2 - 2n ) \sum_{i=1}^n \Vert\tilde{X}_i\Vert_{2}^{4} + (3n^3 - 4n^2) \text{Tr}[\hat{\Sigma}^2]  - 2 n^2 \text{Tr}^2 [\hat{\Sigma}],\\
% \end{align*}
% \begin{align*}
    \circled{\tiny{2}} &= \sum_{i \neq j}\Vert X_{i} - X_{j}\Vert_{2}^4 = \sum_{i} \left[ \sum_{j} - \sum_{j = i} \right] \Vert X_{i} - X_{j}\Vert_{2}^4 = \sum_{i,j} \inp{X_{i} - X_{j}}{X_{i}- X_{j}}_{2}^2 \\ &\stackrel{\eqref{Eq:inp-i,j}}{=}   2n\sum^n_{i=1} \Vert \tilde{X_i} \Vert_{2}^4 +  4n^2 \textup{Tr} [\hat{\Sigma}^2] + 2 n^2 \textup{Tr}^2 [\hat{\Sigma}],
\end{align*}
and
\begin{align*}
    \circled{\tiny{3}} &=  \sum_{i \neq j \neq l \neq m} \inp{X_i - X_j}{X_{l}- X_{m}}^2_{2} \\
    &= \left[\sum_{i} \left( \sum_{j}  - \sum_{j=i} \right) \left( \sum_{l}  - \sum_{l \in \{j,i\}}   \right) \left( \sum_{m} - \sum_{m \in \{i,j,l\}} \right) \right]  \inp{X_i - X_j}{X_{l}- X_{m}}^2_{2} \\
    &= \left[ \sum_{i,j}  \left( \sum_{l}  - \sum_{l \in \{j,i\}}   \right) \left( \sum_{m} - \sum_{m \in \{i,j\}} \right)  \right] \inp{X_i - X_j}{X_{l}- X_{m}}^2_{2} \\
    &= \left[ \sum_{i,j,l,m} - \sum_{i,j,l,m \in \{i,j\}  }  - \sum_{i,j, l \in \{j,i\} , m } + \sum_{i,j, l \in \{j,i\}, m \in \{i,j\}} \right] \inp{X_i - X_j}{X_{l}- X_{m}}^2_{2} \\
    &= \sum_{i,j,l,m} \inp{X_i - X_j}{X_{l}- X_{m}}^2_{2} - 4 \sum_{i,j,l} \inp{X_i - X_j}{X_{l}- X_{j}}^2_{2}+ 2 \sum_{i,j} \inp{X_i - X_j}{X_{i}- X_{j}}^2_{2}  \\
    &\stackrel{\eqref{Eq:inp-i,j}-\eqref{Eq:inp-i,j,k,l}}{=} 4n^4 \textup{Tr}[\hat{\Sigma}^2] - 4 \left[ n^2 \sum_{i=1}^n \Vert \tilde{X_i} \Vert_{2}^{4} + 3n^3 \textup{Tr}[\hat{\Sigma}^2] \right] \\
    &\qquad\qquad\qquad
    + 2 \left[ 2n\sum^n_{i=1} \Vert \tilde{X_i} \Vert_{2}^4 +  4n^2 \textup{Tr} [\hat{\Sigma}^2] + 2 n^2 \textup{Tr}^2 [\hat{\Sigma}]\right] \\
    &= (4n -4 n^2) \sum_{i=1}^{n} \Vert \tilde{X_i} \Vert_{2}^4 + (4n^4-12n^3+ 8n^2) \textup{Tr}[\hat{\Sigma}^2] + 4n^2  \textup{Tr}^2 [\hat{\Sigma}].
\end{align*}
Combining $\circled{\tiny{1}}$, $\circled{\tiny{2}}$, and $\circled{\tiny{3}}$ in \eqref{Eq:1-3} yields\\
\begin{align*}
    \hat{\Delta}_{\text{general}} 
    &=   \frac{2n-4}{4\cdot\Mycomb[n]{2}\cdot \Myperm[n]{3}} \left[ (n^2 - 2n ) \sum_{i=1}^n \Vert\tilde{X}_i\Vert_{2}^{4} + (3n^3 - 4n^2) \text{Tr}[\hat{\Sigma}^2]  - 2 n^2 \text{Tr}^2 [\hat{\Sigma}] \right] \\
%     \end{align*}
% \begin{align*}
    & \qquad + \frac{1}{4\cdot\Mycomb[n]{2}\cdot \Myperm[n]{2}} \left[ 2n\sum^n_{i=1} \Vert \tilde{X_i} \Vert_{2}^4 +  4n^2 \textup{Tr} [\hat{\Sigma}^2] + 2 n^2 \textup{Tr}^2 [\hat{\Sigma}] \right] \\
    &\qquad -  \frac{2n-3}{4\cdot\Mycomb[n]{2}\cdot \Myperm[n]{4} } \left[ (4n -4 n^2) \sum_{i=1}^{n} \Vert \tilde{X_i} \Vert_{2}^4 + (4n^4-12n^3+ 8n^2) \textup{Tr}[\hat{\Sigma}^2] + 4n^2 \textup{Tr}^2 [\hat{\Sigma}] \right] 
        \end{align*}
\begin{align*}
    &= \frac{1}{4\cdot\Mycomb[n]{2}\cdot \Myperm[n]{2}} \left[ \left( 2(n^2-2n) + 2n - \frac{(2n-3)(4n-4n^2)}{(n-2)(n-3)} \right) \sum_{i=1}^{n} \Vert \tilde{X_i} \Vert_{2}^4 \right] \\
     &\qquad +\frac{1}{4\cdot\Mycomb[n]{2}\cdot \Myperm[n]{2}} \left[ \left( 2(3n^3 - 4n^2)+ 4n^2 - \frac{(4n^4-12n^3+ 8n^2)(2n-3)}{(n-2)(n-3)} \right) \textup{Tr} [\hat{\Sigma}^2]   \right] \\
     &\qquad +\frac{1}{4\cdot\Mycomb[n]{2}\cdot \Myperm[n]{2}} \left[ \left(-4n^2 + 2n^2- \frac{4n^2 (2n-3)}{(n-2)(n-3)} \right) \textup{Tr}^2[\hat{\Sigma}]   \right] \\
     &= \frac{1}{(n-2)(n-3)} \sum_{i=1}^{n} \Vert \tilde{X_i} \Vert_{2}^4 - \frac{n(n+1)}{(n-1)^2(n-3)} \textup{Tr} [\hat{\Sigma}^2]  - \frac{n}{(n-1)(n-2)(n-3)} \textup{Tr}^2 [\hat{\Sigma}].
\end{align*}
Doing similar analysis for $\hat{\Delta}_{\text{degen}}$ yields
\begin{align*}
    \hat{\Delta}_{\text{degen}} 
    &\stackrel{(*)}{=} \frac{1}{4\cdot\Mycomb[n]{2} \cdot\Myperm[n]{2} } \cdot \circled{\tiny{2}} - \frac{1}{4\cdot\Mycomb[n]{2}\cdot \Myperm[n]{4}} \cdot\circled{\tiny{3}} \\
    &= \frac{1}{4\cdot\Mycomb[n]{2}\cdot \Myperm[n]{2}} \left[  2n\sum^n_{i=1} \Vert \tilde{X_i} \Vert_{2}^4 +  4n^2 \textup{Tr} [\hat{\Sigma}^2] + 2 n^2 \textup{Tr}^2 [\hat{\Sigma}] \right] \\
    & \quad  - \frac{1}{4\cdot\Mycomb[n]{2}\cdot \Myperm[n]{4}} \left[  (4n -4 n^2) \sum_{i=1}^{n} \Vert \tilde{X_i} \Vert_{2}^4 + (4n^4-12n^3+ 8n^2) \textup{Tr}[\hat{\Sigma}^2] + 4n^2 \textup{Tr}^2 [\hat{\Sigma}] \right] \\
    &=  \frac{n (n^2-3n+4)}{2\cdot\Mycomb[n]{2}\cdot \Myperm[n]{4}}  \sum_{i=1}^{n} \Vert \tilde{X_i} \Vert_{2}^4 -  \frac{2n^2(n-2)}{\Mycomb[n]{2}\cdot \Myperm[n]{4}} \textup{Tr}[\hat{\Sigma}^2] +  \frac{n^2 (n^2-5n+4)}{2\cdot\Mycomb[n]{2} \cdot\Myperm[n]{4} }  \textup{Tr}^2 [\hat{\Sigma}],
\end{align*}
where $(*)$ is exactly the restatement of Example~\ref{Ex:cov-2} with $K(x,y)=\langle x,y\rangle_2,\,x,y\in\mathbb{R}^d$.\\

Note that, 
    \begin{align*}
        \hat{C}&= \frac{1}{\Mycomb[n]{2}} \sum_{i< j} \frac{(X_i - X_j)(X_i - X_{j})^\top}{2}  = \frac{1}{2\cdot\Myperm[n]{2}} \sum_{i, j} (X_i - X_j)(X_i - X_{j})^\top\\
%          \end{align*}
%          \begin{align*}
       &= \frac{1}{2\cdot\Myperm[n]{2}} \sum_{i,j} (\tilde{X}_{i} - \tilde{X_{j}})(\tilde{X_{i}} - \tilde{X_{j}})^\top\\
        &= \frac{1}{2\cdot\Myperm[n]{2}} \sum_{i , j} (\tilde{X}_{i} \tilde{X}_{i}^\top - \tilde{X}_j \tilde{X}_{i}^\top - \tilde{X}_i \tilde{X}_{j}^\top + \tilde{X}_j \tilde{X}_{j}^\top) = \frac{n}{n-1} \hat{\Sigma}
    \end{align*}
and
    \begin{align*}
        \Vert\hat{C} - I_d \Vert_{F}^2 & =\left\Vert \frac{n}{n-1} \hat{\Sigma}- I_d \right\Vert_{F}^2 = \frac{n^2}{(n-1)^2} \text{Tr}[\hat{\Sigma}^2] -\frac{2n}{n-1} \text{Tr} [\hat{\Sigma}] + d,
    \end{align*}
    thereby proving the result.
\end{proof}

\end{appendices}    
\end{document}